\documentclass[a4paper]{amsart}

\usepackage{pgf,tikz}
\usepackage{ifthen}
\usepackage{intcalc}
\usepackage[draft]{fixme} 
\usepackage[utf8]{inputenc}
\usepackage[T1]{fontenc}
\usepackage{float}
\usepackage{amsthm}
\usepackage{xspace}
\usepackage{amssymb,amsmath,stmaryrd}
\usepackage[english]{babel}
\usepackage[shortlabels]{enumitem}
\usepackage[all]{xy}
\usepackage{caption}
\usepackage{subcaption}
\usepackage{wrapfig}
\usepackage{xcolor}
\usepackage{xr-hyper}
\usepackage{hyperref}
 \theoremstyle{plain}
\newtheorem{theorem}{Theorem}[section]
\newtheorem{remark}[theorem]{Remark}
\newtheorem{lemma}[theorem]{Lemma}
\newtheorem{corollary}[theorem]{Corollary}
\newtheorem{proposition}[theorem]{Proposition}

\newtheorem{observation}[theorem]{Observation}
\newtheorem{fact}[theorem]{Fact}	
\theoremstyle{definition}
\newtheorem{assumption}[theorem]{Assumption}
\newtheorem{definition}[theorem]{Definition}
\newtheorem{notation}[theorem]{Notation}
\newtheorem{example}[theorem]{Example}
\numberwithin{equation}{section}
 \newcommand{\ca}{\mbox{$C\sp*$-}al\-ge\-bra\xspace}
\newcommand{\cas}{\mbox{$C\sp*$-}al\-ge\-bras\xspace}
\newcommand{\starhomo}{\mbox{$\sp*$-}ho\-mo\-morphism\xspace}
\newcommand{\starhomos}{\mbox{$\sp*$-}ho\-mo\-morphisms\xspace}
\newcommand{\stariso}{\mbox{$\sp*$-}iso\-morphism\xspace}
\newcommand{\starisos}{\mbox{$\sp*$-}iso\-morphisms\xspace}
\newcommand{\fct}[3]{\ensuremath{#1\colon #2\rightarrow #3}\xspace}
\newcommand{\fctw}[3]{$#1$ from $#2$ to $#3$\xspace}
\newcommand{\Z}{\ensuremath{\mathbb{Z}}\xspace}
\newcommand{\N}{\ensuremath{\mathbb{N}}\xspace}
\newcommand{\C}{\ensuremath{\mathbb{C}}\xspace}

\newcommand{\K}{\ensuremath{\mathbb{K}}\xspace}
\newcommand{\ie}{\emph{i.e.}\xspace}
\newcommand{\cf}{\emph{cf.}\xspace}
\newcommand{\eg}{\emph{e.g.}\xspace}
\newcommand{\id}{\ensuremath{\operatorname{id}}\xspace}

\newcommand{\cok}{\operatorname{cok}}

\newcommand{\FKRplus}{\ensuremath{\operatorname{FK}^+_\mathcal{R}}\xspace}
\newcommand{\FKR}{\ensuremath{\operatorname{FK}_\mathcal{R}}\xspace}

\newcommand{\A}{\ensuremath{\mathfrak{A}}\xspace}
\newcommand{\B}{\ensuremath{\mathfrak{B}}\xspace}
\newcommand{\Asf}{\mathsf{A}}
\newcommand{\Bsf}{\mathsf{B}}
\newcommand{\Csf}{\mathsf{C}}
\newcommand{\Esf}{\mathsf{E}}
\newcommand{\Prim}{\operatorname{Prim}}
\newcommand{\Prime}{\operatorname{Prime}}
\newcommand{\calP}{\ensuremath{\mathcal{P}}\xspace}
\newcommand{\GLZ}[1][n]{\ensuremath{\operatorname{GL}(#1,\Z)}\xspace}
\newcommand{\GL}{\ensuremath{\operatorname{GL}}\xspace}

\newcommand{\SL}{\ensuremath{\operatorname{SL}}\xspace}
\newcommand{\GLPZ}[1][\mathbf{n}]{\ensuremath{\operatorname{GL}_\calP(#1,\Z)}\xspace}
\newcommand{\GLP}{\ensuremath{\operatorname{GL}_\calP}\xspace}
\newcommand{\SLPZ}[1][\mathbf{n}]{\ensuremath{\operatorname{SL}_\calP(#1,\Z)}\xspace}
\newcommand{\SLP}{\ensuremath{\operatorname{SL}_\calP}\xspace}

\newcommand{\MZ}[1][\mathbf{m}\times\mathbf{n}]{\ensuremath{\mathfrak{M}(#1,\Z)}\xspace}
\newcommand{\MPZ}[1][\mathbf{m}\times\mathbf{n}]{\ensuremath{\mathfrak{M}_\calP(#1,\Z)}\xspace}
\newcommand{\MPZc}[1][\mathbf{m}\times\mathbf{n}]{\ensuremath{\mathfrak{M}^\circ_\calP(#1,\Z)}\xspace}
\newcommand{\MPZcc}[1][\mathbf{m}\times\mathbf{n}]{\ensuremath{\mathfrak{M}^{\circ\circ}_\calP(#1,\Z)}\xspace}
\newcommand{\MPZccc}[1][\mathbf{m}\times\mathbf{n}]{\ensuremath{\mathfrak{M}^{\circ\circ\circ}_\calP(#1,\Z)}\xspace}
\newcommand{\Mplus}[1][m\times n]{\ensuremath{\mathfrak{M}^+(#1,\Z)}\xspace}
\newcommand{\MPplusZ}[1][\mathbf{m}\times\mathbf{n}]{\ensuremath{\mathfrak{M}^+_\calP(#1,\Z)}\xspace}
\newcommand{\ftn}[3]{ #1 \colon #2 \rightarrow #3 }
\newcommand{\setof}[2]{\left\{ #1 \;\middle|\; #2 \right\}}
\newcommand{\GLPE}{\GLP-equivalent\xspace}
\newcommand{\SLPE}{\SLP-equivalent\xspace}
\newcommand{\GLPEe}{\GLP-equivalence\xspace}
\newcommand{\SLPEe}{\SLP-equivalence\xspace}

\newcommand{\Meq}{\ensuremath{\sim_{M\negthinspace E}}\xspace}
\newcommand{\MCeq}{\ensuremath{\sim_{C\negthinspace E}}\xspace}

\newcommand{\OO}{\mbox{\texttt{\textup{(O)}}}\xspace}
\newcommand{\II}{\mbox{\texttt{\textup{(I)}}}\xspace}
\newcommand{\RR}{\mbox{\texttt{\textup{(R)}}}\xspace}
\newcommand{\SSS}{\mbox{\texttt{\textup{(S)}}}\xspace}
\newcommand{\CC}{\mbox{\texttt{\textup{(C)}}}\xspace}

\newcommand{\CO}{\mbox{\texttt{\textup{(Col)}}}\xspace}
 
\newenvironment{smallpmatrix}{\left(\begin{smallmatrix}}{\end{smallmatrix}\right)}
\newenvironment{proofsk}{\noindent\emph{Sketch of proof. }}{\hfill$\square$}

\newcommand{\CAtemp}{{{0}}}

\newcommand{\frX}[2]{\FKRplus(#1;C^*(#2))}

\usetikzlibrary{arrows}
\usetikzlibrary{calc}
\newcommand{\AF}{\fcolorbox{black}{blue}{\textcolor{white}{-1}}}
\newcommand{\CA}{\fbox{0}}
\newcommand{\PI}{\fcolorbox{black}{red}{\textcolor{white}{1}}}
\newcommand{\nn}{{\underline{n}}}

\newcommand{\mm}{{\underline{m}}}

\newcommand{\coker}{\operatorname{cok}}

\newcommand{\efvs}[1]{{#1}^0_{\operatorname{cycl}}}
\newcommand{\Primt}{\Prime_\gamma}
\newcommand{\I}{\mathfrak I}	
\newcommand{\J}{\mathfrak J}

\newcommand{\oooI}[1]{\ifthenelse{\equal{#1}{gray}}{\AFPI}{\ifthenelse{\isodd{#1}}{\PI}{\CA}}}
\newcommand{\ooIo}[1]{\ifthenelse{\equal{#1}{gray}}{\AFPI}{\ifthenelse{\isodd{\intcalcDiv{#1}{2}}}{\PI}{\CA}}}
\newcommand{\oIoo}[1]{\ifthenelse{\equal{#1}{gray}}{\AFPI}{\ifthenelse{\isodd{\intcalcDiv{#1}{4}}}{\PI}{\CA}}}
\newcommand{\Iooo}[1]{\ifthenelse{\equal{#1}{gray}}{\AFPI}{\ifthenelse{\isodd{\intcalcDiv{#1}{8}}}{\PI}{\CA}}}

	\newcommand{\corona}[1]{\mathcal{Q}(#1)}
\newcommand{\Ab}{\boldsymbol{Ab}}
\newcommand{\oo}{\equiv_O}

\newcommand{\ii}[1]{\equiv_{I,#1}}
\newcommand{\twop}[1]{
\begin{tikzpicture}
\node at ( 0,0) {\ooIo{#1}};
\node at ( 1,0) {\oooI{#1}};
\draw [->]  (0.2,0) -- (0.8,0);
\end{tikzpicture}}

\newcommand{\threeplin}[1]{
\begin{tikzpicture}
\node at ( 0,0) {\oIoo{#1}};
\node at ( 1,0) {\ooIo{#1}};
\node at ( 2,0) {\oooI{#1}};
\draw [->]  (0.2,0) -- (0.8,0);
\draw [->]  (1.2,0) -- (1.8,0);
\end{tikzpicture}}

\newcommand{\threeplinmo}[1]{
\begin{tikzpicture}
\node at ( 0,0) {\oIoo{#1}};
\node at ( 1,0) {\ooIo{#1}};
\node at ( 2,0) {\AF};
\draw [->]  (0.2,0) -- (0.8,0);
\draw [->]  (1.2,0) -- (1.7,0);
\end{tikzpicture}}

\newcommand{\threepin}[1]{
\begin{tikzpicture}
\node at ( 0,0) {\oIoo{#1}};
\node at ( 1,0) {\ooIo{#1}};
\node at ( 2,0) {\oooI{#1}};
\draw [->]  (0.2,0) -- (0.8,0);
\draw [->]  (1.8,0) -- (1.2,0);
\end{tikzpicture}}

\newcommand{\threepout}[1]{
\begin{tikzpicture}
\node at ( 0,0) {\oIoo{#1}};
\node at ( 1,0) {\ooIo{#1}};
\node at ( 2,0) {\oooI{#1}};
\draw [->]  (0.8,0) -- (0.2,0);
\draw [->]  (1.2,0) -- (1.8,0);
\end{tikzpicture}}

\pgfmathsetmacro{\nodedistancethree}{0.5}
\pgfmathsetmacro{\nodedistancefour}{0.5}

\pgfmathsetmacro{\boundingboxscalethree}{2}
\pgfmathsetmacro{\boundingboxscalefour}{2}

\tikzset{  mynode/.style = {circle, fill, draw, minimum size = 2 mm, inner sep =
    1pt,thick},
    arrow/.style = {->,>=stealth',shorten >=1pt, shorten <=1pt,semithick },
    loop_3_1/.style = {in=110,out=70,loop,looseness = 15 },
    loop_3_2/.style = {in=230,out=190,loop,looseness = 15 },
    loop_3_3/.style = {in=350,out=310,loop,looseness = 15 },
    loop_4_1/.style = {loop above },
    loop_4_2/.style = {loop left },
    loop_4_3/.style = {loop below },
    loop_4_4/.style = {loop right }}

\newcommand{\fournodes}{\coordinate (O) at (0,0);
\node[mynode] (1) at (0,\nodedistancefour) {};
\node[mynode] (2) at ($(O)!1!90:(1)$) {};
\node[mynode] (3) at ($(O)!1!180:(1)$) {};
\node[mynode] (4) at ($(O)!1!270:(1)$) {};
\useasboundingbox ($\boundingboxscalefour*(1)$) -- ($\boundingboxscalefour*(2)$) -- ($\boundingboxscalefour*(3)$)-- ($\boundingboxscalefour*(4)$);}

\newcommand{\fed}[2]{\draw[arrow] (#1) to (#2);}
\newcommand{\flo}[1]{\draw[arrow,loop_4_#1] (#1) to (#1);}

\newcommand{\wastheta}{{\upsilon}}
\newcommand{\mytheta}[1][E]{{\mathcal Y}_{\Bsf_{#1}}}
\newcommand{\mytau}[1][E]{\mathcal{T}_{\Bsf_{#1}}}

\title[Graph $C^*$-algebras over finite graphs]{Geometric classification of graph $C^*$-algebras over finite graphs}
\date{\today}
\author{S\o{}ren Eilers}
\address{Department of Mathe\-matical Sciences, University of Copen\-hagen, Universi\-tets\-park\-en~5, DK-2100 Copen\-hagen, Den\-mark}
\email{eilers@math.ku.dk}
\author{Gunnar Restorff}
\address{Department of Science and Technology, University of the Faroe Islands, N\'{o}at\'{u}n~3, FO-100 T\'{o}rshavn, the Faroe Islands}
\email{gunnarr@setur.fo}
\author{Efren Ruiz}
\address{Department of Mathematics, University of Hawaii, Hilo, 200 W.~Kawili St., Hilo, Hawaii, 96720-4091 USA}
\email{ruize@hawaii.edu}
\author{Adam P.~W.~S\o{}rensen}
\address{Department of Mathematics, University of Oslo, PO BOX 1053 Blindern, N-0316 Oslo, Norway}
\email{apws@math.uio.no}

\begin{document}

\begin{abstract}
We address the classification problem for graph \cas of
finite graphs (finitely many edges and vertices), containing the class of Cuntz-Krieger algebras as a
prominent special case. Contrasting earlier work, we do not assume
that the graphs satisfy the standard condition (K), so that the graph
\cas may come with uncountably many ideals.

We find that in this generality, stable isomorphism of graph
\cas does not coincide with the geometric notion of Cuntz
move equivalence. However, adding a modest condition on the
graphs, the two notions are proved to be mutually equivalent and
equivalent to the \cas having isomorphic $K$-theories. This
proves in turn that under this condition, the graph
\cas are in fact classifiable by $K$-theory, providing in
particular complete classification when the \cas in question
are either of real rank zero or type I/postliminal. The key ingredient in obtaining these results is a characterization of Cuntz move equivalence using the adjacency matrices of the graphs.

Our results are applied to discuss the classification problem for the quantum lens spaces defined by Hong and Szyma\'nski, and to complete the classification of graph \cas associated to all simple graphs with four vertices or less.
\end{abstract}

\maketitle

\section{Introduction}

The classification problem for Cuntz-Krieger algebras has a long and prominent history. Indeed, R\o rdam's classification
(\cite{MR1340839}) of the simple such $C^*$-algebras by appealing to
fundamental results in symbolic dynamics paved the way for the
sweeping generalization by Kirchberg and Phillips (\cite{MR1796912} and \cite{MR1745197}) to all simple,
nuclear, separable and purely infinite $C^*$-algebras in the UCT class, and Restorff's
generalization (\cite{MR2270572}) to the general case of Cuntz-Krieger
algebras with finitely many ideals (equivalent to Cuntz'
Condition (II)) was a key inspiration for the recent surge in results
concerning nonsimple purely infinite $C^*$-algebras.

Until now, almost nothing has been known about the classification of
Cuntz-Krieger $C^*$-algebras having infinitely many ideals --- failing
Condition (II) --- even though the symbolic dynamical systems that
define them are often extremely simple. In this paper, we will
establish classification up to stable isomorphism between the Cuntz-Krieger algebras
defined from a large class of graphs including the pairs of graphs
given in (a) and (b) of Figure \ref{firstexx}, but must leave open the
question concerning some more complicated graphs such as the ones in (c).
 
 \begin{figure}
 \begin{center}
 \begin{tabular}{|cc|cc|ccc|}\hline
 \qquad&&&&&&\\
 $\xymatrix{\bullet\ar[d]\ar@(ur,dr)[]\ar@(u,r)[]\\\bullet\ar[d]\ar@(ur,dr)[]\\\bullet\ar@(ur,dr)[]}$\qquad\qquad&
 $\xymatrix{\bullet\ar@/_15pt/[dd]\ar@/_/[d]\ar@/^/[d]\ar@(ur,dr)[]\ar@(u,r)[]\\\bullet\ar[d]\ar@(ur,dr)[]\\\bullet\ar@(ur,dr)[]}$\qquad\qquad&$\xymatrix{\bullet\ar[d]\ar@(ur,dr)[]\\\bullet\ar@(u,r)[]\ar[d]\ar@(ur,dr)[]\\\bullet\ar@(ur,dr)[]}$\qquad\qquad\qquad&
 $\xymatrix{\bullet\ar[d]\ar@(ur,dr)[]\ar@/_15pt/[dd]\ar@/_25pt/[dd]\\\bullet\ar[d]\ar@(u,r)[]\ar@(ur,dr)[]\\\bullet\ar@(ur,dr)[]}$\qquad\qquad&
 $\xymatrix{\bullet\ar[d]\ar@(ur,dr)[]\\\bullet\ar[d]\ar@(ur,dr)[]\\\bullet\ar[d]\ar@(u,r)[]\ar@(ur,dr)[]\\\bullet\ar@(ur,dr)[]}$\qquad\qquad\qquad&
 $\xymatrix{\bullet\ar[d]\ar@(ur,dr)[]\\\bullet\ar[d]\ar@(ur,dr)[]\ar@/_15pt/[dd]\ar@/_25pt/[dd]\\\bullet\ar[d]\ar@(u,r)[]\ar@(ur,dr)[]\\\bullet\ar@(ur,dr)[]}$&\\
 &&&&&&
\\\hline
 \multicolumn{2}{|c|}{(a)}& \multicolumn{2}{|c|}{(b)}& \multicolumn{3}{|c|}{(c)}\\\hline
 \end{tabular}
 \end{center}
 \caption{Six graphs}\label{firstexx}
 \end{figure}
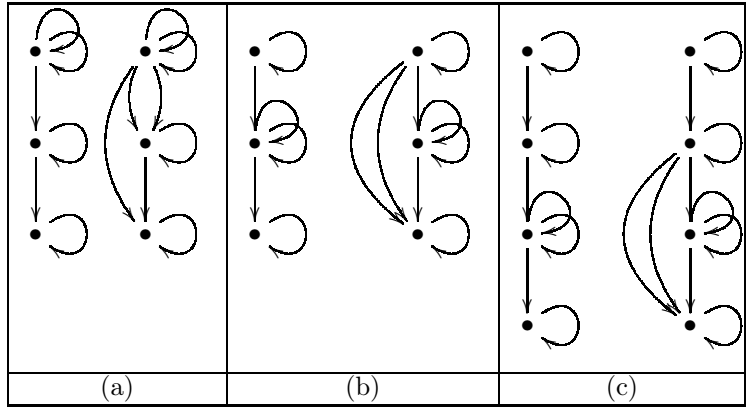
 
 We work in the more general (and more natural) setting of graph $C^*$-algebras over finite graphs, where Condition (II) is replaced by the standard Condition (K). 
Following \cite{MR3082546} and \cite{MR3047630} we emphasize the question of  \emph{geometric} classification, the aim being to generate the equivalence relation on graphs induced by stable isomorphism of the associated $C^*$-algebras as the coarsest equivalence relation containing the class of  basic moves on the graphs, resembling the role of Reidemeister moves on knots. These moves are closely related to those defining flow equivalence for shift spaces, apart from the 
 the so-called \emph{Cuntz splice} which plays a special role and also fails to preserve the canonical diagonal Abelian subalgebra of the graph algebras (cf.\ \cite{MR3276420}, \cite{arXiv:1410.2308v1}).
 
We will largely approach the problem  following the strategy from \cite{MR1340839} and \cite{MR2270572} to reduce the stable  classification problem for graph $C^*$-algebras to questions concerning flow equivalence of shifts of finite type. To do so requires three new tools as listed below.

First and foremost, we need to know that the Cuntz splice leaves the $C^*$-algebras in question invariant up to stable isomorphism also in this generality. This we proved in \cite{arXiv:1602.03709v2}. Second, we need to develop the theory of a \emph{gauge invariant prime ideal space} which in our case will serve as a substitute for the standard primitive ideal space. The fact that this space is finite is key to our largely combinatorial approach throughout the paper, and we will equip it with a \emph{temperature map} to help us align the graphs so that the various types of gauge simple subquotients are matched. Finally,  we introduce a procedure of \emph{plugging} and \emph{unplugging} sinks to pass between the cases allowing sinks and cases disallowing them, giving us the option to appeal to stronger general classification results in one case and a direct connection to symbolic dynamics in the other.

In the course of proving the above mentioned results, we extract and generalize from \cite{MR2270572} and \cite{MR1990568} some strong results concerning \GLPEe and \SLPEe, allowing us from the existence  of certain such equivalences to deduce conclusions about the existence  of move equivalences or Cuntz move equivalences between the graphs or about the existence of (stable) isomorphisms between the graph \cas, and \emph{vice versa}. This gives us some very concrete and hands-on tools to decide such questions.

In most cases, such as the one illustrated in Figure~\ref{firstexx}~(a), stable isomorphism
of the $C^*$-algebras associated to a pair of graphs allow for a
geometric realization by a finite number of moves, and we crystallize this out via the notion of  \emph{Condition (H)} which we introduce here. In
sporadic cases failing this condition, such as the ones illustrated in Figure~\ref{firstexx}~(b) and~(c), we will
establish that no finite sequence of the moves defining the concept of Cuntz move equivalence can connect the two graphs
in each pair, even though the $K$-theoretical invariants of the
associated $C^*$-algebras are the same. In the case of (b), we may in fact 
prove by appealing to \emph{ad hoc} classification results that the
$C^*$-algebras are stably isomorphic, proving that stable isomorphism of $C^*$-algebras is not always attainable via the moves hitherto studied.

Condition (H) generalizes Condition (K) and turns out to be met in a lot of other important special cases. When the graph $C^*$-algebras
defined are of type I/postliminary, our results may be refined further and lead to
the classification of a class of quantum lens spaces introduced and
studied by Hong and Szyma\'nski in \cite{MR2015735}. Moreover, specializing
to the graph $C^*$-algebras associated to simple graphs with four
vertices or less, we give a complete classification. These results have bearing on the Abrams-Tomforde conjecture (\cite{MR2775826}).

In forthcoming work
(\cite{Eilers-Restorff-Ruiz-Sorensen-2}) we will  introduce a final new move and prove, among
many other things, that indeed all Cuntz-Krieger algebras are
classified by their $K$-theory, because any isomorphism at the level
of $K$-theory may be realized using an  enlarged family of moves, all leaving the stabilized $C^*$-algebra invariant. The
present paper is self-contained and does not draw on the much more
complicated approach in \cite{Eilers-Restorff-Ruiz-Sorensen-2}. We will, however, develop basic results in the paper at hand in generality not needed here to anticipate applications in \cite{Eilers-Restorff-Ruiz-Sorensen-2}.

 The paper is organized as follows. In Section \ref{genprel} we outline key concepts for the paper, mainly stemming from the theory of graph \cas, and discuss the moves that constitute our fundamental notion of Cuntz move equivalence. In Section \ref{gipis} we develop the idea of the gauge invariant prime ideal space, which is completely essential for everything that follows, and we connect this to $K$-theory,  block matrices and partially ordered sets in Section \ref{sec:notation-for-proof}, introducing also the key notion of tempered ideal spaces.
 
 In Section \ref{CC} we then prove a complete characterization of Cuntz move equivalence for finite graphs, drawing heavily on ideas from \cite{MR2270572} augmented with a trick of \emph{plugging} sinks which we also develop there. Section \ref{Ccas} contains our geometric classification theorem for finite graphs with Condition (H), as well as examples showing the necessity of this condition, and in Section \ref{applications} we detail the applications listed above.

\section{General preliminaries}\label{genprel}

In this section, we introduce notation and fundamental concepts concerning graphs and their $C^*$-algebras.

\subsection{\texorpdfstring{$C^*$}{C*}-algebras over topological spaces}
Let $X$ be a topological space satisfying the $T_0$ separation condition and let $\mathbb{O}( X)$ be the set of open subsets of $X$, partially ordered by set inclusion $\subseteq$.  
A subset $Y$ of $X$ is called \emph{locally closed} if $Y = U \setminus V$ where $U, V \in \mathbb{O} ( X )$ and $V \subseteq U$.  
The set of all locally closed subsets of $X$ will be denoted by $\mathbb{LC}(X)$.  
The partially ordered set $( \mathbb{O} ( X ) , \subseteq )$ is a \emph{complete lattice}, that is, any subset $S$ of $\mathbb{O} (X)$ has both an infimum $\bigwedge S$ and a supremum $\bigvee S$, which are for any subset $S$ of $\mathbb{O} ( X )$ defined as 
\begin{equation*}
\bigwedge_{ U \in S } U = \left( \bigcap_{ U \in S } U \right)^{\circ} \quad \mathrm{and} \quad \bigvee_{ U \in S } U = \bigcup_{ U \in S } U.
\end{equation*}
Note that if $S$ is empty, these are $X$ and $\emptyset$, respectively.

For a \ca \A, let $\mathbb{I} ( \A )$ be the set of closed ideals of \A, partially ordered by $\subseteq$.  The partially ordered set $( \mathbb{I} ( \A ), \subseteq )$ is a complete lattice.  More precisely, for any subset $S$ of $\mathbb{I} ( \A )$, 
\begin{equation*}
\bigwedge_{ \mathfrak{I} \in S } \mathfrak{I} = \bigcap_{ \mathfrak{I} \in S } \mathfrak{I}  \quad \mathrm{and} \quad \bigvee_{ \mathfrak{I} \in S } \mathfrak{I} = \overline{ \sum_{ \mathfrak{I} \in S } \mathfrak{I} }.
\end{equation*}

\begin{definition}
Let \A be a \ca.  Let $\Prim ( \A )$ denote the \emph{primitive ideal space} of \A, equipped with the usual hull-kernel topology, also called the Jacobson topology.

Let $X$ be a topological space.  A \emph{\ca over $X$} is a pair $( \A
, \psi )$ consisting of a \ca \A and a continuous map $\ftn{ \psi
}{ \Prim ( \A ) }{ X }$.  \end{definition}

We identify $\mathbb{O} ( \Prim ( \A ) )$ and $\mathbb{I} ( \A )$ using the lattice isomorphism
\begin{equation*}
U \mapsto \bigcap_{ \mathfrak{p} \in \Prim ( \A ) \setminus U } \mathfrak{p}.
\end{equation*}
 Let $( \A , \psi )$ be a \ca over $X$.  Then we get a map $\ftn{ \psi^{*} }{ \mathbb{O} ( X ) }{ \mathbb{O} ( \Prim ( \A ) )  \cong \mathbb{I} ( \A ) }$ defined by
\begin{equation*}
U \mapsto \setof{ \mathfrak{p} \in \Prim ( \A ) }{ \psi ( \mathfrak{p} ) \in U }.
\end{equation*}
Using the isomorphism $\mathbb{O} ( \Prim ( \A ) ) \cong \mathbb{I} ( \A )$, we get a map from $\mathbb{O}(X)$ to $\mathbb{I}(\A )$ by
\begin{align*}
U \mapsto \bigcap \setof{ \mathfrak{p} \in \Prim ( \A ) }{ \psi ( \mathfrak{p} ) \notin U }.
\end{align*}
Denote this ideal by $\A(U)$.  For $Y = U \setminus V \in \mathbb{LC} ( X )$, set $\A(Y) = \A (U) / \A(V)$.   By \cite[Lemma~2.15]{MR2545613}, $\A ( Y)$ does not depend (up to a canonical \stariso) on $U$ and $V$.

We can equivalently define an $X$-algebra by giving a map from $\mathbb{O}(X)$ to $\mathbb{O}(\Prim(\A))$ that preserves finite infima and arbitrary suprema (so the empty set is mapped to the empty set, and $X$ is mapped to $\Prim(\A)$). 

\begin{example}
For any \ca \A, the pair $( \A , \id_{ \Prim ( \A ) } )$ is a  \ca over $\Prim ( \A )$.  For each $U \in \mathbb{O} ( \Prim ( \A ) )$, the ideal $\A ( U )$ equals $\bigcap_{ \mathfrak{p} \in \Prim ( \A ) \setminus U } \mathfrak{p}$. 
\end{example}

\begin{definition}
Let \A and \B be \cas over $X$.  
A \starhomo $\Phi\colon\A\rightarrow\B$ is \emph{$X$-equivariant} if $\Phi ( \A (U) ) \subseteq \B ( U )$ for all $U \in \mathbb{O}(X)$.  
Hence, for every $Y = U \setminus V$, $\Phi$ induces a \starhomo $\Phi_{Y}\colon\A ( Y ) \rightarrow\B (Y)$.  
Let $\mathcal{C}_X$ be the category whose objects are \cas over $X$ and whose morphisms are $X$-equivariant homomorphisms.  
\end{definition}

\subsection{Graphs and their matrices}

By a \emph{graph} we mean a directed graph. Formally:

\begin{definition}
A graph $E$ is a quadruple $E = (E^0 , E^1 , r, s)$ where $E^0$ and $E^1$ are sets, and $r$ and $s$ are maps from $E^1$ to $E^0$. 
The elements of $E^0$ are called \emph{vertices}, the elements of $E^1$ are called \emph{edges}, the map $r$ is called the \emph{range map}, and the map $s$ is called the \emph{source map}. 
\end{definition}

All graphs considered will be \emph{countable}, \ie, there are countably many vertices and edges. 
We call a graph \emph{finite}, if there are only finitely many vertices and edges.
We will freely identify graphs up to graph isomorphism.

\begin{definition}
A \emph{loop} is an edge with the same range and source. 

A \emph{path} $\mu$ in a graph is a finite sequence $\mu = e_1 e_2 \cdots e_n$ of edges satisfying 
$r(e_i)=s(e_{i+1})$, for all $i=1,2,\ldots, n-1$, and we say that the \emph{length} of $\mu$ is $n$. 
We extend the range and source maps to paths by letting $s(\mu) = s(e_1)$ and $r(\mu) = r(e_n)$. 
Vertices in $E$ are regarded as \emph{paths of length $0$} (also called empty paths). 

A \emph{cycle} is a nonempty path $\mu$ such that $s(\mu) = r(\mu)$.
We call a cycle $e_1e_2\cdots e_n$ a \emph{vertex-simple cycle} if $r(e_i)\neq r(e_j)$ for all $i\neq j$. A cycle $e_1e_2\cdots e_n$ is said to have an \emph{exit} if there exists an edge $f$ such that $s(f)=s(e_k)$ for some $k=1,2,\ldots,n$ with $e_k\neq f$. 
A \emph{return path} is a cycle $\mu = e_1 e_2 \cdots e_n$ such that $r(e_i) \neq r(\mu)$ for $i < n$.

For a loop, cycle or return path, we say that it is \emph{based} at the source vertex of its path. 
We also say that a vertex \emph{supports} a certain loop, cycle or return path if it is based at that vertex. 

Note that in \cite{MR1988256,MR1989499,MR2023453,MR1914564}, the authors use the term \emph{loop} where we use \emph{cycle}.
\end{definition}

\begin{definition}
A vertex $v\in E^0$ in $E$ is called \emph{regular} if $s^{-1}(v) := \setof{ e \in E^1 }{ s(e) = v }$ is finite and nonempty. We denote the set of regular vertices by
$E_{\mathrm{reg}}^0$. We call the remaining vertices \emph{singular}
and write $E_{\mathrm{sing}}^0=E^0\setminus E_{\mathrm{reg}}^0$.

A vertex $v\in E^0$ in $E$ is called a \emph{source} if $r^{-1}(v) := \setof{ e \in E^1 }{ r(e) = v }$ is the empty set. 
A vertex $v\in E^0$ in $E$ is called a \emph{sink} if $s^{-1}(v)=\emptyset$.
An \emph{isolated vertex} is both a sink and a source. 
\end{definition}

\begin{definition}
Let $E$ be a graph.  
We say that $E$ satisfies \emph{Condition~(K)} if for every vertex $v\in E^0$ in $E$, either there is no return path based at $v$ or there are at least two distinct return paths based at $v$. 
\end{definition}

\begin{notation}
If there exists a path from vertex $u$ to vertex $v$, then we write $u \geq v$ --- this is a preorder on the vertex set, \ie, it is reflexive and transitive, but need not be antisymmetric. 
\end{notation}

It is essential for our approach to graph \cas to be able to shift between a graph and its adjacency matrix. 
In what follows, we let \N denote the set of positive integers, while $\N_0$ denotes the set of nonnegative integers.

\begin{definition}
Let $E = (E^0 , E^1 , r, s)$ be a graph.
We define its \emph{adjacency matrix} $\Asf_E$ as an $E^0\times E^0$ matrix with the $(u,v)$'th entry being
$$\left\vert\setof{e\in E^1}{s(e)=u, r(e)=v}\right\vert.$$
As we only consider countable graphs, $\Asf_E$ will be a finite matrix or a countably infinite matrix, and it will have entries from $\N_0\sqcup\{\infty\}$.

Let $X$ be a set.
If $A$ is an $X \times X$ matrix with entries from $\N_0\sqcup\{\infty\}$, we let $\Esf_{A}$ be the graph with vertex set $X$ such that between two vertices $x,x' \in X$ we have $A(x,x')$ edges.
\end{definition}

It will be convenient for us to alter the adjacency matrix of a graph in two very specific ways, removing singular rows and subtracting the identity, so we introduce notation for this. 

\begin{notation}
Let $E$ be a graph and $\Asf_E$ its adjacency matrix. 
Denote by $\Asf_{E}^\bullet$ the matrix obtained from $\Asf_{E}$ by removing all rows corresponding to singular vertices of $E$.

Let $\Bsf_E$ denote the matrix $\Asf_{E} - I$, and let $\Bsf_{E}^\bullet$ be $\Bsf_E$ with the rows corresponding to singular vertices of $E$ removed. 
\end{notation}

\subsection{Graph \texorpdfstring{$C^*$}{C*}-algebras}
We follow the notation and definition for graph \cas in \cite{MR1670363}; this is not the convention used in Raeburn's monograph \cite{MR2135030}. 

\begin{definition} \label{def:graphca}
Let $E = (E^0,E^1,r,s)$ be a graph.
The \emph{graph \ca} $C^*(E)$ is defined as the universal \ca generated by
a set of mutually orthogonal projections $\setof{ p_v }{ v \in E^0 }$ and a set $\setof{ s_e }{ e \in E^1 }$ of partial isometries satisfying the relations
\begin{itemize}
	\item $s_e^* s_f = 0$ if $e,f \in E^1$ and $e \neq f$,
	\item $s_e^* s_e = p_{r(e)}$ for all $e \in E^1$,
	\item $s_e s_e^* \leq p_{s(e)}$ for all $e \in E^1$, and,
	\item $p_v = \sum_{e \in s^{-1}(v)} s_e s_e^*$ for all $v \in E^0$ with $0 < |s^{-1}(v)| < \infty$.
\end{itemize}
Whenever we have a set of mutually orthogonal projections $\setof{ p_v }{ v \in E^0 }$ and a set $\setof{ s_e }{ e \in E^1 }$ of partial isometries in a \ca satisfying the relations, then we call these elements a \emph{Cuntz-Krieger $E$-family}. 
\end{definition}

It is clear from the definition that an isomorphism between graphs induces a canonical isomorphism between the corresponding graph \cas.

\begin{definition}
Let $E=(E^0,E^1,r,s)$ be a graph. 
By universality there is a canonical gauge action $\gamma\colon\mathbb{T}\rightarrow\operatorname{Aut}(C^*(E))$ such that for any $z\in\mathbb{T}$, we have that $\gamma_z(p_v)=p_v$ for all $v\in E^0$ and $\gamma_z(s_e)=zs_e$ for all $e\in E^1$. 
We say that an ideal $\mathfrak{I}$ of $C^*(E)$ is gauge invariant, if $\gamma_z(\mathfrak{I})\subseteq \mathfrak{I}$ for all $z\in\mathbb{T}$, and we let $\mathbb{I}_\gamma(C^*(E))$ denote the subset of $\mathbb{I}(C^*(E))$ consisting of gauge invariant ideals. 
\end{definition}

It is clear that the lattice operations preserve the gauge invariance, so $\mathbb{I}_\gamma(C^*(E))$ is a sublattice. We collect some standard facts about graph \cas below.

\begin{remark}
Every graph \ca (of a countable graph) is separable, nuclear in the UCT class (\cite{MR1738948},\cite{MR2117597}). 
A graph \ca is unital if and only if the corresponding graph has finitely many vertices. 
A graph \ca is isomorphic to a Cuntz-Krieger algebra if and only if the corresponding graph is finite with no sinks, see~\cite[Theorem~3.12]{MR3391894}.
\end{remark}

 \subsection{Moves on graphs}\label{sec:moves}
In this section we describe the moves on graphs we will allow.

\begin{definition}[Move \SSS: Remove a regular source] 
Let $E = (E^0 , E^1 , r, s)$ be a graph, and let $w\in E^0$ be a source that is also a regular vertex. 
Let $E_S$ denote the graph $(E_S^0 , E_S^1 , r_S , s_S )$ defined by
$$E_S^0 := E^0 \setminus \{w\}\quad
E_S^1 := E^1 \setminus s^{-1} (w)\quad
r_S := r|_{E_S^1}\quad
s_S := s|_{E_S^1}.$$
We call $E_S$ the \emph{graph obtained by removing the source $w$ from $E$}, and say $E_S$ is formed by performing Move \SSS to $E$.
\end{definition}

\begin{definition}[Move \RR: Reduction at a regular vertex] 
Suppose that $E = (E^0 , E^1 , r, s)$ is a graph, and let $w \in E^0$ be a regular vertex with the property that $s(r^{-1} (w)) = \{x\}$, $s^{-1} (w) = \{f \}$, and $r(f ) \neq w$. 
Let $E_R$ denote the graph $(E_R^0, E_R^1, r_R , s_R )$ defined by
\begin{align*}
E_R^0&:= E^0 \setminus \{w\} \\
E_R^1&:= \left(E^1 \setminus (\{f \} \cup r^{-1}(w))\right) \cup \setof{e_f}{e \in E^1 \text{ and } r(e) = w} \\
r_R (e) &:= r(e)\text{ if }e \in E^1 \setminus (\{f \} \cup r^{-1}(w)) \quad\text{and}\quad r_R (e_f ) := r(f ) \\
s_R (e) &:= s(e)\text{ if }e \in E^1 \setminus (\{f \} \cup r^{-1}(w)) \quad\text{and}\quad s_R (e_f ) := s(e) = x.
\end{align*}
We call $E_R$ the \emph{graph obtained by reducing $E$ at $w$}, and say $E_R$ is a reduction
of $E$ or that $E_R$ is formed by performing Move \RR to $E$.
\end{definition}

\begin{definition}[Move \OO: Outsplit at a non-sink]
Let $E = (E^0 , E^1 , r, s)$ be a graph, and let $w \in E^0$ be a vertex that is not a sink. 
Partition $s^{-1} (w)$ as a disjoint union of a finite number of nonempty sets
$$s^{-1}(w) = \mathcal{E}_1\sqcup \mathcal{E}_2\sqcup \cdots \sqcup\mathcal{E}_n$$
with the property that at most one of the $\mathcal{E}_i$ is infinite. 
Let $E_O$ denote the graph $(E_O^0, E_O^1, r_O , s_O )$ defined by 
\begin{align*}
E_O^0&:= \setof{v^1}{v \in E^0\text{ and }v \neq w}\cup\{w^1, \ldots, w^n\} \\
E_O^1&:= \setof{e^1}{e \in E^1\text{ and }r(e) \neq w}\cup \setof{e^1, \ldots , e^n}{e \in E^1\text{ and }r(e) = w} \\
r_{E_O} (e^i ) &:= 
\begin{cases}
r(e)^1 & \text{if }e \in E^1\text{ and }r(e) \neq w\\
w^i & \text{if }e \in E^1\text{ and }r(e) = w
\end{cases} \\
s_{E_O} (e^i ) &:= 
\begin{cases}
s(e)^1 & \text{if }e \in E^1\text{ and }s(e) \neq w \\
s(e)^j & \text{if }e \in E^1\text{ and }s(e) = w\text{ with }e \in \mathcal{E}_j.
\end{cases}
\end{align*}
We call $E_O$ the \emph{graph obtained by outsplitting $E$ at $w$}, and say $E_O$ is formed by
performing Move \OO to $E$.
\end{definition}

\begin{definition}[Move \II: Insplit at a regular non-source]
Suppose that $E = (E^0 , E^1 , r, s)$ is a graph, and let $w \in E^0$ be a regular vertex that is not a source.
Partition $r^{-1} (w)$ as a disjoint union of a finite number of nonempty sets 
$$r^{-1} (w) = \mathcal{E}_1\sqcup \mathcal{E}_2\cdots\sqcup\mathcal{E}_n.$$
Let $E_I$ denote the graph $(E_I^0 , E_I^1 , r_I , s_I )$ defined by
\begin{align*}
E_I^0 &:= \setof{v^1}{v \in E^0\text{ and }v \neq w} \cup \{w^1,\ldots, w^n \} \\
E_I^1 &:= \setof{e^1}{e \in E^1\text{ and }s(e) \neq w} \cup \setof{e^1, \ldots, e^n}{e \in E^1\text{ and }s(e) = w} \\
r_{E_I} (e^i ) &:= 
\begin{cases}
r(e)^1 &\text{if }e \in E^1\text{ and }r(e) \neq w \\
r(e)^j &\text{if }e \in E^1\text{ and }r(e) = w\text{ with }e \in \mathcal{E}_j
\end{cases} \\
s_{E_I} (e^i ) &:= 
\begin{cases}
s(e)^1 &\text{if }e \in E^1\text{ and }s(e) \neq w \\
w^i &\text{if }e \in E^1\text{ and }s(e) = w.
\end{cases}
\end{align*}
We call $E_I$ the \emph{graph obtained by insplitting $E$ at $w$}, and say $E_I$ is formed by performing Move \II to $E$.
\end{definition}

\begin{definition}[Move \CC: Cuntz splicing] \label{def:cuntzsplice}
Let $E = (E^0 , E^1 , r , s )$ be a graph and let $v \in E^0$ be a regular vertex that supports at least two distinct return paths.
Let $E_C$ denote the graph $(E_C^0 , E_C^1 , r_C , s_C)$ defined by 
\begin{align*}
E_C^0 &:= E^0\sqcup\{u_1 , u_2 \} \\
E_C^1 &:= E^1\sqcup\{e_1 , e_2 , f_1 , f_2 , h_1 , h_2 \},
\end{align*}
where $r_{C}$ and $s_{C}$ extend $r$ and $s$, respectively, and satisfy
$$s_{C} (e_1 ) = v,\quad s_{C} (e_2 ) = u_1 ,\quad s_{C} (f_i ) = u_1 ,\quad s_{C} (h_i ) = u_2 ,$$
and
$$r_{C} (e_1 ) = u_1 ,\quad r_{C} (e_2 ) = v,\quad r_{C} (f_i ) = u_i ,\quad r_{C} (h_i ) = u_i . $$
We call $E_C$ the \emph{graph obtained by Cuntz splicing $E$ at $v$}, and say $E_C$ is formed by performing Move \CC to $E$. 
\end{definition}

\begin{definition}\label{def:graph-equivalences}
The equivalence relation generated by the moves \OO, \II, \RR, \SSS together with graph isomorphism is called \emph{move equivalence}, and denoted \Meq. 
The equivalence relation generated by the moves \OO, \II, \RR, \SSS, \CC together with graph isomorphism is called \emph{Cuntz move equivalence}, and denoted \MCeq. 
\end{definition}

The following two theorems were essentially proved in \cite{MR2054048}, see also  \cite[Propositions~3.1, 3.2 and 3.3 and Theorem~3.5]{MR3082546}.

\begin{theorem}[\cite{MR3082546}]\label{thm:moveimpliesstableisomorphism}
Let $E_1$ and $E_2$ be graphs such that $E_1\Meq E_2$. 
Then $C\sp*(E_1)\otimes \K\cong C\sp*(E_2)\otimes \K$. 
\end{theorem}

For the move \OO, we actually  obtain isomorphism rather than just stable isomorphism. 
\begin{proposition}\label{prop:moveOimpliesisomorphism}
Let $E_1$ and $E_2$ be graphs such that one is obtained from the other using Move \OO, then $C\sp*(E_1)\cong C\sp*(E_2)$. 
\end{proposition}

For the move \CC, it has recently been proved in \cite{arXiv:1602.03709v2} that it preserves the Morita equivalence class for arbitrary graphs. 

\begin{theorem}[{\cite[Theorem~4.8]{arXiv:1602.03709v2}}]
\label{thm:cuntz-splice-implies-stable-isomorphism}
Let $E$ be a graph and let $v$ be a vertex that supports two distinct return paths. Then $C\sp*(E)\otimes\K\cong C\sp*(E_C)\otimes\K$. 
\end{theorem}

We also extend the notation of equivalences to adjacency matrices. 

\begin{definition}
If $A,A'$ are square matrices with entries in $\N_0\sqcup\{\infty\}$, we define them to be \emph{move equivalent}, and write $A \Meq A'$ if $\Esf_A \Meq \Esf_{A'}$. 
We define \emph{Cuntz move equivalence} similarly. 
\end{definition}

\begin{remark}
The Cuntz move equivalence, $\MCeq$, is called \emph{move prime equivalence} in \cite{MR3082546}. Since the similarity of the two terms could create confusion, we have chosen to use the term \emph{Cuntz move equivalence} instead. 
\end{remark}

\subsection{Derived moves}

We now discuss --- following and  generalizing  \cite[Section~5]{MR3082546} --- ways of changing the graphs without   changing their move equivalence class. We will introduce  a \emph{collapse} move, and present criteria allowing us to conclude that two graphs are move equivalent when one arises from the other by a row or column addition of the $\Bsf$-matrices.
As we shall see, knowing move invariance of these derived moves dramatically  simplifies working with $\Meq$. 

\begin{definition}[Collapse a regular vertex that does not support a loop] \label{def:collapse}
Let $E = (E^0 , E^1 , r , s )$ be a graph and let $v$ be a regular vertex in $E$ that does not support a loop. 
Define a graph $E_{COL}$ by 
\begin{align*}
E_{COL}^0 &= E^0 \setminus \{v\}, \\
E_{COL}^1 &= E^1 \setminus (r^{-1}(v) \cup s^{-1}(v))
\sqcup \setof{[ef ]}{e \in r^{-1}(v)\text{ and }f \in s^{-1}(v)},
\end{align*}
the range and source maps extend those of $E$, and satisfy $r_{E_{COL}}([ef ]) = r (f )$
and $s_{E_{COL}} ([ef ]) = s (e)$.
\end{definition}

According to \cite[Theorem~5.2]{MR3082546} $E\Meq E_{COL}$ when $|E^0|<\infty$ --- in fact, the collapse move can be obtained using the moves \OO and \RR. We denote the move \CO.

Below, we will show how we can perform row and column additions on $\Bsf_E$ without changing the move equivalence class of the associated graphs, when $E$ is a graph with finitely many vertices. 

The setup we need is slightly different from what was considered in \cite[Section 7]{MR3082546} --- it was considered in \cite{arXiv:1602.03709v2}.
For the convenience of the reader, we collect the needed results from \cite{arXiv:1602.03709v2} in one proposition. Note that the definition of move equivalence in \cite{arXiv:1602.03709v2} is slightly different from the one above in order to be able to deal with graphs with infinitely many vertices --- but in the case of finitely many vertices they do in fact coincide. 

\begin{proposition}[\cite{arXiv:1602.03709v2}]
\label{prop:matrix-moves}
Let $E = (E^0, E^1, r,s)$ be a graph with finitely many vertices. 
Suppose $u,v \in E^0$ are distinct vertices with a path from $u$ to $v$.
Let $E_{u,v}$ be equal to the identity matrix except for on the $(u,v)$'th entry, where it is $1$. 
Then $\Bsf_E E_{u,v}$ is the matrix formed from $\Bsf_E$ by adding the $u$'th column into the $v$'th column, while $E_{u,v}\Bsf_E $ is the matrix formed from $\Bsf_E$ by adding the $v$'th row into the $u$'th row.
Then the following holds.
\begin{enumerate}[(i)]
\item Suppose $u$ supports a loop or suppose that there is an edge from $u$ to $v$ and $u$ emits at least two edges.  Then
\[
	\Asf_E \Meq \Bsf_E E_{u,v} + I.
\]\label{prop:matrix-moves:I}
\item Suppose $v$ is regular and either $v$ supports a loop or there is an edge from $u$ to $v$.  Then  
\[
	\Asf_E \Meq E_{u,v}\Bsf_E  + I.
\]\label{prop:matrix-moves:II}
\end{enumerate}
\end{proposition}
\begin{remark} \label{rmk:columnAdd}
As in \cite{arXiv:1602.03709v2}, we can use the above proposition
backwards to subtract columns or rows in $\Bsf_E$ as long as the
addition that undoes the subtraction is legal.
\end{remark}

Since legal row and column additions preserve $\Meq$, the resulting graph $C^*$-algebras will be stably isomorphic. The column addition in \ref{prop:matrix-moves:I} above rarely preserves the actual isomorphism class, but under  modest additional assumptions, the row addition in \ref{prop:matrix-moves:II} does.

\begin{proposition}\label{p:isoaddingrows}
If condition \ref{prop:matrix-moves:II} in Proposition \ref{prop:matrix-moves} is met with $v$ regular, supporting a loop {and} an edge from $u$ to $v$, then $C^*(E)\cong C^*(\Esf_A)$, where 
$A=E_{u,v}\Bsf_E  + I$. 
\end{proposition}

\begin{proof} 
Let $F=\Esf_A$ and denote the given  edge from $u$ to $v$ in $E$ by $f$.  Then $F$ is formed by removing $f$ but adding for each $e \in s_{E}^{-1} (v)$ an edge $\overline{e}$ with $s_{F} ( \overline{e} ) = u$ and $r_{F} ( \overline{e} ) = r_{E} ( e )$.  Moreover, since $v$ supports a loop, we have that $r_{E}^{-1}( v ) \setminus \{ f \} \neq \emptyset$.  Set $\mathcal{E}_{1} = r_{E}^{-1} (v) \setminus \{ f  \}$ and $\mathcal{E}_{2} = \{ f \}$.  Using this partition, we form $E_{{I}}$, which replaces $v$ with $v^{1}$ and $v^{2}$.  The vertex $v^{1}$ receives the edges of $v$ except $f$ and also receives one edge from $v^{2}$ for each loop based at $v$.  The vertex $v^{2}$ only receives the edge $f$.  Both vertices emit copies of the edges $v$ emitted and do so in such a way that there is no loop based at $v^{2}$.  By \cite[Proposition~3.6]{MR3045151}, there exists a \stariso $\ftn{\Psi_{1}}{ C^{*} (E) }{ p_{V} C^{*} ( E_{{I}} ) p_{V} }$ where $V =  E_I^{0}\backslash\{v^2\}$. 

Since $v^{2}$ does not support a loop, we may collapse this vertex, yielding $F$. 
Set $q_{w} = p_{w}^{E}$ for all  $w \in F^{0}$, $t_{e} = s_{e}^{E}$ for all $e\in E_I^1\setminus (r^{-1}_{E_I}(v^2)\cup s^{-1}_{E_I}(v^2))$ and $t_{[ee']}=s_e^Es_{e'}^E$ for $e\in r^{-1}_{E_I}(v^2)$ and $e'\in s^{-1}_{E_I}(v^2)$.
One can easily check that $\Psi_2 ( p_{v}^{F} ) = q_{v}$ and $\Psi_2 ( s_{e}^{F} ) = t_{e}$ provides a \stariso $\ftn{ \Psi_{2} }{ C^{*} (F) }{ p_{V} C^{*} ( E_{{I}} ) p_{V} }$.  Hence, $\ftn{ \Phi = \Psi_{2}^{-1} \circ \Psi_{1} }{ C^{*} (E) }{  C^{*} ( F ) }$ is a \stariso. \end{proof}

\section{The gauge invariant prime ideal space}\label{gipis}

We now provide definitions and fundamental results concerning the  gauge invariant prime ideal spaces of graph \cas. Although this is a very natural thing to do when we have the graph given, we are not aware of any place in the literature where this has been done only using the graph \ca and not the underlying graph. For the benefit of further applications elsewhere, we carry out the analysis in full generality.

\subsection{Structure of graph \texorpdfstring{$C^*$}{C*}-algebras}
 It is important for us to view the graph \cas as $X$-algebras over a topological space $X$ that --- in general --- is different from the primitive ideal space. This is due to the fact that when there exist ideals that are not gauge invariant, then there are infinitely many ideals. 
The space we choose to work with corresponds to the distinguished ideals being exactly the gauge invariant ideals. We show a \ca{ic} characterization of the gauge invariant ideals, and describe the space $X=\Prime_\gamma(C^*(E))$ in this subsection. 

\begin{definition}
Let $E=(E^0,E^1,r,s)$ be a graph. 
A subset $H\subseteq E^0$ is called \emph{hereditary} if whenever $v,w\in E^0$ with
$v\in H$ and $v\geq w$, then $w\in H$. 
A subset $S\subseteq E^0$ is called \emph{saturated} if whenever $v\in E_{\mathrm{reg}}^0$ with $r(s^{-1}(v))\subseteq S$, then $v\in S$. 
For any saturated hereditary subset $H$, the \emph{breaking vertices} corresponding to $H$ are the elements of the set
$$B_H :=\setof{v\in E^0}{|s^{-1}(v)|=\infty\text{ and } 0<|s^{-1}(v)\cap r^{-1}(E^0\setminus H)|<\infty }.$$

It is clear that $\emptyset$ and $E^0$ are both saturated and hereditary subsets.  The intersection of any family of hereditary subsets is again hereditary. 
Thus, for every subset $S\subseteq E^0$, there exists a smallest hereditary subset of $E^0$ containing $S$ --- this set is called the hereditary subset generated by $S$ and is denoted $H(S)$. 
The intersection of any family of saturated subsets is again saturated. 
Thus, for every subset $S\subseteq E^0$, there is a smallest saturated subset of $E^0$ containing $S$ --- this set is called the saturation of $S$ and is denoted $\overline{S}$.
The saturation of a hereditary set is again hereditary.
It is also clear that the union of any family of hereditary sets is again hereditary. 
This makes the set of saturated hereditary subsets of $E^0$ into a complete lattice. 

An \emph{admissible pair} $(H,S)$ consists of a saturated hereditary subset $H\subseteq E^0$ and a subset $S\subseteq B_H$. 
We order the collection of admissible pairs by defining $(H,S)\leq (H',S')$ if and only if $H\subseteq H'$ and $S\subseteq H'\cup S'$. 
This makes the collection of admissible pairs into a lattice. 
\end{definition}

\begin{fact}\label{fact:structure-1}
Let $E=(E^0,E^1,r,s)$ be a graph. 
For any admissible pair $(H,S)$, we let $\mathfrak{J}_{(H,S)}$ denote the ideal generated by
$$\setof{p_v}{v\in H}\cup\setof{p_{v_0}^{H}}{v_0\in S},$$
where $p_{v_0}^H$ is the \emph{gap projection}
$$p_{v_0}^H=p_{v_0}-\sum_{\substack{s(e)=v_0 \\ r(e)\not\in H}}s_es_e^*.$$
If $B_H=\emptyset$, for a saturated hereditary subset $H\subseteq E^0$, then we write $\mathfrak{J}_H$ for $\mathfrak{J}_{(H,\emptyset)}$. 
The map $(H,S)\mapsto \mathfrak{J}_{(H,S)}$ is a lattice isomorphism between the lattice of admissible pairs and the lattice of gauge invariant ideals of $C^*(E)$ (\cf\ \cite[Theorem~3.6]{MR1988256}).
\end{fact}

\begin{lemma}\label{lem:structure-1}
Let $E=(E^0,E^1,r_E,s_E)$ and $F=(F^0,F^1,r_F,s_F)$ be graphs and let $\mathfrak{I}$ be an ideal of $C^* (E)$.  Then $\mathfrak{I}$ is gauge-invariant if and only if $\mathfrak{I}$ is generated by projections.  Consequently, every \stariso from $C^*(E)$ to $C^*(F)$ will send gauge invariant ideals to gauge invariant ideals and every \stariso from $C^*(E)\otimes\K$ to $C^*(F)\otimes\K$ will send gauge invariant ideals to gauge invariant ideals under the identification of the ideal lattice of $C^*(E)$ and $C^*(F)$ with the ideal lattice of $C^*(E)\otimes\K$ and $C^*(F)\otimes\K$, respectively.
\end{lemma}

\begin{proof}
Suppose $\mathfrak{I}$ is a gauge-invariant ideal.  Then by Fact~\ref{fact:structure-1}, $\mathfrak{I}$ is generated by vertex projections and gap projections.  Suppose $\mathfrak{I}$ is generated by projections $S = \{ p_1, p_2, \dots \}$.  By \cite[Theorem~3.4 and Corollary~3.5]{MR3310950}, each $p_i$ is Murray-von~Neumann equivalent to sums of vertex projections and gap projections in $C^*(E)$, where the Murray-von~Neumann equivalence and sums are in $C^*( E) \otimes \K$.  
But this implies that $\overline{ C^* (E) p_i C^* (E) }$ is generated by vertex projections and gap projections.  
Hence, $\mathfrak{I} = \overline{\operatorname{span} C^* (E) S C^*(E) }$ is generated by vertex projections and gap projections.  Since vertex projections and gap projections are fixed by the gauge action, we have that $\mathfrak{I}$ is a gauge-invariant ideal.

Suppose $\Phi \colon C^* (E) \rightarrow C^* (F)$ is a \stariso.  Let $\mathfrak{I}$ be a gauge-invariant ideal of $C^* (E)$.  Then from the first part of the lemma, we have that $\mathfrak{I}$ is generated by projections.  Since $\Phi$ is a \stariso, we have that $\Phi ( \mathfrak{I} )$ is also generated by projections.  Thus, $\Phi ( \mathfrak{I} )$ is a gauge-invariant ideal.  

For a \ca \A, we say an ideal in $\A \otimes \K$ is generated by projections in \A if it is generated by projections in $\mathfrak{A} \otimes e_{11}$.  
Suppose $\Psi \colon C^* (E) \otimes \K \rightarrow C^* (F) \otimes \K$ is a \stariso.  
Let $\mathfrak{I}$ be a gauge-invariant ideal of $C^* (E) \otimes \K$, \ie, $\mathfrak{I} = \mathfrak{J}_{ (H,S) } \otimes \K$.  
So, in particular, $\mathfrak{I}$ is generated by projections.  
Since $\Psi$ is a \stariso, $\Psi ( \mathfrak{I} )$ is generated by projections.  
By \cite[Theorem~3.4 and Corollary~3.5]{MR3310950} and using a similar argument as in the first paragraph, we get that $\Psi ( \mathfrak{I} )$ is generated by vertex projections and gap projections in $C^* (F)$.  
Hence, $\Psi ( \mathfrak{I} )$ is gauge-invariant.
\end{proof}

\begin{definition}
Let $E = (E^0 , E^1 , r , s )$ be a graph. 
Let $\Prime_\gamma(C^*(E))$ denote the set of all proper ideals that are prime within the set of proper gauge invariant ideals, \ie, $\mathfrak{p}\in\Prime_\gamma(C^*(E))$ if and only if $\mathfrak{p}$ is a proper gauge invariant ideal of $C^*(E)$ and
$$\mathfrak{I}_1\mathfrak{I}_2\subseteq\mathfrak{p}
\Rightarrow\mathfrak{I}_1\subseteq\mathfrak{p}\vee\mathfrak{I}_2\subseteq\mathfrak{p},$$
for all (proper) gauge invariant ideals $\mathfrak{I}_1,\mathfrak{I}_2$ of $C^*(E)$. 

Recall that for an ideal $\mathfrak{I}$, we let $\operatorname{hull}(\mathfrak{I})$ denote the set of primitive ideals containing $\mathfrak{I}$, \ie, $\setof{\mathfrak{p} \in\Prim(C^*(E))}{\mathfrak{p}\supseteq\mathfrak{I}}$. 
Similarly, for every ideal $\mathfrak{I}$, we let $\operatorname{hull}_\gamma(\mathfrak{I})$ denote the set $\setof{\mathfrak{p} \in\Prime_\gamma(C^*(E))}{\mathfrak{p}\supseteq\mathfrak{I}}$. 
We equip $\Prime_\gamma(C^*(E))$ with a topology similar to the hull-kernel topology for primitive ideals, \ie, the closure of a subset $S\subseteq\Prime_\gamma(C^*(E))$ is 
$$\operatorname{hull}_\gamma(\cap S)=\setof{\mathfrak{p}\in\Prime_\gamma(C^*(E))}{\mathfrak{p}\supseteq\cap S}.$$
To check that this closure operation defines a unique topology, we need only to check that it satisfies the four Kuratowski closure axioms --- but the first two paragraphs of \cite[5.4.6~Theorem]{MR1074574} show this. 
With an argument similar to \cite[5.4.7~Theorem]{MR1074574}, it also follows that the topology is $T_0$. 
\end{definition}

When $\fct{\Phi}{C^*(E)}{C^*(F)}$ is a \stariso, we get by Lemma \ref{lem:structure-1} an induced homeomorphism $\fct{\Phi_\sharp}{\Prime_\gamma(C^*(E))}{\Prime_\gamma(C^*(F))}$.

It is an elementary fact, that every primitive ideal of a \ca is a
(closed) prime ideal (\eg\ \cite[5.4.5~Theorem]{MR1074574}). For a
separable \ca, the converse is true, which can be seen by showing that
the primitive ideal space of a separable \ca is a Baire space
(\eg\ \cite[II.6.5.15~Corollary]{MR2188261}), but as shown by Weaver in \cite{MR2003352}
 the concepts differ for nonseparable  $C^*$-algebras. In fact, there are counterexamples even for nonseparable graph $C^*$-algebras
(see \cite{MR3426227}), but since we only consider countable graphs, this will not be an issue here.

\begin{lemma}\label{lem:prime-ideals-1}
Let $E = (E^0 , E^1 , r , s )$ be a graph.
Every primitive gauge invariant ideal of $C^*(E)$ is in $\Prime_\gamma(C^*(E))$. 
Every primitive ideal of $C^*(E)$ that is not gauge invariant has a largest gauge invariant ideal as a subset, and this gauge invariant ideal is in $\Prime_\gamma(C^*(E))$. 

If $\mathfrak{I}$ is a proper gauge invariant ideal of $C^*(E)$, then 
$$\mathfrak{I}=\cap\setof{\mathfrak{p}\in\Prime_\gamma(C^*(E))}{\mathfrak{p}\supseteq \mathfrak{I}}=\cap\operatorname{hull}_\gamma(\mathfrak{I}).$$
\end{lemma}
\begin{proof}
First, note that all primitive ideals of $C^*(E)$ are described in \cite[Corollary~2.11]{MR2023453} --- we will use the terminology from there. As pointed out in \cite{MR3142035} there is a minor mistake in the description of the topology of the primitive ideal space in \cite{MR2023453}, but this has no consequences for this paper, since we are not using the description of the topology. 
So we have a bijection from $\mathcal{M}_\gamma(E)\sqcup BV(E)\sqcup(\mathcal{M}_\tau(E)\times\mathbb{T})$ to $\Prim(C^*(E))$ given by 
\begin{align*}
\mathcal{M}_\gamma(E)\ni M&\mapsto \mathfrak{J}_{\Omega(M),\Omega(M)_\infty^\mathrm{fin}}, \\
BV(E)\ni v&\mapsto \mathfrak{J}_{\Omega(v),\Omega(v)_\infty^\mathrm{fin}\setminus\{v\}}, \\
\mathcal{M}_\tau(E)\times\mathbb{T}\ni (N,z)&\mapsto \mathfrak{R}_{N,t},
\end{align*}
where the gauge invariant primitive ideals are exactly the ideals coming from $\mathcal{M}_\gamma(E)$ and $BV(E)$. 
Note that every gauge invariant primitive ideal of $C^*(E)$ is also prime in the set of ideals of $C^*(E)$. 
Thus every gauge invariant primitive ideal of $C^*(E)$ is in $\Prime_\gamma(C^*(E))$. 

Note that for $N\in\mathcal{M}_\tau(E)$ and $z\in\mathbb{T}$, the ideal $\mathfrak{J}_{\Omega(N),\Omega(N)_\infty^\mathrm{fin}}$ is the largest gauge invariant ideal contained in $R_{N,z}$ (\cf\ \cite[Lemma~2.6]{MR2023453}). 

Let $N\in\mathcal{M}_\tau(E)$ and assume that $\mathfrak{I}_1\mathfrak{I}_2\subseteq\mathfrak{J}_{\Omega(N),\Omega(N)_\infty^\mathrm{fin}}$ for some gauge invariant ideals $\mathfrak{I}_1,\mathfrak{I}_2$. 
Then $\mathfrak{I}_1\mathfrak{I}_2\subseteq\mathfrak{R}_{N,-1}$. 
Since $\mathfrak{R}_{N,-1}$ is a primitive ideal in $C^*(E)$, it is prime in the collection of all ideals of $C^*(E)$. Therefore either $\mathfrak{I}_1\subseteq\mathfrak{R}_{N,-1}$ or $\mathfrak{I}_2\subseteq\mathfrak{R}_{N,-1}$.
But since $\mathfrak{J}_{\Omega(N),\Omega(N)_\infty^\mathrm{fin}}$ is the largest gauge invariant ideal contained in $\mathfrak{R}_{N,-1}$, we have $\mathfrak{I}_1\subseteq\mathfrak{J}_{\Omega(N),\Omega(N)_\infty^\mathrm{fin}}$ or 
$\mathfrak{I}_2\subseteq\mathfrak{J}_{\Omega(N),\Omega(N)_\infty^\mathrm{fin}}$.
This shows that also $\mathfrak{J}_{\Omega(N),\Omega(N)_\infty^\mathrm{fin}}\in\Prime_\gamma(C^*(E))$ when $N\in\mathcal{M}_\tau(E)$.

Let $\mathfrak{I}$ be a proper gauge invariant ideal of $C^*(E)$. 
Then $\mathfrak{I}$ is the intersection of all the primitive ideals containing it. 
The only primitive ideals that are not in $\Prime_\gamma(C^*(E))$ are the ideals $\mathfrak{R}_{N,z}$ for $N\in\mathcal{M}_\tau(E)$ and $z\in\mathbb{T}$ --- but if $\mathfrak{I}\subseteq\mathfrak{R}_{N,z}$ then we can replace it in the intersection by the ideal $\mathfrak{J}_{\Omega(N),\Omega(N)_\infty^\mathrm{fin}}\in\Prime_\gamma(C^*(E))$. 
So we have shown that 
\begin{align*}
\mathfrak{I}&=\bigcap\Big(\setof{\mathfrak{J}_{\Omega(M),\Omega(M)_\infty^{\mathrm{fin}}}}{\mathfrak{J}_{\Omega(M),\Omega(M)_\infty^{\mathrm{fin}}}\supseteq \mathfrak{I},M\in\mathcal{M}_\gamma(E)\cup\mathcal{M}_\tau(E)} \\
&\qquad\qquad\cup\setof{\mathfrak{J}_{\Omega(v),\Omega(v)_\infty^{\mathrm{fin}}\setminus\{v\}}}{\mathfrak{J}_{\Omega(v),\Omega(v)_\infty^{\mathrm{fin}}\setminus\{v\}}\supseteq \mathfrak{I},v\in BV(E)}\Big) \\
&= \bigcap_{\substack{\mathfrak{p}\in\Prim_\gamma(C^*(E)) \\ \mathfrak{p}\supseteq\mathfrak{I}}}\mathfrak{p}
=\cap\operatorname{hull}_\gamma(\mathfrak{I}),
\end{align*}
since the second intersection contains all the sets from the first intersection. 
\end{proof}

\begin{lemma}\label{lem:orderreversingprime}
The map 
$$\mathfrak{I}\mapsto\operatorname{hull}_\gamma(\mathfrak{I})
=\setof{\mathfrak{p}\in\Prime_\gamma(C^*(E))}{\mathfrak{p}\supseteq\mathfrak{I}}$$ 
is an order-reversing $1-1$ correspondence between  the gauge invariant ideals of $C^*(E)$ and the closed subsets of $\Prime_\gamma(C^*(E))$. 
Its inverse map is $S\mapsto \cap S$.
\end{lemma}
\begin{proof}
This proof follows the lines of the proof of \cite[5.4.7~Theorem]{MR1074574}.
\end{proof}

The following lemma tells us exactly how we may consider a graph \ca as an algebra over $\Prime_\gamma(C^*(E))$ such that the distinguished ideals are exactly the gauge invariant ideals.

\begin{lemma}\label{lem:Xaction}
Let $E$ be a graph.
Consider the map $\zeta$ from $\Prim(C^{*}(E))$ to $\Prime_\gamma(C^{*}(E))$ sending each primitive ideal to the largest element of $\Prime_\gamma(C^{*}(E))$ that it contains. 
This map is continuous and surjective, and it makes $C^{*}(E)$ into a $\Prime_\gamma(C^{*}(E))$-algebra in a canonical way. Moreover,
\begin{equation}\label{eq:lem:Xaction}
\zeta^{-1}(\operatorname{hull}_\gamma(\mathfrak{I}))
=\operatorname{hull}(\mathfrak{I}),
\end{equation}
for every gauge invariant ideal $\mathfrak{I}$ of $C^*(E)$, so the distinguished ideals under the action are exactly the gauge invariant ideals. 
\end{lemma}

\begin{proof}
The validity of the definition of the map $\zeta$ follows from Lemma~\ref{lem:prime-ideals-1}. 
First we show \eqref{eq:lem:Xaction}. Then continuity follows since every closed set of $\Prime_\gamma(C^*(E))$ is of the form $\operatorname{hull}_\gamma(\mathfrak{I})$. So let $\mathfrak{I}$ be a gauge invariant ideal.

Let $\mathfrak{p}\in \zeta^{-1}(\operatorname{hull}_\gamma(\mathfrak{I}))$. 
Then $\zeta(\mathfrak{p})\supseteq\mathfrak{I}$. 
Since, by definition, $\mathfrak{p}\supseteq\zeta(\mathfrak{p})$, it is clear that  $\mathfrak{p}\supseteq\mathfrak{I}$. Therefore $\mathfrak{p}\in \operatorname{hull}(\mathfrak{I})$. 

Now let $\mathfrak{p}\in \operatorname{hull}(\mathfrak{I})$. 
Then $\mathfrak{p}\supseteq\mathfrak{I}$. 
If $\mathfrak{p}$ is gauge invariant, then $\zeta(\mathfrak{p})=\mathfrak{p}\supseteq\mathfrak{I}$, so $\mathfrak{p}\in \zeta^{-1}(\operatorname{hull}_\gamma(\mathfrak{I}))$.
If, on the other hand, $\mathfrak{p}$ is not gauge invariant, then $\zeta(\mathfrak{p})$ is the largest gauge invariant ideal contained in $\mathfrak{p}$, \cf~Lemma~\ref{lem:prime-ideals-1}. 
Thus $\mathfrak{p}\supseteq\zeta(\mathfrak{p})\supseteq\mathfrak{I}$, so also in this case $\mathfrak{p}\in \zeta^{-1}(\operatorname{hull}_\gamma(\mathfrak{I}))$. 

Now we want to show surjectivity of the map. For this we use the notation of \cite{MR1988256} and \cite{MR2023453} and the content of the proof of Lemma~\ref{lem:prime-ideals-1}.
Recall that every gauge invariant ideal $\mathfrak{I}$ of $C^*(E)$ is of the form $\mathfrak{I}=\mathfrak{J}_{H,B}$ for some saturated hereditary subset $H \subseteq E^0$ and some subset $B\subseteq B_H=H_\infty^\textrm{fin}$ --- in fact, if $H_\mathfrak{I}=\setof{v\in E^0}{p_v\in \mathfrak{I}}$ and $B_\mathfrak{I} = \setof{v\in B_{H_\mathfrak{I}}}{p_v^{H_\mathfrak{I}}\in\mathfrak{I} }$, then $\mathfrak{I}=\mathfrak{J}_{H_\mathfrak{I},B_\mathfrak{I}}$. 
Note that if $(H,S_1)$ and $(H,S_2)$ are admissible pairs, then $(H,S_1)\wedge(H,S_2)$ is $(H,S_1\cap S_2)$. 

Now assume that $\mathfrak{I}\in\Prime_\gamma(C^*(E))$, so $\mathfrak{I}=\mathfrak{J}_{H_\mathfrak{I},B_\mathfrak{I}}$.  
Since $\mathfrak{I}$ is a proper ideal, $H_\mathfrak{I}\neq E^0$, so $M=E^0\setminus H_\mathfrak{I}$ is nonempty. 
The proof of \cite[Lemma~4.1]{MR1988256} shows that $M$ is a maximal tail. 
Note that $\Omega(M)=E^0\setminus M=H_\mathfrak{I}$. 

We want to show that $|B_{H_\mathfrak{I}}\setminus B_\mathfrak{I}|\leq 1$. 
So assume that $v_1,v_2\in B_{H_\mathfrak{I}}\setminus B_\mathfrak{I}$ with $v_1\neq v_2$. 
It follows from \cite[Proposition~3.9]{MR1988256} that
$$\mathfrak{J}_{H_\mathfrak{I},B_\mathfrak{I}\cup\{v_1\}}\cap \mathfrak{J}_{H_\mathfrak{I},B_\mathfrak{I}\cup\{v_2\}}
=\mathfrak{J}_{H_\mathfrak{I},B_\mathfrak{I}}=\mathfrak{I}.$$
But $\mathfrak{J}_{H_\mathfrak{I},B_\mathfrak{I}\cup\{v_i\}} \not\subseteq\mathfrak{J}_{H_\mathfrak{I},B_\mathfrak{I}}=\mathfrak{I}$, for $i=1,2$, which contradicts that $\mathfrak{I}$ is prime within the proper gauge invariant ideals of $C^*(E)$. Hence $|B_{H_\mathfrak{I}}\setminus B_\mathfrak{I}|\leq 1$. 

Now assume that $B_{H_\mathfrak{I}}\setminus B_\mathfrak{I}=\{v\}$. 
We want to show that $v\in BV(E)$, \ie, we need to show that $v$ supports a cycle. 
So assume that $v$ does not support a cycle. Since $v$ is an infinite emitter, $H_2=\overline{H(v)\setminus\{v\}}$ is a saturated hereditary subset not containing $v$. 
Note that $v\not\in B_{H_2}$. 
From \cite[Proposition~3.9]{MR1988256}, it follows that
$$\mathfrak{J}_{H_\mathfrak{I},B_{H_\mathfrak{I}}}
\cap\mathfrak{J}_{H_2,B_{H_2}}
\subseteq\mathfrak{J}_{H_\mathfrak{I},B_\mathfrak{I}}=\mathfrak{I}.$$
But $\mathfrak{J}_{H_\mathfrak{I},B_{H_\mathfrak{I}}}\not\subseteq\mathfrak{I}$ and $\mathfrak{J}_{H_2,B_{H_2}}\not\subseteq\mathfrak{I}$, which contradicts that $\mathfrak{I}$ is prime within the proper gauge invariant ideals of $C^*(E)$. Hence $v\in BV(E)$. 
Now we also want to show that $\Omega(v)=H_\mathfrak{I}$. 
From the definition, it is clear that $\Omega(v)\supseteq H_\mathfrak{I}$. 
From \cite[Proposition~3.9]{MR1988256}, it follows that
$$\mathfrak{J}_{\Omega(v),B_{\Omega(v)}\setminus \{v\}}
\cap\mathfrak{J}_{H_\mathfrak{I},B_{H_\mathfrak{I}}}
\subseteq\mathfrak{J}_{H_\mathfrak{I},B_\mathfrak{I}}=\mathfrak{I}.$$
Since $\mathfrak{J}_{H_\mathfrak{I},B_{H_\mathfrak{I}}}\not \subseteq \mathfrak{I}$ and $\mathfrak{I}$ is prime within the proper gauge invariant ideals of $C^*(E)$, it follows that $\mathfrak{J}_{\Omega(v),B_{\Omega(v)}\setminus \{v\}}\subseteq\mathfrak{I}$. 
Therefore $\Omega(v)\subseteq H_\mathfrak{I}$. 

Now it follows from the proof of Lemma~\ref{lem:prime-ideals-1}, that $\zeta$ is surjective.
\end{proof}

\begin{remark}
Assume that $E$ is a graph with finitely many vertices. Then $E$ satisfies Condition~(K) if and only if $C^*(E)$ has finitely many ideals, and in this case
$\Prim(C^*(E))=\Prime_\gamma(C^*(E))$. 
\end{remark}

 \subsection{The component poset}
For our purposes, it will be essential to work with block matrices in a way that resembles the ideal structure and the filtered $K$-theory of the graph \cas. 
To do this, we need to put the graph in a certain form and to order the vertices in a certain way such that the adjacency matrix has a certain nice block form. 
It is also essential to our work, that the topological space $\Prime_\gamma(C^*(E))$ is built into this construction. 
For the benefit of possible applications to other settings, we will allow infinite emitters, but it is essential for the exposition that we allow only \emph{finitely many vertices}. 

As we shall see, it will be necessary to modify the given graph up to move equivalence to deal with certain complication introduced by transitional and breaking vertices. This will not change the $C^*$-algebras in question up to stable isomorphism, and is hence unproblematic for the work in this paper. But to pave the way for classification of the graph $C^*$-algebras themselves, we keep track of the isomorphism class as far as possible.

\begin{definition}\label{def:structure-a}
Let $E=(E^0,E^1,r,s)$ be a graph with finitely many vertices. 
We say that a nonempty subset $S$ of $E^{0}$ is \emph{strongly connected} if for any two vertices $v,w\in S$ there exists a nonempty path from $v$ to $w$. 
In particular every vertex in a strongly connected set has to be the base of a cycle. 
The maximal strongly connected subsets of $E^{0}$ are all disjoint, and these are called the strongly connected components of $E$. 
We let $\Gamma_E$ denote the set of all strongly connected components together with all singletons consisting of singular vertices that are not the base point of a cycle. 
The sets in $\Gamma_E$ are all disjoint. 
We call the sets in $\Gamma_E$ the \emph{components of the graph} $E$ and the vertices in $E^{0}\setminus\cup\Gamma_E$ the \emph{transition states} of $E$ --- the transition states are by definition all the regular vertices that are not the base point of a cycle.  Note that with this terminology, all regular sources are also transition states.  A strongly connected component is called a \emph{cyclic component} if one of its vertices (and thus all of its vertices) has exactly one return path.

We define a relation $\geq$ on $\Gamma_E$ by saying that $\gamma_1\geq\gamma_2$ if there exist vertices $v_1\in\gamma_1$ and $v_2\in\gamma_2$ such that $v_1\geq v_2$. 
By definition this is the same as for all vertices $v_1\in\gamma_1$ and all vertices  $v_2\in\gamma_2$ we have that $v_1\geq v_2$. 
Thus it is clear that $\geq$ is a partial order. 

We say that a subset $\sigma\subseteq\Gamma_E$ is hereditary if whenever $\gamma_1,\gamma_2\in \Gamma_E$ with $\gamma_1\in \sigma$ and $\gamma_1\geq\gamma_2$, then $\gamma_2\in \sigma$. 
We equip $\Gamma_E$ with the topology that has the hereditary subsets as open sets --- this makes $\Gamma_E$ into a $T_0$-space. 
For every subset $\sigma\subseteq\Gamma_E$, we let $\eta(\sigma)$ denote the smallest hereditary subset of $\Gamma_E$ containing $\sigma$, \ie, the set $\setof{\gamma\in\Gamma_E}{\exists \gamma'\in\sigma\colon\gamma'\geq\gamma}$. 
\end{definition}

We recall the definition of an Alexandrov space and some of their properties.

\begin{definition}\label{def:structure-b}
A topological space is called an \emph{Alexandrov space} if arbitrary intersections of open subsets are again open. If we have a topological space $X$, then we can define a preorder on $X$ by $x\geq y$ if and only if $x$ is in the closure of $\{y\}$ --- this preorder is called the \emph{specialization preorder}. In the opposite direction, for a preordered set $(X,\geq)$ we can let the sets $F\subseteq X$ satisfying $x\geq y\wedge y\in F\Rightarrow x\in F$ be the closed sets. This topology is the finest topology satisfying that $x\geq y$ if and only if $x$ is in the closure of $\{y\}$. It is also clear that this is an Alexandrov topology.

If an Alexandrov space is given, and we take its specialization preorder, then the Alexandrov topology is uniquely determined from the specialization preorder by the above construction. Thus there is a one-to-one correspondence between Alexandrov topologies and preorders on a space. A map between two Alexandrov spaces is continuous if and only if it is an order preserving map with respect to the specialization preorders. 
\end{definition}

Note that often the specialization preorder is written as the opposite order compared to above. Both conventions are used in the literature, while the convention used here is chosen since it fits better with our setup, as we will see now. 

\begin{remark}\label{rem:structure-a}
We will mainly consider the topological spaces $\Prime_\gamma(C^*(E))$ and $\Gamma_E$ for graphs with finitely many vertices. Assume that $E$ is a graph with finitely many vertices. Although $\Prim(C^*(E))$ often will be infinite (in the case of a cyclic component), the sets $\Prime_\gamma(C^*(E))$ and $\Gamma_E$ are finite. Thus it is clear that arbitrary intersections of open subsets are again open. 
Thus $\Prime_\gamma(C^{*}(E))$ is an Alexandrov space. We see immediately from the definition that $\mathfrak{p}_1$ is in the closure of $\{\mathfrak{p}_2\}$ if and only if $\mathfrak{p}_1\supseteq\mathfrak{p}_2$. 
So the specialization preorder $\geq$ is set containment. 
Similarly, $\Gamma_E$ is an Alexandrov space and its specialization preorder is exactly the order $\geq$. 
\end{remark}

\begin{lemma}\label{lem:structure-a}
Let $E=(E^0,E^1,r,s)$ be a graph with finitely many vertices. 
Let $\eta\subseteq\Gamma_E$ be a hereditary subset. Assume that $v\in E^0_\mathrm{reg}$ and that there is no path from $v$ to any of the components in $\Gamma_E\setminus\eta$. 
Then $v\in\overline{H(\cup\eta)}$. 
\end{lemma}

\begin{proof}
There has to be a path from $v$ to some component --- thus a component in $\eta$. 
If $v$ supports a cycle, clearly $v\in\cup\eta\subseteq\overline{H(\cup\eta)}$. Let $H_0=H(\cup\eta)$. 
Using the description in \cite[Remark~3.1]{MR1988256}, we get a non-decreasing sequence of hereditary sets $\Sigma_0(H_0)=H_0$, $\Sigma_1(H_0)$, $\Sigma_2(H_0)$, \ldots. 
If $v\not\in \Sigma_k(H_0)$, then the length of the longest path from $v$ to $\Sigma_k(H_0)$ is one less than the length of the longest path from $v$ to $\Sigma_{k-1}(H_0)$. Thus eventually $v\in\Sigma_k(H_0)$ for some $k$, \ie, $v\in\overline{H_0}$. 
\end{proof}

\begin{lemma}\label{lem:structure-b}
Let $E=(E^0,E^1,r,s)$ be a graph with finitely many vertices. 
Then the map $\eta\mapsto\overline{H(\cup\eta)}$ from the set of hereditary subsets of $\Gamma_E$ to the set of saturated hereditary subsets of $E^0$ is a bijective order isomorphism (with respect to the order coming from set containment). 
In fact, $\cup\eta=(\cup\Gamma_E)\cap\overline{H(\cup\eta)}$. 
Moreover, for any saturated hereditary subset $H\subseteq E^0$, 
the set $(\cup\Gamma_E)\cap H$ is a (disjoint) union of all components that intersect $H$ nontrivially; and if we let $\eta\subseteq\Gamma_E$ be the set of these components, then $\eta$ is hereditary and 
$\overline{H(\cup\eta)}=H$.
\end{lemma}

\begin{proof}
Assume that $\eta\subseteq\Gamma_E$ is a hereditary subset. 
Clearly, $\cup\eta\subseteq(\cup\Gamma_E)\cap\overline{H(\cup\eta)}$. 
Let $v\in(\cup\Gamma_E)\cap\overline{H(\cup\eta)}$. 
Suppose that $v\in\gamma_1\in\Gamma_E$ but $\gamma_1\not\in\eta$. 
Then $v\not\in H(\cup\eta)$. 
Let $H_0=H(\cup\eta)$. 
Using the description in \cite[Remark~3.1]{MR1988256}, we get a non-decreasing sequence of hereditary sets $\Sigma_0(H_0)=H_0$, $\Sigma_1(H_0)$, $\Sigma_2(H_0)$, \ldots, such that $v\in\Sigma_k(H_0)\setminus\Sigma_{k-1}(H_0)$, for some $k=1,2,3,\ldots$. 
This means that $v\in E_\mathrm{reg}^0$ and $r(s^{-1}(v))\subseteq\Sigma_{k-1}(H_0)$. 
Thus, clearly $v$ cannot support a loop. But $v$ cannot either support a cycle, since $\Sigma_{k-1}(H_0)$ is hereditary and all edges that $v$ emit go into $\Sigma_{k-1}(H_0)$.
So we get a contradiction, and therefore $v\in\cup\eta$. 

So now it is clear that we have an injective map $\eta\mapsto\overline{H(\cup\eta)}$ from the set of hereditary subsets of $\Gamma_E$ to the set of saturated hereditary subsets of $E^0$. It is also clear that it is order preserving. 

Now let there be given a saturated hereditary subset $H\subseteq E^0$. 
For each $v\in(\cup\Gamma_E)\cap H$, all $v'$ that belong to the same component as $v$ are elements of $(\cup\Gamma_E)\cap H$. 
So let $\eta\subseteq\Gamma_E$ be the (uniquely determined) set such that 
$\cup\eta=(\cup\Gamma_E)\cap H$. 
Since $\cup\eta\subseteq H$, it is clear that $H(\cup\eta)\subseteq H$. 
Let $H_0=\overline{H(\cup\eta)}\subseteq H$. 
Suppose $v\in H\setminus H_0$. 
Then $v$ needs to be a transition state, so $v\in E_\mathrm{reg}^0$ and $v$ does not support a cycle. Consequently, it has to have a path to at least one component, but it cannot have any path to a component not in $\eta$. Lemma~\ref{lem:structure-a} now implies that $v\in H_0$, which is a contradiction. 
Consequently, $H_0=\overline{H(\cup\eta)}= H$, and therefore the map is surjective. 
\end{proof}

As an immediate consequence we get the following corollary. 

\begin{corollary}\label{cor:structure-a}
Let $E=(E^0,E^1,r,s)$ be a graph with finitely many vertices, and assume that $E$ does not have any transition state. 
Then every hereditary subset of $E^0$ is saturated and $\eta\mapsto \cup \eta$ is a lattice isomorphism between the hereditary subsets of $\Gamma_E$ and the saturated hereditary subsets of $E^0$. 
\end{corollary}

The following is also clear.

\begin{lemma}\label{lem:structure-c}
Let $E=(E^0,E^1,r,s)$ be a graph with finitely many vertices. 
If every infinite emitter emits infinitely many edges to any vertex it emits any edge to, then $B_H=\emptyset$ for every saturated hereditary subset $H \subseteq E^0$. 
\end{lemma}

\begin{lemma}\label{lem:structure-d}
Let $E=(E^0,E^1,r,s)$ be a graph with finitely many vertices, and assume that every infinite emitter emits infinitely many edges to any vertex it emits any edge to. 

Define a map \fct{\wastheta_E}{\Gamma_E}{\Prime_\gamma(C^*(E))} as follows. For each $\gamma_0\in\Gamma_E$, let $\wastheta_E(\gamma_0)$ denote the ideal 
$$\mathfrak{J}_{\overline{H(\cup\eta_{\gamma_0}})},$$
where 
$$\eta_{\gamma_0}=\Gamma_E\setminus\setof{\gamma\in\Gamma_E}{\gamma\geq\gamma_0}.$$
This is in fact an element of $\Prime_\gamma(C^*(E))$ and this makes $\wastheta_E$ into a bijection. 
Moreover, $\gamma_1\geq\gamma_2$ if and only if $\wastheta_E(\gamma_1)\supseteq \wastheta_E(\gamma_2)$. Consequently, $\wastheta_E$ is a homeomorphism. 
\end{lemma}
\begin{proof}
From \cite{MR2023453} and the proof of Lemma~\ref{lem:prime-ideals-1}, it is clear that the ideals in $\Prime_\gamma(C^{*}(E))$ are exactly the ideals $\mathfrak{J}_{E^{0}\setminus M}$, where $M\neq\emptyset$ is a maximal tail. 
Assume that $M\neq\emptyset$ and let $H=E^0\setminus M$. 
That $M$ is a maximal tail means that $M$ satisfies the conditions (MT1), (MT2) and (MT3) in \cite{MR2023453}. 
Condition (MT1) is equivalent to $H$ being hereditary, while (MT2) is equivalent to $H$ being saturated. 
Since $E^0$ is assumed to be finite, (MT3) is equivalent to the existence of $w\in M$ such that $v\geq w$ for all $v\in M$, \ie, $M$ has a least element (we will use this terminology although this is only a preorder and not a partial order in general). 

Let $\gamma_0\in\Gamma_E$, and let 
$$\eta_{\gamma_0}=\Gamma_E\setminus\setof{\gamma\in\Gamma_E}{\gamma\geq\gamma_0}.$$
It is clear that $\eta_{\gamma_0}$ is hereditary. Clearly, by the definition above $\wastheta_E(\gamma_0)$ defines an ideal. Let 
$$H_0=\overline{H(\cup\eta_{\gamma_0})}.$$
We want to show that $E^0\setminus H_0$ is a maximal tail. The only thing we need to show is that it has a least element. 
Choose $v_0\in\gamma_0$, and let $v\in E^0\setminus H_0$ be given. 
Assume that $v\not\geq v_0$. 
If $v\in\cup\Gamma_E$, then $v\in \cup\eta_{\gamma_0}$ and thus $v\in H_0$ (which is a contradiction). 
Therefore, we would need to have that $v$ is a transition state --- so $v\in E_\mathrm{reg}^0$ and $v$ does not support a cycle. 
There exists a path to some component in $\Gamma_E$, and, clearly, no such component can be in $\Gamma_E\setminus\eta_{\gamma_0}=\setof{\gamma\in\Gamma_E}{\gamma\geq\gamma_0}$. 
From Lemma~\ref{lem:structure-a} it follows that $v\in\overline{H(\cup\eta_{\gamma_0})}=H_0$, which is a contradiction as well. Therefore $E^0\setminus H_0$ is a maximal tail, and $\wastheta_E(\gamma_0)$ is an element of $\Prime_\gamma(C^*(E))$. 

From Fact~\ref{fact:structure-1} and Lemma~\ref{lem:structure-a} it follows that $\wastheta_E$ is injective. 
Given an element of $\Prime_\gamma(C^*(E))$, then it has to be of the form $\mathfrak{J}_{H_0}$ for some saturated hereditary subset $H_0\subsetneq E^0$ with $E^0\setminus H_0$ having a least element $v_0$. 
First note that $v_0$ cannot be a transition state, so $v_0\in\gamma_0$ for some $\gamma_0\in\Gamma_E$. 
Let $\eta\subseteq\Gamma_E$ be such that $\cup\eta=(\cup\Gamma_E)\cap H_0$. 
Clearly $\gamma_0\not\in\eta$. 
Let $v\in\gamma\in\Gamma_E\setminus\eta$. 
Then $v\geq v_0$, since $v\in\gamma\subseteq E^{0}\setminus H_0$. 
Consequently, $\gamma\geq\gamma_0$. 
On the other hand, assume that $\gamma\geq\gamma_0$ and let $v\in\gamma$. 
Then $v\geq v_0$, so $v\in E^{0}\setminus H_0$. 
Consequently, $\Gamma_E\setminus\eta= \setof{\gamma\in\Gamma_E}{\gamma\geq\gamma_0}$. 
Thus the map $\wastheta_E$ is surjective. 

That $\gamma_1\geq\gamma_2$ implies $\wastheta_E(\gamma_1)\supseteq \wastheta_E(\gamma_2)$ is clear from the definition. 
That $\wastheta_E(\gamma_1)\supseteq \wastheta_E(\gamma_2)$ implies $\gamma_1\geq\gamma_2$ is clear from the definition and Lemma~\ref{lem:structure-b}. 
\end{proof}

\begin{lemma}\label{lem:structure-2}
Let $E=(E^0,E^1,r,s)$ be a graph with finitely many vertices. 
\begin{enumerate}[(i)]
\item\label{lem:structure-2-2}
If every transition state has exactly one edge going out, then 
$H_1\cup H_2=\overline{H_1\cup H_2}$ for all saturated hereditary subsets $H_1,H_2\subseteq E^0$. 
\item\label{lem:structure-2-3}
If $\gamma\in\Gamma_E$, then $\overline{H(\gamma)\setminus\gamma}$ is the largest proper saturated hereditary subset of $\overline{H(\gamma)}$. 
\item\label{lem:structure-2-4}
If every transition state has exactly one edge going out, then the collection 
$\overline{H(\gamma)}\setminus\overline{H(\gamma)\setminus\gamma}$, $\gamma\in\Gamma_E$ is a partition of $E^0$. 
\item\label{lem:structure-2-6}
There exists a graph $F$ with finitely many vertices such that every infinite emitter emits infinitely many edges to any vertex it emits any edge to, every transition state has exactly one edge going out, $E\Meq F$, and $C^*(E)\cong C^*(F)$, 
\item\label{lem:structure-2-7}
If every infinite emitter in $E$ emits infinitely many edges to any vertex it emits any edge to and every transition state has exactly one edge going out, then there exists a graph $F$ with finitely many vertices, such that every infinite emitter emits infinitely many edges to any vertex it emits any edge to, $F$ has no transition states $F^0=\cup\Gamma_E\subseteq E^0$, $\Gamma_E=\Gamma_{F}$ and they carry the same order $\geq$, 
\begin{equation}\label{eq:structure-2-7-eq1}
s_E^{-1}(\cup\Gamma_E)\cap r_E^{-1}(\cup\Gamma_E)\subseteq s_F^{-1}(\cup\Gamma_{F})\cap r_F^{-1}(\cup\Gamma_{F})
\end{equation}
and there exists an injective \starhomo from $C^*(E)$ to $C^*(F)\otimes\K$ such that the image of each ideal $\mathfrak{J}_{\overline{H(S)}}$ is a full corner in $\mathfrak{J}_{H(S)}\otimes\K$ for every hereditary subset $S\subseteq \Gamma_E$. 
\item\label{lem:structure-2-8}
In the setting of part \ref{lem:structure-2-7}, we can get all cyclic components of $F$ to be singletons at the cost of \eqref{eq:structure-2-7-eq1} not necessarily holding anymore and only having a canonical identification of $\Gamma_E$ with $\Gamma_F$. 
\end{enumerate}
\end{lemma}

\begin{proof}

\ref{lem:structure-2-2}: 
Let $H_1$ and $H_2$ be saturated hereditary subsets of $E^0$.  Since $H_1 \cup H_2$ is hereditary, it is enough to show that $H_1 \cup H_2$ is saturated.  Let $x \in E^0$ be a regular vertex such that $r( s^{-1} (x) ) \subseteq H_1 \cup H_2$.  Suppose $x$ is a transitional vertex.  Then by assumption $s^{-1}(x) = \{e\}$.  Therefore, $r( s^{-1}(x) ) = \{ r( e) \} \subseteq H_1$ or $r( s^{-1}(x) ) = \{ r(e) \}  \subseteq H_2$.  Since $H_1$ and $H_2$ are saturated, we have that $x \in H_1$ or $x \in H_2$ which implies that $x \in H_1 \cup H_2$.  Suppose $x \in \gamma$ for some $\gamma \in \Gamma_E$.  Then there exists a path $\mu = \mu_1 \cdots \mu_n$ such that $s( \mu_1 ) = r( \mu_n ) = x$ (we are using the fact that $x$ is a regular vertex).  Since $r( s^{-1} (x) ) \subseteq H_1 \cup H_2$ and $\mu_1 \in s^{-1}(x )$, we have that $r( \mu_1 ) \in H_1 \cup H_2$.  Since $H_1 \cup H_2$ is hereditary, $x = r( \mu_n ) \in H_1 \cup H_2$.  

In both cases, we have shown that $x \in H_1 \cup H_2$.  Therefore, $H_1 \cup H_2$ is saturated.

\ref{lem:structure-2-3}: 
Let $H=\overline{H(\gamma)\setminus\gamma}$. 
Clearly $H$ is saturated and hereditary, and $H\subseteq\overline{H(\gamma)}$. 
We want to show that $H$ is a proper subset of $\overline{H(\gamma)}$ and $\gamma\cap H=\emptyset$. 
So assume first that $\gamma\cap H\neq\emptyset$. 
Then $H\setminus\gamma$ is not saturated. 
Thus there exists a $v\in E_{\mathrm{reg}}^{0}$ such that $r(s^{-1}(v))\subseteq H\setminus \gamma$ and $v\not\in H\setminus\gamma$. 
Since $H\setminus\gamma\subseteq H$ and $H$ is saturated, $v\in H$. 
Thus $v\in\gamma$. 
Since $v\in\gamma\in\Gamma_E$, we have that $v$ supports a cycle within $\gamma$ or $v$ is singular --- both being contradictions. 
Consequently, we have that $\gamma\cap H=\emptyset$. 
Now it is clear that $H$ is a proper subset of $\overline{H(\gamma)}$. 

Now we want to show that $H$ is the largest proper saturated hereditary subset of $\overline{H(\gamma)}$. 
It is enough to show that for all $v\in \overline{H(\gamma)}\setminus H$, we have that $\overline{H(v)}=\overline{H(\gamma)}$. 
So let $v\in \overline{H(\gamma)}\setminus H$ be given. Clearly $\overline{H(v)}\subseteq\overline{H(\gamma)}$. 
If $v\in\gamma$, then $H(v)=H(\gamma)$, so $\overline{H(v)}=\overline{H(\gamma)}$.  Suppose $v \notin \gamma$.  
Since $H(\gamma)\setminus \gamma\subseteq H$, we have that $v\not\in H(\gamma)\setminus \gamma$. 
So now assume that $v\in \overline{H(\gamma)}\setminus H$. 
Note that $\overline{H(\gamma)}\setminus\gamma$ is saturated, and thus $H\subseteq \overline{H(\gamma)}\setminus\gamma$. 
Then $\overline{H(\gamma)}\setminus\{v\}$ cannot be saturated, so $v\in E_\mathrm{reg}^{0}$ and $r(s^{-1}(v))\subseteq \overline{H(\gamma)}\setminus\{v\}$.
By assumption, we must have that $r(s^{-1}(v))\not\subseteq H$.
Using the description of the saturation from \cite[Remark~3.1]{MR1988256}, it follows that there exists a $v_0\in\gamma$ such that $v\geq v_0$. Thus $H(v)\supseteq H(\gamma)$ and $\overline{H(v)}\supseteq\overline{H(\gamma)}$ follows. 

\ref{lem:structure-2-4}: 
It follows from \ref{lem:structure-2-3} that $\gamma\subseteq\overline{H(\gamma)}\setminus \overline{H(\gamma)\setminus \gamma}$ for each $\gamma\in\Gamma_E$. The transition states are the regular vertices not supporting a cycle. Since we only have finitely many vertices, every transition state will have a path to a component (the sinks are also components). Moreover, since every transition state has exactly one outgoing edge, each transition state has a unique shortest path to a component through transition states. If we have a transition state $v$ and the first component every path from $v$ reaches is $\gamma$, then it follows from \cite[Remark~3.1]{MR1988256} that $v\in\overline{H(\gamma)}$. From the proof of \ref{lem:structure-2-3}, we have that $\gamma\cap \overline{H(\gamma)\setminus\gamma}=\emptyset$, so $v\not\in \overline{H(\gamma)\setminus\gamma}$. Thus we have shown that every vertex belongs to at least one of the sets $\overline{H(\gamma)}\setminus \overline{H(\gamma)\setminus \gamma}$, $\gamma\in\Gamma_E$. Let $\gamma,\gamma'\in\Gamma_E$. If $\gamma\geq\gamma'$ and  $\gamma \neq \gamma'$, then $\gamma' \subseteq H(\gamma)$, and therefore $\gamma'\cap(\overline{H(\gamma)}\setminus \overline{H(\gamma)\setminus \gamma})=\emptyset$. If $\gamma\not\geq\gamma'$, then $\overline{H(\gamma)}\setminus \gamma'$ is a saturated set that contains $H(\gamma)$, and, consequently, $\gamma'\cap \overline{H(\gamma)}=\emptyset$. Therefore, the vertices of the components belong to a unique set in the collection $\overline{H(\gamma)}\setminus \overline{H(\gamma)\setminus \gamma}$, $\gamma\in\Gamma_E$. Now let $v$ be a transition state and let $\gamma$ be the first component every path from $v$ reaches. Assume that $v\in\overline{H(\gamma')}\setminus \overline{H(\gamma')\setminus \gamma'}$ for a $\gamma'\in\Gamma_E$ with $\gamma'\neq\gamma$. If $\gamma'\geq\gamma$, then $\gamma\subseteq H(\gamma')\setminus \gamma'$ and therefore $v\in\overline{H(\gamma')\setminus \gamma'}$.  So this is a contradiction since $v\in\overline{H(\gamma')}\setminus \overline{H(\gamma')\setminus \gamma'}$.  If $\gamma'\not\geq\gamma$, then we have seen that $\gamma\cap \overline{H(\gamma')}=\emptyset$ while $v\in \overline{H(\gamma')}$ implies that $\gamma \subseteq \overline{H(\gamma')}$. So this is a contradiction. Thus we have shown that each transition state belongs to a unique set in the collection $\overline{H(\gamma)}\setminus \overline{H(\gamma)\setminus \gamma}$, $\gamma\in\Gamma_E$ --- namely, the first component every path from it reaches.

\ref{lem:structure-2-6}: 
First we show how to modify $E$ to get a graph with the property that if $v$ is an infinite emitter, then $v$ emits infinitely many edges to any vertex it emits any edge to. 
Let $v \in E^0$ be an infinite emitter. 
If there exists a vertex $u \in E^0$ such that $v$ emits only finitely many edges to $u$, we partition $s^{-1}(v)$ into two sets, $\mathcal{E}_1 = \setof{ e \in s^{-1}(v)}{ |s^{-1}(v) \cap r^{-1}(r(e))| < \infty }$ and $\mathcal{E}_2 = \setof{ e \in s^{-1}(v)}{ |s^{-1}(v) \cap r^{-1}(r(e))| = \infty }$, \ie, $\mathcal{E}_1$ consists of the edges out of $v$ that only have finitely many parallel edges. 
Note that since $E^0$ is finite, $\mathcal{E}_1$ is a finite set. 
Hence we can perform Move \OO according to this partition, resulting in a graph $E'$ that is move equivalent to $E$. 
Call the vertices $v$ got split into $v_1$ and $v_2$. 
In $E'$, $v_2$ is an infinite emitter with the property that it emits infinitely many edges to any vertex it emits any edge to, and any infinite emitter in $E$ that already had that property keeps it. 
On the other hand $v_1$ is a finite emitter. 
Since $E^0$ is finite, we can do the above process a finite number of times, ending with a graph $G$ that is move equivalent to $E$, and with the property that if $v$ is an infinite emitter, then $v$ emits infinitely many edges to any vertex it emits any edge to. 

Let $n\in\N$ and let $v \in G^0$ be a transition state of $G$, \ie, a regular vertex that is not the base point of a cycle. 
Assume that $|s^{-1}(v)|\geq 2$, and that the shortest path from $v$ to a component of $G$ is $n$. 
Since $v$ is regular, we can partition $s^{-1}(v)$ into finitely many disjoint singletons $\mathcal{E}_1'$,$\mathcal{E}_2',\ldots,\mathcal{E}_{|s^{-1}(v)|}'$. 
Now we can perform Move \OO according to this partition, resulting in a graph $G'$ that is move equivalent to $G$ such that vertices that $v$ got split into are still transition states but each having exactly one outgoing edge, and the shortest path from each of them to a component is at least $n$.
A vertex in $G'$ is a transition state if and only if it is one of the vertices that $v$ got split into or it is a transition state of $G$. 
All transition states in $G$ that had exactly one outgoing edge and a path to a component of length $n$ or shorter will still have exactly one outgoing edge and a path of length at most $n$. 
Also, every infinite emitter in $G'$ emits infinitely many edges to any vertex it emits any edge to. 
We repeat this for all transition states emitting at least two edges and with the shortest path to a component having length $n$. 
By induction on $n$, we can get a graph $F$ with finitely many vertices such that every infinite emitter emits infinitely many edges to any vertex it emits any edge to, every transition state has exactly one edge going out. 

We got $F$ from $E$ by using Move~\OO a number of times. 
Therefore we clearly have that $E\Meq F$, and it follows from Proposition~\ref{prop:moveOimpliesisomorphism} that $C^*(E)\cong C^*(F)$.

\ref{lem:structure-2-7}: 
Let $F$ be the graph obtained by continuing to collapse all transitional vertices of $E$.  It is clear from the construction of $F$ that $F^0 = \cup \Gamma_E \subseteq E^0$, $\Gamma_E = \Gamma_F$, they carry the same order, and $s_E^{-1} ( \cup \Gamma_E) \cap r_E^{-1}( \cup\Gamma_E ) \subseteq s_F^{-1} ( \cup \Gamma_F) \cap r_F^{-1}( \cup\Gamma_F )$.  Now, there exists an injective \starhomo $\Phi_1 \colon C^* (F) \rightarrow C^* (E)$ such that $\Phi_1 ( C^* (F) ) = P C^* (E) P$, where $P$ is the sum of vertex projections of the vertices from $F$.  Since $\overline{H(F^0)} = E^0$, we have that $\Phi_1 ( C^* (F) )$ is a full corner of $C^*(E)$.  Therefore, $\Phi_1 ( \mathfrak{J}_{ H(S) } ) = P \mathfrak{J}_{ \overline{H(S)} } P$ for every hereditary subset $S \subseteq \Gamma_F$.  By \cite{MR0454645} there exists a partial isometry $v$ in $\mathcal{M}( C^* (E) \otimes \K )$ such that $v^*v = \Phi_1 ( 1_{C^*(F)} ) \otimes 1_{\mathbb{B} ( \ell^2 ) }$ and $vv^* = 1_{ \mathcal{M}( C^* (E) \otimes \K ) }$.  Set $\Phi_2 = \mathrm{Ad} (v) \circ (\Phi_1 \otimes \mathrm{id}_\K)$.  Hence, $\Phi_2 \colon C^* (F) \otimes \K \rightarrow C^* (E) \otimes \K$ is a \stariso such that $\Phi_2 ( \mathfrak{J}_{ H(S) } \otimes \K)$ is a full corner of $\mathfrak{J}_{ \overline{H(S)} } \otimes \K$ for every hereditary subset $S \subseteq \Gamma_F$.  

Set $\Psi =  \Phi_2^{-1} \circ \kappa$, where $\kappa$ is the embedding $C^* (E)$ to $C^* (E) \otimes \K$ given by $a \mapsto a \otimes e_{11}$. Therefore, $\Psi \colon C^* (E) \rightarrow C^* (F) \otimes \K$ is an injective \starhomo such that $\Psi ( \mathfrak{J}_{\overline{H(S)} } )$ is a full corner of $\mathfrak{J}_{H(S)}$ for every hereditary subset $S \subseteq \Gamma_F$.  Since $\Gamma_F = \Gamma_E$, $S$ is hereditary in $\Gamma_F$ if and only if $S$ is hereditary in $\Gamma_E$.  So, $\Psi \colon C^* (E) \rightarrow C^* (F) \otimes \K$ is an injective \starhomo such that $\Psi ( \mathfrak{J}_{\overline{H(S)} } )$ is a full corner of $\mathfrak{J}_{H(S)}$ for every hereditary subset $S \subseteq \Gamma_E$.

\ref{lem:structure-2-8}:  
In addition to the process in \ref{lem:structure-2-7} of collapsing all transitional vertices of $E$, we also collapse all regular vertices of $E$ that are base points of cyclic components (but not of a loop).  Using a similar argument as the proof of \ref{lem:structure-2-7}, we get the desired result.
\end{proof}

\begin{proposition}\label{prop:thetamap}
Let $E$ be a graph with finitely many vertices such that every infinite emitter emits infinitely many edges to any vertex it emits any edge to. 
In Lemma~\ref{lem:structure-d} we have defined a homeomorphism $\wastheta_E$ from $\Gamma_E$ to $\Prime_\gamma(C^*(E))$. 
This homeomorphism induces a lattice isomorphism from the open subsets of $\Gamma_E$ to the open subsets of $\Prime_\gamma(C^*(E))$. We also denote this map $\wastheta_E$. 

Let $\omega_E$ denote the map given by Lemma~\ref{lem:structure-b} and Fact~\ref{fact:structure-1}, \ie, 
$$\omega_E(\eta)=\mathfrak{J}_{\overline{H(\cup\eta)}}$$
for every hereditary subset $\eta$ of $\Gamma_E$, and let $\varepsilon_E$ denote the map from $\mathbb{O}(\Prime_\gamma(C^*(E)))$ to $\mathbb{I}_\gamma(C^*(E))$ given in Lemma~\ref{lem:orderreversingprime}, \ie, 
$$\varepsilon_E(O)=\cap(\Prime_\gamma(C^*(E))\setminus O)$$
for every open subset $O\subseteq\Prime_\gamma(C^*(E))$. 
Then we have a commuting diagram 
$$\xymatrix{\mathbb{O}(\Gamma_E)\ar[d]^{\wastheta_E}_\cong\ar[r]_-\cong^-{\omega_E} & 
\mathbb{I}_\gamma(C^*(E))\ar@{=}[d] \\ 
\mathbb{O}(\Prime_\gamma(C^*(E)))\ar[r]_-\cong^-{\varepsilon_E}
& \mathbb{I}_\gamma(C^*(E))}$$
of lattice isomorphisms. 
\end{proposition}

\begin{proof}
The only new statement in the proposition is the commutativity of the diagram. Note that the inverse of the map $\varepsilon_E$ is also given in Lemma~\ref{lem:orderreversingprime}, and it is $\mathfrak{I}\mapsto\Gamma_E\setminus\operatorname{hull}_\gamma(\mathfrak{I})$. 
Let $\eta\subseteq\Gamma_E$ be a hereditary subset. Then  $$\varepsilon^{-1}_E\circ\omega_E(\eta)=\Prime_\gamma(C^*(E))\setminus\operatorname{hull}_\gamma(\mathfrak{J}_{\overline{H(\cup\eta)}}).$$
From the description of the elements in $\Prime_\gamma(C^*(E))$ in the proof of Lemma~\ref{lem:structure-d}, we see that this set is exactly the set
$$\setof{\mathfrak{J}_{\overline{H(\cup(\Gamma_E\setminus\setof{\gamma\in\Gamma_E}{\gamma\geq\gamma_0}))}}}{\gamma_0\in\Gamma_E, \overline{H(\cup(\Gamma_E\setminus\setof{\gamma\in\Gamma_E}{\gamma\geq\gamma_0}))}\not\supseteq \overline{H(\cup\eta)}}.$$
But this is exactly the image of 
$$\setof{\gamma_0\in\Gamma_E}{\Gamma_E\setminus\setof{\gamma\in\Gamma_E}{\gamma\geq\gamma_0}\not\supseteq \eta}$$
under the homeomorphism $\wastheta_E$. Since $\eta$ is hereditary, this set is exactly $\eta$. 
\end{proof}

\begin{example} We will now take a look at an example of how we
    get the space $\Prime_\gamma(C^{*}(E))$ from the space $\Gamma_E$
    for a graph $E$ with finitely many vertices (where all the
    infinite emitters emit infinitely many edges to any vertex they
    emit any edge to).  Let us say that the ordered set $\Gamma_E$
    consists of four points $\gamma_1$, $\gamma_2$, $\gamma_3$,
    $\gamma_4$ with the relations $\gamma_1\geq\gamma_3$,
    $\gamma_2\geq\gamma_3$, $\gamma_3\geq\gamma_4$ (and thus also
    $\gamma_1,\gamma_2\geq\gamma_4$), while $\gamma_1\not\geq\gamma_2$
    and $\gamma_2\not\geq\gamma_1$. This can be illustrated by the
    component graph as in Figure~\ref{figure:component-graph}.

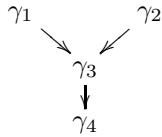
\begin{figure}[h]\begin{center}
$$\xymatrix@=0.3cm{\gamma_1\ar[dr]&&\gamma_2\ar[dl] \\ 
&\gamma_3\ar[d] \\ & \gamma_4}$$
\end{center}
\caption{The component graph $\Gamma_E$}
\label{figure:component-graph}
\end{figure}

For each $\gamma_i$, $i=1,2,3,4$, we consider the hereditary subset $\eta_i=\Gamma_E\setminus\setof{\gamma\in\Gamma_E}{\gamma\geq\gamma_i}$. 
These subsets are illustrated in Figure~\ref{figure:primeexample} by marking the elements of the subset in red. 

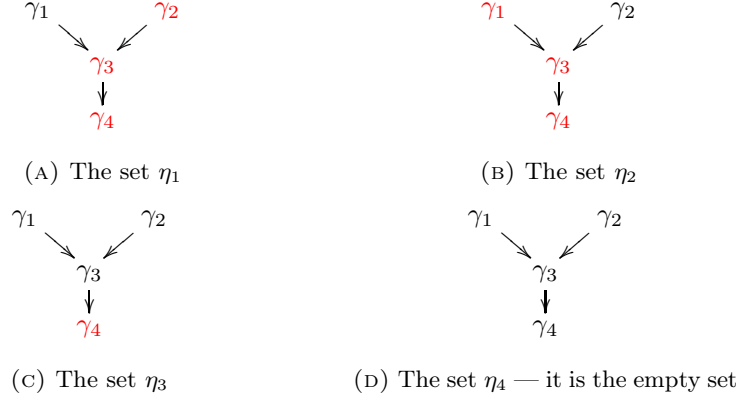
\begin{figure}[H]\quad\begin{subfigure}[b]{0.4\textwidth}
$$\xymatrix@=0.3cm{\gamma_1\ar[dr]&&{\color{red}\gamma_2\ar[dl]} \\ 
&{\color{red}\gamma_3}\ar[d] \\ & {\color{red}\gamma_4}}$$
        \caption{The set $\eta_1$}
        \label{fig:prime-1}
    \end{subfigure}\qquad
    ~           \begin{subfigure}[b]{0.4\textwidth}
$$\xymatrix@=0.3cm{{\color{red}\gamma_1}\ar[dr]&&\gamma_2\ar[dl] \\ 
&{\color{red}\gamma_3\ar[d]} \\ & {\color{red}\gamma_4}}$$
        \caption{The set $\eta_2$}
        \label{fig:prime-2}
    \end{subfigure}\qquad
    ~           \begin{subfigure}[b]{0.4\textwidth}
$$\xymatrix@=0.3cm{\gamma_1\ar[dr]&&\gamma_2\ar[dl] \\ 
&\gamma_3\ar[d] \\ & {\color{red}\gamma_4}}$$
        \caption{The set $\eta_3$}
        \label{fig:prime-3}
    \end{subfigure}\qquad
    ~           \begin{subfigure}[b]{0.4\textwidth}
$$\xymatrix@=0.3cm{\gamma_1\ar[dr]&&\gamma_2\ar[dl] \\ 
&\gamma_3\ar[d] \\ & \gamma_4}$$
        \caption{The set $\eta_4$ --- it is the empty set}
        \label{fig:prime-4}
    \end{subfigure}\quad
\caption{The components marked with red show the elements of $\eta_i$, for $i=1,2,3,4$}
\label{figure:primeexample}
\end{figure}
So the corresponding gauge invariant ideals $\upsilon_E(\gamma_i)=\omega_E(\eta_i)$ of $C^*(E)$ are $\mathfrak{J}_{\overline{H(\cup\eta_i)}}$, \ie,  $\mathfrak{J}_{\overline{H(\gamma_2\cup\gamma_3\cup\gamma_4)}}$, 
$\mathfrak{J}_{\overline{H(\gamma_1\cup\gamma_3\cup\gamma_4)}}$,
$\mathfrak{J}_{\overline{H(\gamma_4)}}$, 
$\mathfrak{J}_{\overline{H(\emptyset)}}=\{0\}$, respectively. 
The topology on $\Prime_\gamma(C^*(E))$ is given by the specialization preorder, so we can illustrate it as in Figure~\ref{figure:prime-space}, where an arrow (or path) from $x$ to $y$ indicates that $x$ is in the closure of $\{y\}$. 
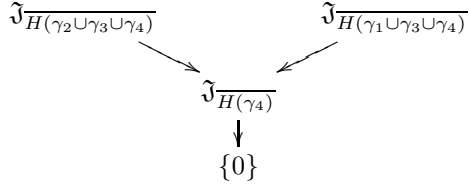
\begin{figure}[H]\begin{center}
$$\xymatrix@=0.4cm{\mathfrak{J}_{\overline{H(\gamma_2\cup\gamma_3\cup\gamma_4)}}
\ar[dr]&&\mathfrak{J}_{\overline{H(\gamma_1\cup\gamma_3\cup\gamma_4)}}\ar[dl] \\ 
&\mathfrak{J}_{\overline{H(\gamma_4)}}\ar[d] \\ & \{0\}}$$
\end{center}
\caption{An illustration of $\Prime_\gamma(C^*(E))$}
\label{figure:prime-space}
\end{figure}
\end{example}

\begin{example}\label{linearcase}
In the case that the ordered set $\Gamma_E$ is linearly ordered
\[
\gamma_1\geq \gamma_2\geq \cdots\geq \gamma_n
\]
(where all the infinite emitters emit infinitely many edges to any vertex they emit any edge to) we get hereditary subsets $\eta_i=\setof{\gamma}{\gamma_i >\gamma}$ and prime gauge invariant ideals
\[
\mathfrak p_i=\upsilon_E(\gamma_i)=\omega_E(\eta_i)=\mathfrak{J}_{\overline{H(\cup{\eta_i})}}=
\begin{cases}\mathfrak{J}_{\overline{H(\gamma_{i+1})}}&i<n\\
\{0\}&i=n.
\end{cases}
\]
Note that the $\mathfrak p_i$ decrease as $i$ increases. We denote the corresponding topological space by $X_n$ and note that it is the Alexandrov space of a linear order on a set of $n$ elements.
\end{example}

\subsection{Reduced filtered \texorpdfstring{$K$}{K}-theory}\label{sec:reducedKtheory}

Let $X$ be a topological space satisfying the $T_0$ separation
property and let \A be a \ca over $X$.  For open subsets $U_{1} ,
U_{2} , U_{3}$ of $X$ with $U_{1} \subseteq U_{2} \subseteq U_{3}$, let $Y_{1} = U_{2} \setminus U_{1}, Y_{2} = U_{3} \setminus U_{1},
Y_{3} = U_{3} \setminus U_{2}\in \mathbb{LC}(X)$.  Then the
diagram
\begin{equation*}
\xymatrix{
K_{0} ( \A ( Y_{1}  ) ) \ar[r]^{ \iota_{*} } & K_{0} ( \A ( Y_{2} ) ) \ar[r]^{ \pi_{*} } & K_{0} ( \A ( Y_{3} ) ) \ar[d]^{\partial_{*}} \\
K_{1} ( \A ( Y_{3}  ) ) \ar[u]^{ \partial_{*}} & K_{1} ( \A ( Y_{2} ) ) \ar[l]^{ \pi_{*} } & K_{1} ( \A ( Y_{1} ) ) \ar[l]^{\iota_{*}}
}
\end{equation*}
is an exact sequence. The collection of all such exact sequences is an invariant of the \cas over $X$ often referred to as the \emph{filtered $K$-theory}. We use here a refined notion:

\begin{definition}
Let $X$ be a finite topological space satisfying the $T_0$ separation property and let \A be a \ca over $X$.  
Note that all singletons of $X$ are locally closed. 

For each $x\in X$, we let $S_x$ denote the smallest open subset that contains $x$, and we let 
$R_x=S_x\setminus\{x\}$, which is an open subset. 
As mentioned above, we get a cyclic six term exact sequence in $K$-theory
\begin{equation}\label{eq:sixtermktheory}\vcenter{
\xymatrix{
K_0(\A(O))\ar[r] & K_0(\A(U))\ar[r] & K_0(\A(U\setminus O))\ar[d] \\
K_1(\A(U\setminus O))\ar[u] & K_1(\A(U))\ar[l] & K_1(\A(O))\ar[l] \\
}}
\end{equation}
whenever we have two open subsets $O\subseteq U\subseteq X$. 
It follows from \cite[Theorem~4.1]{MR2922394} that the map from $K_0$ to $K_1$ is the zero map whenever $\A(O)$ and $\A(U)$ are gauge invariant ideals of a graph \ca. 

Let 
\begin{align*}
I_0(\A)&=\setof{R_x}{x\in X,R_x\neq\emptyset}\cup\setof{S_x}{x\in X}\cup\setof{\{x\}}{x\in X},\\
I_1(\A)&=\setof{\{x\}}{x\in X},
\end{align*}
and let $\operatorname{Imm}(x)$ denote the set 
$$\setof{y\in X}{S_y\subsetneq S_x\wedge \not\exists z\in X\colon S_y\subsetneq S_z\subsetneq S_x}.$$
The \emph{reduced filtered $K$-theory} of \A, $\FKR(X;\A)$, consists of the families of groups 
$(K_0(\A(V)))_{V\in I_0(\A)}$ and 
$(K_1(\A(O)))_{O\in I_1(\A)}$ together with the maps in the sequences
\begin{equation}\label{eq:longtype}
K_1(\A(\{x\}))\to K_0(\A(R_x))\to  K_0(\A(S_x))\to  K_0(\A(\{x\}))
\end{equation}
originating from the sequence \eqref{eq:sixtermktheory}, for all $x\in X$ with $R_x\neq\emptyset$, and 
the maps in the sequences
\begin{equation}\label{eq:shorttype}
K_0(\A(S_y))\to  K_0(\A(R_x))
\end{equation}
originating from the sequence \eqref{eq:sixtermktheory}, for all pairs $(x,y)\in X$ with $y\in\operatorname{Imm}(x)$ and $\operatorname{Imm}(x)\setminus\{y\}\neq\emptyset$.

Let \B be a \ca over $X$. 
A \emph{homomorphism} from $\FKR(X;\A)$ to $\FKR(X;\B)$ consists of 
families of group homomorphisms
$$(\phi_{V,0}\colon K_0(\A(V))\rightarrow K_0(\B(V)))_{V\in I_0(\A)}$$
$$(\phi_{O,1}\colon K_1(\A(O))\rightarrow K_1(\B(O)))_{O\in I_1(\A)}$$ 
such that all the ladders coming from the above sequences commute.  A homomorphism is an \emph{isomorphism} exactly if the group homomorphisms in the family are group isomorphisms.

Analogously, we define the \emph{ordered reduced filtered $K$-theory} of \A, $\FKRplus(X;\A)$, just as $\FKR(X;\A)$ where we also consider the order on all the $K_0$-groups --- and for a homomorphism respectively an isomorphism, we demand that the group homomorphisms respectively the group isomorphisms between the $K_0$-groups are positive homomorphisms respectively order isomorphisms. 
Hereby we get --- in the obvious way --- two functors $\FKR(X;- )$ and $\FKRplus(X;- )$ that are defined on the category of \cas over $X$.
\end{definition}

\begin{remark}\label{howtocompute}
Let $E$ be a graph. 
Then $C^*(E)$ has a canonical structure as a $\Prime_\gamma(C^*(E))$-algebra.
So if $E$ has finitely many vertices --- or, more generally, if $\Prime_\gamma(C^*(E))$ is finite --- then we can consider the reduced filtered $K$-theory, $\FKR(\Prime_\gamma(C^*(E)),C^*(E))$. 
We use the results of \cite{MR2922394} to identify the $K$-groups and the homomorphisms in the cyclic six term sequences using the adjacency matrix of the graph. 
\end{remark}

\begin{remark}
Let \A be an $X$-algebra. 
Since $\mathfrak{I}\mapsto\mathfrak{I}\otimes\K$ is a lattice isomorphism between $\mathbb{I}(\A)$ and $\mathbb{I}(\A\otimes\K)$, the \ca $\A\otimes\K$ is an $X$-algebra in a canonical way, and the embedding $\kappa_\A$ given by $a\mapsto a\otimes e_{11}$ is an $X$-equivariant homomorphism from $\A$ to $\A\otimes\K$. 
Also, it is clear that  $\FKR(X;\kappa_\A)$ is an (order) isomorphism. 
Note also that the invariant $\FKR(X;-)$ has been considered in \cite{MR3177344,MR3349327}.
\end{remark}

\begin{remark}\label{lpacomments}
Appealing to \cite{MR2514392} instead of \cite{MR2023453} one
may define $\Prime_\gamma$ also for Leavitt path algebras over
$\C$, and establish most of the results of this section also in a
purely algebraic setting. Since \cite{MR2514392} discusses only
row-finite graphs and we here  insist that there are only
finitely many vertices, this applies only to finite graphs.
\end{remark}

\section{Specific preliminaries}
\label{sec:notation-for-proof}

In this section we introduce concepts and notation that are required for the remainder of the paper.

\subsection{Block matrices and equivalences}

\begin{notation}
For $m,n\in\N_0$, we let $\MZ[m\times n]$ denote the set of group homomorphisms from $\Z^n$ to $\Z^m$. When $m,n\geq 1$, we can equivalently view this as the $m\times n$ matrices over $\Z$, where composition of group homomorphisms corresponds to matrix multiplication --- the (zero) group homomorphisms for $m=0$ or $n=0$, we will also call empty matrices with zero rows or columns, respectively. 

For $m,n\in\N$, we let \Mplus denote the subset of $\MZ[m\times n]$, where all entries in the corresponding matrix are positive. For an $m\times n$ matrix, we will also write $B>0$ whenever $B\in\Mplus$.

For an $m\times n$ matrix $B$, where $m,n\in\N$, we let $B(i,j)$ denote the $(i,j)$'th entry of the corresponding matrix, \ie, the entry in the $i$'th row and $j$'th column.\end{notation}

\begin{definition} 
Let $m,n\in\N$. 
For an $m\times n$ matrix $B$ over \Z, we let $\gcd B$ be the greatest common divisor of the entries $B(i,j)$, for $i=1,\ldots,m$, $j=1,\ldots,n$, if $B$ is nonzero, and zero otherwise. 
\end{definition}

\begin{assumption} \label{ass:preorder}
Let $N\in\N$. 
For the rest of the paper, we let $\calP=\{1,2,\ldots,N\}$ denote a partially ordered set with order $\preceq$ satisfying
$$i\preceq j\Rightarrow i\leq j,$$
for all $i,j\in\calP$, where $\leq$ denotes the usual order on \N. 
We denote the corresponding irreflexive order by $\prec$.
\end{assumption}

\begin{definition}\label{def:blockmatrices}
Let $\mathbf{m}=(m_i)_{i=1}^{N},\mathbf{n}=(n_i)_{i=1}^{N}\in\N_{0}^N$ be \emph{multiindices}. 
We write $\mathbf{m}\leq\mathbf{n}$ if $m_i\leq n_i$ for all $i=1,2,\ldots,N$, 
and in that case, we let $\mathbf{n}-\mathbf{m}$ be $(n_i-m_i)_{i=1}^N$. 
We also let $\mathbf{m}+\mathbf{n}$ denote $(m_i+n_i)_{i=1}^N$ for any multiindices, and we let $|\mathbf{m}|=m_1+m_2+\cdots+m_N$. We denote the multiindex with $1$ on every entry by $\mathbf 1$. 

We let $\MZ$ denote the set of group homomorphisms from $\Z^{n_1}\oplus\Z^{n_2}\oplus\cdots\oplus\Z^{n_N}$ to $\Z^{m_1}\oplus\Z^{m_2}\oplus\cdots\oplus\Z^{m_N}$, and for such a homomorphism $B$, we let $B\{ i,j\}$ denote the component of $B$ from the $j$'th direct summand to the $i$'th direct summand. 
We also use the notation $B\{i\}$ for $B\{i,i\}$. 
Using composition of homomorphisms, we get in a natural way a category $\mathfrak{M}_N$ with objects $\N_0^N$ and with the morphisms from $\mathbf{n}$ to $\mathbf{m}$ being $\MZ$. 
Moreover, 
$$(BC)\{ i,j\}=\sum_{k=1}^N B \{ i,k\} C\{ k,j\},$$
whenever $B\in\MZ$ and $C\in\MZ[\mathbf{n}\times\mathbf{r}]$ for a multiindex $\mathbf{r}$. 

A morphism $B\in\MZ$ is said to be in $\MPZ$, if 
$$B\{i,j\}\neq 0\Longrightarrow i\preceq j,$$
for all $i,j\in\calP$. 
It is easy to verify, that this gives a subcategory $\mathfrak{M}_\calP$ with the same objects but $\MPZ$ as morphisms. 

Moreover, for a subset $s$ of \calP, we let --- with a slight misuse of notation --- $B\{s\}\in\mathfrak{M}_s((m_i)_{i\in s}\times (n_i)_{i\in s},\Z)$ denote the component of $B$ from $\bigoplus_{i\in s}\Z^{n_i}$ to $\bigoplus_{i\in s}\Z^{m_i}$.

We let $\MZ[\mathbf{n}]$ denote $\MZ[\mathbf{n}\times\mathbf{n}]$, and $\MPZ[\mathbf{n}]$ denote $\MPZ[\mathbf{n}\times\mathbf{n}]$.

For $\mathbf{n}$, we let $\GLPZ$ denote the automorphisms in $\MPZ[\mathbf{n}]$. 
Then $U\in\GLPZ$ if and only if $U\in\MPZ[\mathbf{n}]$ and $U\{ i\}$ is a group automorphism (meaning that the determinant as a matrix is $\pm 1$ whenever $n_i\neq 0$, for every $i\in\calP$). 

An automorphism $U\in\GLPZ$ is in $\SLPZ$ if the determinant of $U\{ i\}$ is $1$ for all $i\in\calP$ with $n_i\neq 0$. 
\end{definition}

\begin{remark}\label{rem:blockmatrices}
Let $\mathbf{m},\mathbf{n}\in\N_{0}^N$ be \emph{multiindices}. 
If $|\mathbf{m}|> 0$ and $|\mathbf{n}|> 0$, we can equivalently view the elements $B\in\MZ$ as block matrices
$$B = 
\begin{pmatrix}
B\{1,1\} & \dots & B\{1,N\} \\
\vdots &  & \vdots \\
B\{N,1\} & \dots & B\{N,N\}
\end{pmatrix}$$
where $B\{i,j\} \in \MZ[m_i\times n_j]$ with $B\{i,j\}$ the empty matrix if $m_{i} = 0$ or $n_{j} = 0$.

Note that from this point of view, the matrices in \MPZ are upper triangular matrices with a certain zero block structure dictated by the order on \calP, and the matrices in \GLPZ (respectively \SLPZ) are matrices in \MPZ with all nonempty diagonal blocks having determinant $\pm 1$ (respectively $1$). 

Note that if $B\in\MZ$ and $C\in\MZ[\mathbf{n}\times\mathbf{r}]$ for a multiindex $\mathbf{r}$, 
then the matrix product makes sense, and --- as matrices --- we have that 
\begin{equation}\label{eq:howtomul}
(BC)\{ i,j\}=\sum_{k\in\calP, n_k\neq 0} B \{ i,k\} C\{ k,j\},
\end{equation}
for all $i,j\in\calP$ with $m_i\neq 0$ and $r_j\neq 0$. 

We will therefore also allow ourselves to talk about matrices with no rows or no columns (by considering it as an element of $\MZ[m\times n]$ with $m = 0$ or $n = 0$); and then $B\{s\}$ for a subset $s$ of \calP as defined above is just the principal submatrix corresponding to indices in $s$ (remembering the block structure). 
\end{remark}

\begin{definition}\label{def:glpandslpeq}
Let $\mathbf{m}$ and $\mathbf{n}$ be multiindices. 
Two matrices $B$ and $B'$ in \MPZ are said to be \emph{\GLPE} (respectively \emph{\SLPE}) if there exist 
$U \in\GLPZ[\mathbf{m}]$ and $V \in\GLPZ$ (respectively $U \in\SLPZ[\mathbf{m}]$ and $V \in\SLPZ$) such that 
$$U B V = B'.$$
\end{definition}

\begin{definition}\label{def:iotar}
Let $\mathbf{r}=(r_i)_{i=1}^{N}\in \N_{0}^N$ be a multiindex. We now want to define a functor $\iota_{\mathbf{r}}$ from $\mathfrak{M}_N$
to $\mathfrak{M}_N$. 
For objects, we let $\iota_{\mathbf{r}}(\mathbf{n})=\mathbf{n}+\mathbf{r}$, for all multiindices $\mathbf{n}\in\N_0^N$. 
We define an embedding $\iota_{\mathbf{r}}$ from \MZ to $\MZ[(\mathbf{m}+\mathbf{r})\times(\mathbf{n}+\mathbf{r})]$, for all multiindices $\mathbf{m}=(m_i)_{i=1}^{N}$, $\mathbf{n}=(n_i)_{i=1}^{N}\in\N_{0}^N$,
as follows. 
The block $\iota_{\mathbf{r}}(B)\{i,j\}$ has $B\{i,j\}$ as upper left corner. 
Outside this corner this block is equal to the zero matrix if $i\neq j$. 
If $i=j$, then the lower right $r_i\times r_i$ corner of this (diagonal) block is the identity matrix and zero elsewhere (outside the upper left and lower right corner). 

It is easy to check that $\iota_{\mathbf{r}}$ gives a faithful functor from $\mathfrak{M}_N$ to $\mathfrak{M}_N$ that also induces a faithful functor from $\mathfrak{M}_\calP$ to $\mathfrak{M}_\calP$. 
\end{definition}

Note that this is a generalization of the definitions in \cite{MR1907894,MR1990568} (in the finite matrix case) to the cases with rectangular diagonal blocks or vacuous blocks. 

\begin{remark}\label{rem:iotar}
We see that \GLPZ and \SLPZ are groups for all multiindices $\mathbf{n}=(n_i)_{i=1}^{N}\in\N_{0}^N$. 
We also see that $\iota_{\mathbf{r}}$ is an injective homomorphisms from \GLPZ to $\GLPZ[\mathbf{n}+\mathbf{r}]$ and from \SLPZ to $\SLPZ[\mathbf{n}+\mathbf{r}]$ preserving the identity, for all multiindices $\mathbf{n}$, $\mathbf{r}\in\N_{0}^N$ (since it is a faithful functor). Moreover, $\iota_{\mathbf{r}'}\circ\iota_{\mathbf{r}}=\iota_{\mathbf{r}+\mathbf{r}'}$, for all multiindices $\mathbf{r}$, $\mathbf{r}'\in\N_{0}^N$, and $\iota_{\mathbf{r}}$ is the identity functor whenever $\mathbf{r}=(0,0,\ldots,0)$.
\end{remark}

\subsection{\texorpdfstring{$K$}{K}-web and induced isomorphisms}\label{UVinduce}
We define the $K$-web, $K(B)$, of a matrix $B \in \MPZ$ and describe how a \GLPEe	 $\ftn{(U,V)}{B}{B'}$ induces an isomorphism $\ftn{ \kappa_{(U,V)} }{ K(B) }{K(B')}$.

For an element $B\in\MZ[m\times n]$ (\ie, a group homomorphism $\ftn{B}{\Z^n}{\Z^m}$), we define as usual $\cok B$ to be the abelian group $\Z^m/B\Z^n$ and $\ker B$ to be the abelian group $\setof{x\in\Z^n}{Bx=0}$. Note that if $m=0$, then $\cok B=\{0\}$ and $\ker B=\Z^n$, and if $n=0$, then $\cok B=\Z^m$ and $\ker B=\{0\}$. 

For $m,n\in\N_0$, $B,B'\in\MZ[m\times n]$, $U\in\GLZ[m]$ and $V\in\GLZ$ with $UBV=B'$, it is now clear that this equivalence induces isomorphisms
$$
\xymatrix{\cok B \ar[rr]_{\xi_{(U,V)}}^{ [x] \mapsto [Ux] } & &\cok B'} \quad \text{and} \quad \xymatrix{\ker B \ar[rr]_{\delta_{(U,V)}}^{ [x] \mapsto [V^{-1}x] } & & \ker B'.}
$$

\begin{lemma}\label{lem: Kweb 2 components}
Consider $\calP=\calP_2 = \{1,2\}$ as a partially ordered set and let $B\in\MPZ$.
Then the following sequence 
\[
\xymatrix@C=40pt{
\cok B\{1\} \ar[r]^-{[v] \mapsto \left[ \begin{smallpmatrix} v \\ 0 \end{smallpmatrix} \right] } & 
\cok B \ar[r]^-{\left[ \begin{smallpmatrix} v \\ w \end{smallpmatrix} \right] \mapsto [w] } & 
\cok B\{2\} \ar[d]^0 \\ 
\ker B\{ 2 \} \ar[u]^-{ v \mapsto [ B\{1,2\}v ] } & 
\ker B \ar[l]^-{ w \mapsfrom \begin{smallpmatrix}v \\ w\end{smallpmatrix} } & 
\ker B\{ 1 \} \ar[l]^-{ \begin{smallpmatrix}v \\ 0\end{smallpmatrix}\mapsfrom v  } }
\]
is exact.

Moreover, if $B$ and $B'$ are elements of \MPZ and $\ftn{ (U,V) }{ B }{ B' }$ is a \GLPEe, then $(U,V)$ induces an isomorphism 
$$(\xi_{(U\{1\},V\{1\})},\xi_{(U,V)},\xi_{(U\{2\},V\{2\})},\delta_{(U\{1\},V\{1\})},\delta_{(U,V)},\delta_{(U\{2\},V\{2\})})$$
of (cyclic six-term) exact sequences. 
\end{lemma}

\begin{proof}
The first part of the lemma follows directly from the snake lemma applied to the diagram
$$\xymatrix{ 
0\ar[r] & \Z^{n_1}\ar[r]\ar[d]^{B\{1\}} & \Z^{n_1}\oplus\Z^{n_2}\ar[r]\ar[d]^{B} & \Z^{n_2}\ar[d]^{B\{2\} }\ar[r] & 0 \\ 
0\ar[r] & \Z^{m_1}\ar[r] & \Z^{m_1}\oplus\Z^{m_2}\ar[r] & \Z^{m_2}\ar[r] & 0
}$$
The second part of the proof is a straightforward verification.
\end{proof}

Completely analogous to \cite{MR1990568}, we make the following definitions.

\begin{definition}
A subset $c$ of $\calP$ is called \emph{convex} if $c$ is nonempty and for all $k \in \calP$, 
$$\text{$\{i,j\} \subseteq c$ and $i \preceq k \preceq j \ \implies \ k \in c$.}$$

A subset $d$ of $\calP$ is called a \emph{difference set} if $d$ is convex and there are convex sets $r$ and $s$ in $\calP$ with $r \subseteq s$ such that $d =s \setminus r$ and 
\[
\text{$i \in r$ and $j \in d \ \implies \ j \npreceq i$.}
\]
Whenever we have such sets $r$, $s$ and $d=s\setminus r$, we get a canonical functor from $\mathfrak{M}_\calP$ to $\mathfrak{M}_{\calP_2}$, where $\calP_2=\{1,2\}$ with the usual order if there exist $i\in r$ and $j\in d$ such that $i\preceq j$, and the trivial order otherwise. 
Thus such sets will also give a canonical (cyclic six-term) exact sequence as above. 
\end{definition}

\begin{definition}
Let $B \in \MPZ$. 
The \emph{(reduced) $K$-web} of $B$, $K(B)$, consists of a family of abelian groups together with families of group homomorphisms between these, as described below. 

For each $i \in \calP$, let $r_i = \setof{j\in\calP}{j \prec i}$ and $s_{i} = \setof{ j\in\calP }{ j \preceq i }$. 
Note that if $r_{i}$ in the above definition is nonempty, then $\{ i \} = s_{i} \setminus r_{i}$ is a difference set.  
We let $\mathrm{Imm}(i)$ denote the set of immediate predecessors of $i$ (we say that $j$ is an \emph{immediate predecessor of $i$} if $j \prec i$ and there is no $k$ such that $j \prec k \prec i$).  

For each $i \in \calP$ with $r_i\neq \emptyset$, we get an exact sequence from Lemma~\ref{lem: Kweb 2 components},
\begin{equation}\label{eq:exact-seq-Kweb}
\ker B\{i\}\rightarrow \cok B\{r_i\}\rightarrow \cok B\{s_i\}\rightarrow \cok B\{i\}.
\end{equation}
Moreover, for every pair $(i, j ) \in \calP \times \calP$ satisfying $j \in \mathrm{Imm}(i)$ and $\mathrm{Imm}(i) \setminus \{ j \} \neq \emptyset$ is $s_{j} \subsetneq r_{i}$; consequently we have a homomorphism 
\begin{equation}\label{eq: Kweb hom}
\cok B\{s_j\} \to \cok B\{r_i\} 
\end{equation}
originating from the exact sequence above (\cf\ Lemma~\ref{lem: Kweb 2 components} used on the division into the sets $r_i$, $s_j$ and $r_i\setminus s_j$). 

Set 
\begin{align*}
I_0^{\calP} &= 
\setof{ r_i }{ i \in \calP \text{ and }r_{i} \neq \emptyset } \cup \setof{ s_i }{ i \in \calP } \cup \setof{ \{ i \} }{ i \in \calP }, \\
I_1^{\calP} &= \setof{i\in\calP}{r_i\neq \emptyset}.
\end{align*}

The \emph{$K$-web of $B$}, denoted by $K(B)$, consists of the families $\left( \cok B\{c\} \right)_{ c \in I_0^{\calP}}$ and $\left( \ker B\{i\} \right)_{ i \in I_1^\calP}$ together with all the homomorphisms from the sequences \eqref{eq:exact-seq-Kweb} and \eqref{eq: Kweb hom}.  Let $B'$ be an element of \MPZ[\mathbf{m}'\times\mathbf{n}'].  By a \emph{$K$-web isomorphism}, $\ftn{\kappa}{ K(B) }{ K(B') }$, we mean families 
$$\left( \ftn{ \kappa_{c,0} }{ \cok B\{c\} }{ \cok B'\{c\} } \right)_{ c \in I_0^\calP}$$ 
and $$\left( \ftn{ \kappa_{i,1} }{ \ker B\{i\} }{ \ker B'\{i\} } \right)_{ i \in I_1^\calP }$$ of isomorphisms satisfying that the ladders coming from the sequences in $K(B)$ and $K(B')$ commute.

By Lemma~\ref{lem: Kweb 2 components}, any \GLPEe $\ftn{ (U,V) }{ B }{ B' }$ induces a $K$-web isomorphism from $B$ to $B'$.  We denote this induced isomorphism by $\kappa_{(U,V)}$. 
\end{definition}

\begin{remark}\label{diffKwebFK} We note the obvious likeness
    between the $K$-web and the reduced filtered
    $K$-theory. There are two fundamental differences: In $K(B)$ we
    never consider orders, and the groups $\ker B\{i\}$ are only
    appearing in $K(B)$ when $\{i\}\not=s_i$, whereas the
    corresponding $K_1$-group always appears in
    $\FKR(C^*(\Esf_{B+I}))$.\end{remark}

\begin{remark}\label{rem:iotarandidentifications}
It is clear that the $K$-webs $K(B)$ and $K(\iota_{\mathbf{r}}(B))$ are canonically isomorphic for all multiindices $\mathbf{m},\mathbf{n},\mathbf{r}\in\N_0^N$ and all $B\in\MPZ$. 

Note also that the $K$-webs $K(B)$ and $K(-B)$ are canonically isomorphic, and that $(U,V)$ is a \GLPEe (respectively \SLPEe) from $B$ to $B'$ if and only if $(U,V)$ is a \GLPEe (respectively \SLPEe) from $-B$ to $-B'$, and they will induce exactly the same $K$-web isomorphisms under the above identification. Note that this identification will change the generators of the cokernels and the kernels. 
In this way, we also get a canonical identification of the $K$-webs $K(B)$ and $K(-\iota_{\mathbf{r}}(-B))$ by embedding a vector by setting it to be zero on the new coordinates. This identification preserves the canonical generators of the cokernels and kernels, which will be of importance when we consider positivity.
\end{remark}

\begin{remark} \label{rem:twominusversusnominusiniotar}
The definitions above are completely analogous to the definitions in \cite{MR1990568}, and are the same in the case $m_{i} = n_{i} \neq 0$ for all $i\in\calP$. 
Note that the last homomorphism in \eqref{eq:exact-seq-Kweb} is really not needed, because commutativity with this map is automatic.

The reason we need to use $K(-\iota_{\mathbf{r}}(-B))$ rather than
$K(\iota_{\mathbf{r}}(B))$ (as in \cite{MR1907894,MR1990568}), is that
we let $B=\Bsf_E^\bullet$, where $\Bsf_E=\Asf_E-I$ rather than
$I-\Asf_E$ (as done in \cite{MR1907894,MR1990568}).  One of the
benefits with this approach is that it is somewhat more convenient to
work with positive matrices instead of negative matrices --- and Boyle
actually does this partly himself in his proof of the factorization
theorem, \cf\ \cite[Section~4]{MR1907894}). The reason that we do not
define $\iota_\mathbf{r}$ as extending by $-1$'s instead of $1$'s is
crucial. This would force us to have one definition of embeddings for
the \GLP and \SLP-matrices used for \GLPEe{s} and \SLPEe{s} and
another for the matrices arriving from the adjacency
matrices. Moreover, such a definition would not give a functor. Both
these problems would be very inconvenient for our work. Thus this is a
matter of choosing either to have the convenience of working with
positive matrices or to not need the two minuses in
$K(-\iota_{\mathbf{r}}(-B))$. We have chosen to use the former
convention.
\end{remark}

\subsection{Block structure for graphs}

\begin{definition}\label{def:circ}
Let $E=(E^0,E^1,r,s)$ be a graph. 
We write $\Bsf_E\in\MPZc$ if 
\begin{itemize}
\item 
\calP satisfies Assumption~\ref{ass:preorder}, and there is an isomorphism $\mytheta$ from \calP to $\Gamma_E$ such that $\mytheta$ and $\mytheta^{-1}$ are order reversing,\item
$E$ has finitely many vertices,
\item
every infinite emitter emits infinitely many edges to any vertex it emits any edge to,
\item
every transition state has exactly one edge going out,
\item
$\Bsf_E$ is an $\mathbf{n}\times\mathbf{n}$ block matrix where the vertices of the $i$'th  block correspond exactly to the set $\overline{H(\mytheta(i))}\setminus\overline{H(\mytheta(i))\setminus \mytheta(i)}$, and
\item
$\Bsf_E^\bullet\in\MPZ$.
\end{itemize}
We write $\Bsf_E\in\MPZcc$ if $\Bsf_E\in\MPZc$ and $E$ does not have any transition states, and we write $\Bsf_E\in\MPZccc$ if $\Bsf_E\in\MPZcc$ and $|\gamma|=1$, for every cyclic component $\gamma\in\Gamma_E$. 

According to Lemma~\ref{lem:structure-2}\ref{lem:structure-2-6}, \ref{lem:structure-2-7} and \ref{lem:structure-2-8}, for every graph $E$ with finitely many vertices, there exist graphs $E'$, $E''$ and $E'''$ such that $\Prime_\gamma(C^*(E))\cong\Prime_\gamma(C^*(E'))\cong\Prime_\gamma(C^*(E''))\cong \Prime_\gamma(C^*(E'''))$ in a canonical way and $\Bsf_{E'}\in\MPZc[\mathbf{m}'\times\mathbf{n}']$, $\Bsf_{E''}\in\MPZcc[\mathbf{m}''\times\mathbf{n}'']$, $\Bsf_{E'''}\in\MPZccc[\mathbf{m}'''\times\mathbf{n}''']$, $C^*(E)\cong C^*(E')$, $C^*(E)\otimes\K \cong C^*(E'')\otimes\K$ and $C^*(E)\otimes\K \cong C^*(E''')\otimes\K$ via equivariant isomorphisms. 

If we have $\Bsf_E\in\MPZc$ and $\Bsf_{E'}\in\MPZc[\mathbf{m}'\times\mathbf{n}']$, then we say that a \starhomo $\Phi$ from $C^*(E)$ to $C^*(E')$ (or from $C^*(E)\otimes \K$ to $C^*(E')\otimes \K$) is \calP-equivariant if $\Phi$ is $\Prime_\gamma(C^*(E))$-equivariant under the canonical identification $\Prime_\gamma(C^*(E))\cong\Gamma_E\cong\calP\cong\Gamma_{E'}\cong \Prime_\gamma(C^*(E'))$ coming from the block structure. 
\end{definition}

Note that the conditions above are not only conditions on the graph --- they are also conditions on the adjacency matrix and how we write it (indexed over $\{1,\ldots,|E^0|\}$). In addition to some assumptions about the graph, we choose a specific order of the vertices and index them over $\{1,\ldots,|E^0|\}$ and we have then implicitly chosen an isomorphism $\Gamma_E\cong\calP$, for some appropriate order on $\calP=\{1,2,\ldots,|\Gamma_E|\}$. 
In general, there might be many different such isomorphisms for the same partially ordered set $\calP$ that work depending on the order chosen of the vertices (if $\calP$ admits a nontrivial automorphism), and it might also be possible to choose an order reversing isomorphism $\Gamma_E\cong\calP'$, where $\calP'$ has a different order than $\calP$. 

\subsection{Reduced filtered \texorpdfstring{$K$}{K}-theory, \texorpdfstring{$K$}{K}-web and \texorpdfstring{\GLPEe}{GLP-equivalence}}
\label{sec:red-filtered-K-theory-K-web-GLP-and-SLP-equivalences}

Let $(\calP, \preceq )$ be a partially ordered set that satisfies Assumption~\ref{ass:preorder}. 
We let $\calP^\mathsf{T}$ denote the set $\calP$ with order defined by $i\preceq^\mathsf{T} j$ in $\calP^\mathsf{T}$ if and only if $N+1-j\preceq N+1-i$ in $\calP$, for $i=1,2,\ldots,N$. 
The partially ordered set $(\calP^\mathsf{T},\preceq^\mathsf{T})$ is really the set $\calP$ equipped with the opposite order, followed by a permutation to ensure that it satisfies 
Assumption~\ref{ass:preorder}.
For every multiindex $\mathbf{m}=(m_1,m_2,\ldots,m_N)$ we let $\mathbf{m}^\mathsf{T}=(m_N,\ldots,m_2,m_1)$ and we let $J_\mathbf{m}$ denote the $|\mathbf{m}|\times|\mathbf{m}|$ permutation matrix that reverses the order.

Now assume that we have a graph $E$ with finitely many vertices such that $\Bsf_E\in\MPZcc[\mathbf{m}_E \times \mathbf{n}_E]$. 
It is easy to see, that $\Bsf_{E}^{\bullet} \in \MPZ[\mathbf{m}_E\times \mathbf{n}_E]$ is equivalent to $J_{\mathbf{n}_E}\left(\Bsf_{E}^\bullet\right)^\mathsf{T}J_{\mathbf{m}_E}\in\mathfrak{M}_{\calP^\mathsf{T}}(\mathbf{n}_E^\mathsf{T}\times\mathbf{m}_E^\mathsf{T},\Z)$.
Now assume that we also have a graph $F$ with finitely many vertices such that $\Bsf_F\in\MPZcc[\mathbf{m}_F \times \mathbf{n}_F]$. 
For notational convenience, we let 
\begin{align*}
\Csf_E&=J_{\mathbf{n}_E}\left(\Bsf_{E}^\bullet\right)^\mathsf{T}J_{\mathbf{m}_E} \\ 
\Csf_F&=J_{\mathbf{n}_F}\left(\Bsf_{F}^\bullet\right)^\mathsf{T}J_{\mathbf{m}_F}.
\end{align*}
With the usual description of the $K$-theory and six term exact sequences for graph \cas (\cf\ \cite{MR2922394}), we see that a reduced filtered $K$-theory isomorphism from $\FKR(\calP;C^*(E))$ to $\FKR(\calP;C^*(F))$ corresponds exactly to a (reduced) $K$-web isomorphism from $K(\Csf_E)$ to 
$K(\Csf_F)$ together with an isomorphism from 
$\ker(\Csf_E\{i\})$ 
to $\ker (\Csf_F\{i\})$ for every $i\in(\calP^{\mathsf{T}})_{\min{}}$, where $\calP_{\min}=\setof{i\in\calP}{ j\preceq i\Rightarrow i=j}$ and $(\calP^\mathsf{T})_{\min}=\setof{i\in\calP}{ j\preceq^\mathsf{T} i\Rightarrow i=j}$. 

Positivity is easy to describe on the gauge simple subquotients. For
components with a vertex supporting at least two distinct return
paths, the positive cone is all of $K_0$ (since the corresponding
subquotient is a Kirchberg algebra in the UCT class).  For components
where each vertex supports exactly one return path, the positive cone
is generated by the class of the projections $p_v$, where $v$ are in
this component (the corresponding subquotient is stably isomorphic to
$C(S^1)$).  If such a cyclic component is a singleton, the ordered
$K_0$-group is $(\Z,\N_0)$ under the canonical identifications.  For
components consisting of a single singular vertex not supporting a
cycle, the positive cone is generated by the class of the projection
$p_v$, where $v$ is the vertex in the component (in this case the
subquotient is stably isomorphic to $\K$).  Under the canonical
identifications, the ordered $K_0$-group is also $(\Z,\N_0)$.  The
description of the $K_0$-groups for gauge nonsimple subquotients (or
just gauge nonsimple ideals), turns out to be more complicated in
general (when the \ca is not purely infinite) --- see
\cite[Theorem~2.2]{MR1962131} for a general description of the
order. As it turns out, only the order of the gauge simple
subquotients will play a role --- as a result of our classification
result, we see that the information stored in the order of the other
groups is redundant (see
  Remark~\ref{onlyorderongs}).

We see that a necessary condition for having an isomorphism between the reduced filtered $K$-theories is that $\mathbf{n}_E-\mathbf{m}_E=\mathbf{n}_F-\mathbf{m}_F$. So assume this holds, and choose $\mathbf{m},\mathbf{n}\in\N_0^N$ such that $\mathbf{m}_E,\mathbf{m}_F\leq\mathbf{m}$ and $\mathbf{n}_E,\mathbf{n}_F\leq\mathbf{n}$, and $\mathbf{n}-\mathbf{m}=\mathbf{n}_E-\mathbf{m}_E=\mathbf{n}_F-\mathbf{m}_F$. 
Let $\mathbf{r}_E=\mathbf{m}-\mathbf{m}_E=\mathbf{n}-\mathbf{n}_E$ and let $\mathbf{r}_F=\mathbf{m}-\mathbf{m}_F=\mathbf{n}-\mathbf{n}_F$. 
Then the $K$-webs of $K(-\iota_{\mathbf{r}_E^{\mathsf T}}(-\Csf_E))$ and $K(-\iota_{\mathbf{r}_F^{\mathsf T}}(-\Csf_F))$ are canonically isomorphic to $K(\Csf_E)$ and $K(\Csf_F)$, respectively, and 
$\ker(\Csf_E\{i\})$ and $\ker(\Csf_F\{i\})$ are canonically isomorphic to $\ker (-\iota_{\mathbf{r}_E^{\mathsf T}}(-\Csf_E)\{i\})$ and $\ker (-\iota_{\mathbf{r}_F^{\mathsf T}}(-\Csf_F)\{i\})$ for every $i\in(\calP^{\mathsf{T}})_{\min{}}$. 
We see that a necessary condition for having a positive isomorphism between the reduced filtered $K$-theories is that under the isomorphisms $\Gamma_E\cong\calP\cong\Gamma_{F}$ we have exactly the same strongly connected components, the same cyclic strongly connected components, the same sinks, and the same infinite emitters not supporting a cycle. 

It is clear that $(U,V)\mapsto((J_{\mathbf{n}}
VJ_{\mathbf{n}})^\mathsf{T},(J_{\mathbf{m}}
UJ_{\mathbf{m}})^\mathsf{T})$ gives a one-to-one correspondence
between \GLPEe{s} (respectively \SLPEe{s}) from
$-\iota_{\mathbf{r}_E}(-\Bsf_E^\bullet)$ to
$-\iota_{\mathbf{r}_F}(-\Bsf_F^\bullet)$ and
$\operatorname{GL}_{\calP^\mathsf{T}}$-equivalences (respectively
$\operatorname{SL}_{\calP^\mathsf{T}}$-equivalences) from
$-\iota_{\mathbf{r}_E^{\mathsf T}}(-\Csf_E)$ to $-\iota_{\mathbf{r}_F^{\mathsf T}}(-\Csf_F)$.

So every \GLPEe $(U,V)$ from $-\iota_{\mathbf{r}_E}(-\Bsf_E^\bullet)$ to $-\iota_{\mathbf{r}_F}(-\Bsf_F^\bullet)$ will determine a reduced filtered $K$-theory isomorphism from $\FKR(\calP;C^*(E))$ to $\FKR(\calP;C^*(F))$. 
We call this isomorphism $\FKR(U,V)$. 
In particular, if $\mathbf{m}=\mathbf{m}_E=\mathbf{m}_F$ and $\mathbf{n}=\mathbf{n}_E=\mathbf{n}_F$, then every \GLPEe $(U,V)$ from $\Bsf_E^\bullet$ to $\Bsf_F^\bullet$ will determine a reduced filtered $K$-theory isomorphism $\FKR(U,V)$ from $\FKR(\calP;C^*(E))$ to $\FKR(\calP;C^*(F))$. 
Note that $V^\mathsf{T}$ induces the isomorphisms between the $K_0$-groups while $(U^\mathsf{T})^{-1}$ induces the isomorphisms between the $K_1$-groups with the standard identification of the $K$-groups. 

Note that the hereditary subsets of vertices --- as usually defined for graphs, when we consider graph \cas{} --- correspond to subsets $S$ of $\calP$ satisfying that $i\preceq j$ implies that $j\in S$ whenever $i\in S$ (\cf\ the order reversing bijection between \calP and $\Gamma_E$ in Definition~\ref{def:circ}). 
This is due to that fact that we generally do not work with the transposed matrix in this paper, since we find it more convenient to work with the non-transposed matrix. 
Since we are identifying \calP with $\Gamma_E$ using an order reversing isomorphism, we will  avoid  using terms as minimal, maximal, less than or greater than. We have already introduced the term (immediate) predecessor for elements of \calP. We will define (immediate) successors in the analogous way. 
But we will use the term that $\gamma_1$ is a predecessor of $\gamma_2$ if and only if $\gamma_2$ is a successor of $\gamma_1$ if and only if $\gamma_1\geq\gamma_2$ and $\gamma_1\neq\gamma_2$. Immediate predecessor and immediate successor in $\Gamma_E$ is defined accordingly. This use of the language also fits better with our usual picture of the component set as a graph: if $\gamma_2$ is a successor of $\gamma_1$ this means that there is a path from component $\gamma_1$ to component $\gamma_2$.

\subsection{Temperatures and standard form}

Let $E$ be a graph with finitely many vertices. 
Then $\Prime_\gamma(C^*(E))$ is finite and the gauge simple subquotients are $C^*(E)(\{\mathfrak{p}\})$ for $\mathfrak{p}\in\Prime_\gamma(C^*(E))$. 
These are either simple AF algebras, simple purely infinite \cas or nonsimple. 
They are stably isomorphic to $C(S^1)$, when they are nonsimple. 

\begin{definition}
Let $E$ be a graph with finitely many vertices. 
Then we define the \emph{temperature} as the map $\fct{\tau_E}{\Prime_\gamma(C^*(E))}{\{-1,0,1\}}$ defined by
$$\tau_E(\mathfrak{p})
=\begin{cases}
-1,&\text{if }C^*(E)(\{\mathfrak{p}\})\text{ is a simple AF algebra},\\
0,&\text{if }C^*(E)(\{\mathfrak{p}\})\text{ is nonsimple},\\
1,&\text{if }C^*(E)(\{\mathfrak{p}\})\text{ is simple and purely infinite},
\end{cases}$$
where $\mathfrak{p}\in\Prime_\gamma(C^*(E))$. 
We call $(\Prime_\gamma(C^*(E)),\tau_E)$ the \emph{tempered (gauge invariant) prime ideal space}. 

Let $E$ and $F$ be graphs with finitely many vertices. 
Then an isomorphism $\fct{\Theta}{(\Prime_\gamma(C^*(E)),\tau_E)}{(\Prime_\gamma(C^*(F)),\tau_F)}$ is a homeomorphism \fctw{\Theta}{\Prime_\gamma(C^*(E))}{\Prime_\gamma(C^*(F))} satisfying that $\tau_F\circ\Theta=\tau_E$. 
We write $(\Prime_\gamma(C^*(E)),\tau_E)\cong(\Prime_\gamma(C^*(F)),\tau_F)$ when such an isomorphism exists.
\end{definition}

We note from the outset that $\FKRplus(\Prime_\gamma(C^*(E));C^*(E))$ contains the temperature.

\begin{lemma}\label{taufromK}
Let $E$ and $F$ be graphs with finitely many vertices, let $X=\Prime_\gamma(C^*(E))$ and assume that there is a homeomorphism \fctw{\Theta}{X}{\Prime_\gamma(C^*(F))}. View $C^*(E)$ and $C^*(F)$ as $X$-algebras in the canonical way and assume that there is an isomorphism from $\FKRplus(X;C^*(E))$ to $\FKRplus(X;C^*(F))$. Then $\tau_F\circ\Theta=\tau_E$, so $(\Prime_\gamma(C^*(E)),\tau_E)\cong (\Prime_\gamma(C^*(F)),\tau_F)$.

\end{lemma}
\begin{proof}
We read off the temperatures from the ordered, reduced filtered $K$-theory as
\begin{eqnarray*}
\tau_E(\mathfrak{p})=-1&\Longleftrightarrow&K_0\not=(K_0)_+\wedge K_1=0,\\
\tau_E(\mathfrak{p})=0&\Longleftrightarrow&K_0\not=(K_0)_+\wedge K_1\not =0, \text{ and}\\
\tau_E(\mathfrak{p})=1&\Longleftrightarrow&K_0=(K_0)_+,
\end{eqnarray*}
where $K_*=K_*(C^*(E)(\{\mathfrak{p}\}))$.
\end{proof}

\begin{remark}\label{countremark}
In the case of graphs with finitely many vertices such that every infinite emitter emits infinitely many edges to any vertex it emits any edge to --- in particular for finite graphs --- we have a canonical homeomorphism \fctw{\wastheta_E}{\Gamma_E}{\Prime_\gamma(C^*(E))} (\cf\ Lemma~\ref{lem:structure-d}). Thus we can in this case equally well consider the space $(\Gamma_E,\tau_E\circ\wastheta_E)$ as the tempered gauge invariant prime ideal space. 
Note that if we for $\gamma\in\Gamma_E$ let 
$$\gamma^1=\setof{e\in E^1}{r(e),s(e)\in\gamma}\subseteq E^1,$$
then $(\tau_E\circ\wastheta_E)(\gamma)=\operatorname{sgn}(|\gamma^1|-|\gamma|)$ if we use the conventions $\operatorname{sgn}(0)=0$ and $\operatorname{sgn}(\infty)=1$.
\end{remark}

It follows that $(\Prime_\gamma(C^*(E)),\tau_E)\cong(\Prime_\gamma(C^*(F)),\tau_F)$ whenever $E\MCeq F$, because in this case $C^*(E)\otimes\K\cong C^*(F)\otimes\K$. It is not hard, but somewhat tedious, to check directly that the allowed moves will not change the signs of the numbers $|\gamma^1|-|\gamma|$ occurring.

\begin{definition}
Let $E$ be a graph.  
We say that $E$ satisfies \emph{Condition (H)} if for any regular vertex $v$ supporting a unique return path, either this path has no exit, or there is a vertex $w\not=v$ which is  singular or supports a unique return path so that there is a path from $v$ to $w$, and so that any path from $v$ to $w$ passes through vertices not supporting two distinct return paths (in particular, $w$ cannot support two distinct return paths).
\end{definition}

Under the assumption that the (finite) graphs satisfy Condition (H), we will prove that every stable isomorphism at the level of graph \cas may be realized by the moves defining $\MCeq$.  Note that among the graphs in Figure \ref{firstexx}, the two in (a) have Condition (H) whereas the remaining four do not. Also note that Condition (H) in a sense interpolates between  Condition (K) and the case when no vertex has more than one return path, and is met in both cases.

\begin{lemma}\label{charKH}
Let $E$ be a graph with finitely many vertices. 
Then the following holds.
\begin{enumerate}[(i)]
\item \label{charKH-1}
$E$ satisfies Condition~(K) if and only if $\tau_E(\mathfrak{p})\neq 0$ for every $\mathfrak{p}\in\Prime_\gamma(C^{*}(E))$.
\item \label{charKH-2}
$E$ has no vertices supporting two distinct return paths if and only if $\tau_E(\mathfrak{p})\leq 0$ for every $\mathfrak{p}\in\Prime_\gamma(C^{*}(E))$.
\item \label{charKH-3}
$E$ satisfies Condition (H) if and only if whenever $\tau_E(\mathfrak{p})=0$ then either $\operatorname{Imm}(\mathfrak{p})=\emptyset$ or there is a $\mathfrak{p}'\in\operatorname{Imm}(\mathfrak{p})$ with $\tau_E(\mathfrak{p}')\leq 0$. 
\end{enumerate}
If every infinite emitter emits infinitely many edges to any vertex it emits any edge to, then \ref{charKH-3} can be replaced by
\begin{enumerate}[(i')]\addtocounter{enumi}{2}
\item\label{charKH-4}
 $E$ satisfies Condition~(H) if and only if whenever $\tau_E(\wastheta_E(\gamma))= 0$  then either $\gamma$ has no successor in $\Gamma_E$, or 
$\gamma$ has an immediate successor $\gamma'\in\Gamma_E$ with $\tau_E(\wastheta_E(\gamma'))\leq 0$.
\end{enumerate}
\end{lemma}
\begin{proof}
We start by assuming that every infinite emitter in $E$ emits infinitely many edges to any vertex it emits any edge to, so that Remark \ref{countremark} applies. In this case,
 \ref{charKH-1} and \ref{charKH-2} are obvious, and  \ref{charKH-3} and  \ref{charKH-4} are equivalent. For \ref{charKH-4}, assume first that the condition on $\tau_E$ holds. To show (H), let $v$ support a unique return path with an exit and note that $v$ then lies in some $\gamma$ with $\tau_E(\wastheta_E(\gamma))=0$ where $\gamma$ has a successor in $\Gamma_E$. One such successor $\gamma'$ must be immediate with $\tau_E(\wastheta_E(\gamma'))\leq 0$, and we take $w\in\gamma'$. Then any path from $v$ to $w$ passes through only transitional vertices and vertices in $\gamma\cup\gamma'$, neither of which  supports multiple return paths. Finally, if $w$ is regular, then since it is not a transitional vertex, it must support a unique return path.

In the other direction, assume that the condition on $\tau_E$ fails
and choose $\gamma\in\Gamma_E$ with the property that  $\tau_E(\wastheta_E(\gamma))=0$, $\gamma$ has successors and that all such immediate successors $\gamma'$ have
$\tau_E(\wastheta_E(\gamma'))=1$. We conclude that any path from $v\in\gamma$ to any $w$ not
a transitional vertex must pass through a vertex supporting at least
two different return paths. It remains to check that $v$ cannot be a singular vertex.
But since $v\in\gamma$ and $\tau_E(\wastheta_E(\gamma))=0$, $v$ supports a unique return path. Thus $v$ emits finitely many edges to a vertex in $\gamma$; hence $v$ is regular.

For general graphs with finitely many vertices, we note that by Lemma \ref{lem:structure-2}\ref{lem:structure-2-6} (and its proof), 
we may replace $E$ by $E'$ with the property that every infinite emitter in $E'$ emits infinitely many edges to any vertex it emits any edge to, in the sense that $C^*(E)\cong C^*(E')$ and $E'$ is obtained from $E$ by a number of moves of type \OO. Since these operations preserve all the conditions on the graphs, the result follows.
\end{proof}

According to \cite{MR2069031}, the
conditions in \ref{charKH-1} above translate exactly to $C^*(E)$ being of real
rank zero. According to \cite{MR2001940}, the conditions in \ref{charKH-2} above
translate exactly to $C^*(E)$ being a type I/postliminal \ca.

\begin{notation}
Let $E$ be a graph with finitely many vertices and assume that $\Bsf_E\in\MPZc$. 
This induces a temperature $\mytau=\tau_E\circ\wastheta_E\circ\mytheta$ on $\calP$.
\end{notation}

It will be extremely convenient for us to know that the adjacency matrices for two graphs are aligned with all components having the same number of vertices. For this, we define:

\begin{definition}\label{standardform}
Let $E$ and $F$ be finite graphs. We say that $(\Bsf_E,\Bsf_F)$ is in \emph{standard form} if $\Bsf_E,\Bsf_F\in\MPZccc$ for some multiindices $\mathbf{m}$ and $\mathbf{n}$ and $\mathcal{T}_{\Bsf_E}=\mathcal{T}_{\Bsf_F}$. This means that the adjacency matrices have exactly the same sizes and block structures, and that the temperatures of the components match up. 
\end{definition}

\begin{definition}\label{def: positive matrices}
Define \MPplusZ to be the set of all $B \in \MPZ$ satisfying the following:
\begin{enumerate}[(i)]
\item \label{def: positive matrices-1}
If $i \prec j$ and $B \{ i, j \}$ is not the empty matrix, then $B \{ i , j \} > 0$.

\item \label{def: positive matrices-2}
If $m_{i} = 0$, then $n_{i} =1$.

\item \label{def: positive matrices-3}
If $m_{i} = 1$, then $n_{i} = 1$ and $B \{ i \} = 0$.

\item \label{def: positive matrices-4}
If $m_{i} > 1$, then $B \{ i \} > 0$, $n_{i},m_{i} \geq 3$, and the Smith normal form of $B\{i\}$ has at least two 1's (and thus the rank of $B \{ i \}$ is at least 2).  \end{enumerate}
\end{definition}

\begin{lemma}\label{taugivesstd}
Let $E$ and $F$ be two finite graphs. The following are equivalent
\begin{enumerate}[(1)]
\item	$(\Prime_\gamma(C^*(E)),\tau_E)\cong(\Prime_\gamma(C^*(F)),\tau_F)$.\label{taugivesstdI}
\item We can choose finite graphs $E'$ and $F'$ so that  $(\Bsf_{E'},\Bsf_{F'})$ is in standard form and so that   $E\Meq E'$ and $F\Meq F'$.\label{taugivesstdII}
 \end{enumerate}
In \ref{taugivesstdII}, we may further assume that $\Bsf_{E'}^\bullet,\Bsf_{F'}^\bullet\in\MPplusZ$.
 \end{lemma}
\begin{proof}
When $\Bsf_E\in\MPZccc$, we may read off the temperatures of $i\in\calP$ by the rules
\begin{eqnarray*}
\mytau(i)=-1&\Longleftrightarrow&m_i=0,\\
\mytau(i)=0&\Longleftrightarrow&m_i=1\text{ and }\Bsf_E\{i\}=0, \text{ and}\\
\mytau(i)=1&\Longleftrightarrow&\text{either }m_i=1\text{ and }\Bsf_E\{i\}>0\text{ or }m_i>1.
\end{eqnarray*}
Thus, \ref{taugivesstdII}$\Longrightarrow$\ref{taugivesstdI} follows from the move invariance of the temperature.

For the other direction, assume \ref{taugivesstdI}. It follows from Lemma~\ref{lem:structure-2} that we can assume that  $\Bsf_E\in\MPZccc[\mathbf n'\times\mathbf m']$ and
$\Bsf_F\in\MPZccc[\mathbf n''\times\mathbf m'']$ for appropriate
$\calP$, $\mathbf m'$, $\mathbf m''$, $\mathbf n'$ and $\mathbf
n''$ with $\wastheta_{F}\circ\mytheta[F]\circ\mytheta^{-1}\circ\wastheta_{E}^{-1}$ being the isomorphism given in  \ref{taugivesstdI} that intertwines the temperatures. By assumption,
$n_i'=n_i''=m_i'=m_i''=1$ when $\mytau(i)=\mytau[F](i)=0$, and  $n_i'=n_i''=1$,
$m_i'=m_i''=0$ when $\mytau(i)=\mytau[F](i)=-1$.

When $\mytau(i)=\mytau[F](i)=1$, we may perform Move \CO inside each of these components until we get  vertices $u_i^E$ and $u_i^F$ which support loops. 
Since the components are not cyclic, $u_i^E$ and $u_i^F$ emit at least one other edge than the loop to a vertex in the component, and we may perform \RR in reverse on them successively to increase the sizes of the block to arrive at $n_i'=m_i'=n_i''=m_i''\geq 3$. By doing this (at most) twice more we can ensure that the Smith normal form has at least two ones. After this process $u_i^E$ and $u_i^F$ still support a loop.

We will now show that we may get $\Bsf_E^\bullet\in\MPplusZ$. First we will arrange that all entries are positive in such diagonal blocks. We already have that $u_i^E$ supports a loop, and hence 
Proposition~\ref{prop:matrix-moves}\ref{prop:matrix-moves:I} applies to ensure that any vertex in the component which has an edge to $u_i^E$ also supports a loop. Continuing this way, we get that every vertex supports a loop, and we can use Proposition~\ref{prop:matrix-moves}\ref{prop:matrix-moves:II} to ensure that $u_i$ supports two loops. With this, it is easy to arrange that $\Bsf_E\{i\}>0$. Arguing similarly, we may also arrange that  $\Bsf_E\{i,j\}>0$ and $\Bsf_E\{k,i\}>0$
 for any $i\prec j$ or $k\prec i$. 
 
We continue this process for all $i\in\calP$ satisfying $\mytau(i)=\mytau[F](i)=1$.
 
 Any block $\Bsf_E\{j,k\}$ with $j\prec k$ which is not positive after this process must have $\mytau(j)=0$ and $\mytau(k)\leq 0$ and hence will be a $1\times 1$-matrix. Further, $k$ is not an immediate successor of $j$, so we have $j\prec i\prec k$ for some $i$ an immediate successor of $j$. Then  $\Bsf_E\{j,i\}>0$, and we may use Proposition~\ref{prop:matrix-moves}\ref{prop:matrix-moves:I} again to arrange that  $\Bsf_E\{j,k\}>0$. 
 
We argue similarly for $F$.
\end{proof}
 
 Note that in general there may be several (but finitely many) ways of  choosing $\calP$ and the isomorphisms from $\calP$ to $\Gamma_E$ and $\Gamma_F$ which give the standard forms.

\begin{remark}
Let $E$ be a finite graph. As is well known, we may efficiently describe a partially ordered set such as $\Gamma_E$ by the Hasse
diagram with vertices $\{1,\dots,N\}$ connecting $\gamma$ to $\gamma'$ when
$\gamma'$ is an immediate successor of $\gamma$. Thinking of $\tau_E\circ\wastheta_E$ as providing a coloring of the
vertices of the Hasse diagram thus gives an easy way of visualising
the situation. Noting that the color $-1$ can only occur at the vertices with no successors, we see that the smallest 
cases of (isomorphism classes of) colored Hasse diagrams \emph{not} meeting Condition (H) are the cases 
\begin{equation}\label{twop}
\twop{1}
\end{equation}
 when $|\Gamma_E|=2$ and 
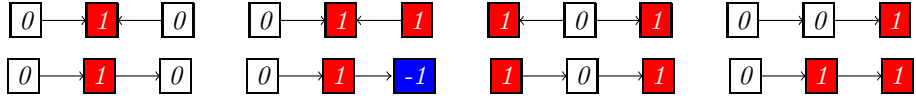
\begin{center}
\threepin{2}\quad \threepin{3}\quad \threepout{5}\quad\threeplin{1}\\
 \threeplin{2}\quad\threeplinmo{2}\quad \threeplin{5}\quad \threeplin{3}
\end{center}
when $|\Gamma_E|=3$ (along with the three cases obtained by adding an unconnected vertex to the one in \eqref{twop}).
\end{remark}

 \section{Classifying move equivalence}\label{CC}

In this section we inspect one of the key results from \cite{MR2270572} to conclude that it holds even for those graphs which are finite with no sinks or
sources, essentially corresponding to the case of Cuntz-Krieger
algebras for matrices not necessarily satisfying the Condition (II)
introduced by Cuntz.

As in  \cite{MR2270572}, we will appeal to the theory of flow equivalence of shifts of finite type; since we work with graph \cas instead of Cuntz-Krieger algebras, we use the edge shifts, defined from a finite graph $E$ as 
\[
{\mathsf X}_E=\{(e_n)\in (E^1)^\Z\mid \forall n:r(e_n)=s(e_{n+1})\}.
\]
This will suffice for our purposes since we may
remove sources via the notion of canonical form, and may replace sinks
by loops as discussed below.

The formal starting point is the following lemma. We must allow for ${\mathsf X}_E=\emptyset$ in the case that no vertex of $E$ supports a return path, and will say that two such empty shift spaces are mutually flow equivalent, and not flow equivalent to any nonempty shift space.

\begin{lemma}\label{flowvsME}
Let $E$ and $F$ be finite graphs. When $E\Meq F$, then ${\mathsf X}_E$ is flow equivalent to 
${\mathsf X}_F$. If neither $E$ nor $F$ have any sinks, the two conditions are equivalent.
\end{lemma}
\begin{proof}
Move \SSS does not affect the shift spaces, and the remaining moves
are precisely the ones allowed in \cite{MR0405385}. Any sink of $E$ or
$F$ will not affect the shift space, so it is not possible to infer in
the opposite direction in general, but if there are no sinks, we may
use Move \SSS to remove all sources and to remove the vertices that become sources (this process will terminate since $E$ and $F$ are finite graphs) to replace $E$ and $F$
with $E'\Meq E$ and $F'\Meq F$ so that neither $E'$ nor $F'$ have
sources. We have ${\mathsf X}_E={\mathsf X}_{E'}$ and ${\mathsf
  X}_F={\mathsf X}_{F'}$, so also ${\mathsf X}_{E'}$ and ${\mathsf
  X}_{F'}$ are flow equivalent. By \cite{MR0405385}, this flow
equivalence is induced by a finite number of the moves  \OO, \II and \RR.
\end{proof}

\subsection{Plugging  sinks}\label{plugging}

We now introduce a way to pass between the case where the
finite graph $E$ has no sinks and the case where the finite graph has
so many sinks that every cycle in $E$ has an exit. The first case is
preferable in the context of symbolic dynamics, whereas the second
case, as we shall see below in Section \ref{unplugging} is preferable
in the operator algebraic context, since it can be used to establish a
certain uniqueness theorem.

We start with the notion of \emph{plugging} sinks. Whenever a graph $E$ is given, $E_\curlywedge$ denotes the graph where a loop has been added to all sinks.

\begin{lemma}\label{CEpassestoplugged}
Let $E$ and $F$ be graphs with finitely many vertices. If $E\MCeq F$, then also $E_\curlywedge\MCeq F_\curlywedge$.
 If $E\Meq F$, then also $E_\curlywedge\Meq F_\curlywedge$.
\end{lemma}
\begin{proof}
Considering plugging of sinks as a move, one checks that it commutes
with all of the moves defining $\MCeq$ and $\Meq$. This is obvious in the case of
\II\ and \CC\ which can never involve a sink. In the case of \OO, \SSS,
and \RR, one sees the claim by noting that sinks are involved only as
receivers of edges in such moves.
\end{proof}

\begin{lemma}\label{Ktheoryplug}
Let $E$ and $F$ be graphs with finitely many vertices and assume that 
\fctw{\Theta}{\Prime_\gamma(C^*(E))}{\Prime_\gamma(C^*(F))}
is a   homeomorphism. Let $X=\Prime_\gamma(C^*(E))$. Since $\Prime_\gamma(C^*(E))$ and $\Prime_\gamma(C^*(F))$ are canonically homeomorphic to $\Prime_\gamma(C^*(E_\curlywedge))$ and $\Prime_\gamma(C^*(F_\curlywedge))$, respectively, we may view $C^*(E)$, $C^*(F)$, $C^*(E_\curlywedge)$, and $C^*(F_\curlywedge)$ as $X$-algebras in the canonical way.  Then the following are equivalent. 
\begin{enumerate}[(1)]
\item $\FKRplus(X;C^*(E))$ and $\FKRplus(X;C^*(F))$ are isomorphic
\item $\FKRplus(X;C^*(E_\curlywedge))$ and $\FKRplus(X;C^*(F_\curlywedge))$ are isomorphic, and $\tau_E=\tau_F\circ\Theta$.
\end{enumerate}
\end{lemma}
\begin{proof}
We note that the changes of $E$ and $F$ only affect the
$K_1$-groups. Since we are only recording the $K_1$-groups at sets
$\{x\}$ and the plugging takes place only at components which have no
successors, in fact no sequences \eqref{eq:longtype} or \eqref{eq:shorttype}
are affected. Thus all that happens is that some of the independent $K_1$-groups that
were originally $0$ are changed to $\Z$, and thus the given isomorphism of the $K$-theories of the original graph \cas readily extend to the plugged versions.  In the other direction, the temperature assumption is to ensure that the number of sinks and the number of cyclic components of $E$ and $F$ are equal.  Thus, an isomorphism of the $K$-theories of the plugged versions restricts to an isomorphism of the $K$-theories of the original graphs.  
\end{proof}

Note finally that when $E$ and $F$ are finite graphs with  $(\Bsf_E,\Bsf_F)$ in standard form, so is $(\Bsf_{E_\curlywedge},\Bsf_{F_\curlywedge})$.
In this situation, $\Bsf^\bullet_E$ is $\SLP$- or $\GLP$-equivalent to $\Bsf^\bullet_F$ precisely when the relation holds between $\Bsf_{E_\curlywedge}$ and $\Bsf_{F_\curlywedge}$. Indeed, if $U\Bsf^\bullet_E V=\Bsf^\bullet_F$ with $U\in\GLPZ[\mathbf m]$ and $V\in\GLPZ$, we get  $\widetilde{U}\in\GLPZ$ so that $\widetilde{U}\Bsf_{E_\curlywedge} V=\Bsf_{F_\curlywedge}$ by padding $U$ with rows and columns from the appropriately sized identity matrix where a plugging has taken place. Conversely, if $\widetilde U$ is given, $U$ is obtained by deleting the relevant rows and columns.
Since $U\in \SLPZ[\mathbf m]$ precisely when $\widetilde{U}\in \SLPZ$, our claim concerning $\SLP$-equivalence is justified.

 Further, when $\Bsf^\bullet_E,\Bsf^\bullet_F\in\MPplusZ$, we conclude that  $\Bsf^\bullet_{E_\curlywedge},\Bsf^\bullet_{F_\curlywedge}\in\MPplusZ[\mathbf n]$. We will use these observations repeatedly without mention below.

\subsection{Move equivalence versus $\SLP$- and $\GLP$-equivalence}

The results in this section are the key to everything that follows and all depend on the following proposition, which was  proved in \cite[Lemma~6.7 and Theorem~6.8]{MR2270572} under the added assumption that the graphs had Condition (K) and no sinks (\ie, were Cuntz-Krieger algebras with Condition (II)). But since we are working only at components which are neither single cycles nor sinks, the same proof applies.

\begin{proposition}\label{GunnarRULES}
Let $E$ and $F$ be finite graphs and assume that $(\Bsf_E,\Bsf_F)$ is in standard form with $\Bsf^\bullet_E,\Bsf^\bullet_F\in\MPplusZ$. If  
$U\in \GLPZ[\mathbf m]$ and $V\in \GLPZ$ are given with 
\[
U\Bsf^\bullet_{E}V=\Bsf^\bullet_{F}
\]
and $i\in\calP$ is given with $\mytau(i)=1$, then there exist $\mathbf r$, graphs $E'$ and $F'$ with $(\Bsf_{E'},\Bsf_{F'})$ in standard form with $\Bsf_{E'}^\bullet,\Bsf_{F'}^\bullet\in\MPplusZ[(\mathbf{m}+\mathbf{r})\times(\mathbf{n}+\mathbf{r})]$,  $U'\in \GLPZ[\mathbf m+\mathbf{r}]$ and $V'\in \GLPZ[\mathbf{n}+\mathbf{r}]$ 
so that
\begin{enumerate}[(i)]
\item \label{GunnarRULES-1}
$\mathbf r=(r_j)$ with $r_i\leq 3$ and $r_j=0$ for $j\not=i$,
\item \label{GunnarRULES-2}
$E\MCeq E'$, $F\MCeq F'$, 
\item \label{GunnarRULES-3}
$U\{j\}=U'\{j\}$, $V\{j\}=V'\{j\}$ for $j\not=i$,
\item \label{GunnarRULES-4}
$\det U'\{i\}=\det V'\{i\}=1$,
\end{enumerate}
and
\[
U'\Bsf^\bullet_{E'}V'=\Bsf^\bullet_{F'}.
\]
\end{proposition}
\begin{proofsk}
The key idea is to note that whenever $U_0BV_0=B'$, then with
\[
\widetilde{U_0}=
\begin{pmatrix}
U_0&\\
&\left(\begin{smallmatrix}-1\end{smallmatrix}\right)
\end{pmatrix}\qquad \widetilde{V_0}=
\begin{pmatrix}
V_0&\\
&\left(\begin{smallmatrix}-1\end{smallmatrix}\right)
\end{pmatrix}
\]
we have
\[
\widetilde{U_0}\begin{pmatrix}
B&\\
&\left(\begin{smallmatrix}-1\end{smallmatrix}\right)
\end{pmatrix}
\widetilde{V_0}
=
\begin{pmatrix}
B'&\\
&\left(\begin{smallmatrix}-1\end{smallmatrix}\right)
\end{pmatrix}
\]
and $\det\widetilde{U_0}=-\det U_0$ and $\det\widetilde{V_0}=-\det V_0$, and with 
\[
\overline{U_0}=
\begin{pmatrix}
U_0&\\
&\left(\begin{smallmatrix}0&1\\1&0
\end{smallmatrix}\right)\end{pmatrix}\qquad \overline{V_0}=
\begin{pmatrix}
V_0&\\
&\left(\begin{smallmatrix}-1&0\\0&-1
\end{smallmatrix}\right)
\end{pmatrix}
\]
we have
\[
\overline{U_0}\begin{pmatrix}
B&\\
&\left(\begin{smallmatrix}0&1\\1&0
\end{smallmatrix}\right)\end{pmatrix}\overline{V_0}
=
\begin{pmatrix}
B'&\\
&\left(\begin{smallmatrix}-1&0\\0&-1
\end{smallmatrix}\right)\end{pmatrix}
\]
and $\det\overline{U_0}=-\det U_0$ and $\det\overline{V_0}=\det V_0$. Thus we can adjust the signs of $U'\{i\}$ and $V'\{i\}$ as required in \ref{GunnarRULES-4}  at the cost of adding one of the matrices
\[
\begin{pmatrix}-1
\end{pmatrix},
\begin{pmatrix}-1&0\\0&-1
\end{pmatrix},
\begin{pmatrix}-1&0&0\\0&-1&0\\0&0&-1
\end{pmatrix},
\begin{pmatrix}0&1\\1&0
\end{pmatrix},
\begin{pmatrix}0&1&0\\1&0&0\\0&0&-1
\end{pmatrix}
\]
in the appropriate diagonal of the $\Bsf$-matrix, and the proposition is proved as soon as we have established that whenever, say, $E$ as in the statement is given, we can find $E'\MCeq E$ with 1,2, or 3 vertices more than $E$ in the component $\mytheta(i)$ so that $\Bsf^\bullet_{E'}$ is $\SLP$-equivalent to the relevant augmentation of $\Bsf^\bullet_E$. Note that we require that $\Bsf^\bullet_{E'}$ has positive entries wherever it can be nonzero, and that we must take care not to alter the diagonal blocks of $U$ and $V$ away from component $\mytheta(i)$.

To add a single $-1$, we perform Move \RR in reverse on one of the loops at the last vertex of $\mytheta(i)$ to get $\widetilde{E}$ with $\Bsf_{\widetilde{E}}\{i\}$ in the form
\[
\begin{pmatrix}
\Bsf_{{E}}\{i\}&\left(\begin{smallmatrix}0\\\vdots\\0\\1\end{smallmatrix}\right)\\
\left(\begin{smallmatrix}0&\cdots&0&1\end{smallmatrix}\right)&\left(\begin{smallmatrix}-1\end{smallmatrix}\right)
\end{pmatrix}
\]
This matrix is clearly $\operatorname{SL}$-equivalent to the desired one, and it is straightforward to obtain $E'$ which has only positive entries in $\Bsf_{{E'}}\{i\}$ by a number of row or column additions. We can also arrange to have positive entries in added rows and columns in each offdiagonal block $\{i,j\}$ or $\{k,i\}$ where $i\prec j$ or $k\prec i$. By Proposition \ref{prop:matrix-moves}, $E\Meq E'$. The $\SLP$-matrices implementing the necessary row or column additions will equal the identity at every diagonal block $\{j\}$ with $i\not=j$, so we will not change these blocks as required in \ref{GunnarRULES-3}.

Repeating this process, we can arrange move equivalences taking us from $E$ to $E'$ with $\Bsf^\bullet_{E'}\in\MPplusZ[(\mathbf{m}+k\mathbf{e}_i)\times(\mathbf{n}+k\mathbf{e}_i)]$ being \SLP-equivalent to $-\iota_{k\mathbf e_i}(-\Bsf_E)$ for any $k\in\N$, where $\mathbf e_i$ is the vector that is $1$ at index $i$ and 0 otherwise.	Thus all that remains is to note that if we perform Move \CC on the  last vertex of $\mytheta(i)$ to get $\overline{E}$ with $\Bsf_{\overline{E}}\{i\}$ in the form
\[
\begin{pmatrix}
\Bsf_{{E}}\{i\}&\left(\begin{smallmatrix}0&0\\\vdots&\vdots\\0&0\\1&0\end{smallmatrix}\right)\\
\left(\begin{smallmatrix}0&\cdots&0&1\\0&\cdots&0&0\end{smallmatrix}\right)&\left(\begin{smallmatrix}0&1\\1&0\end{smallmatrix}\right)
\end{pmatrix}
\]
we again obtain the desired matrix augmentation up to $\SLP$-equivalence, and may arrange for positive entries just as above.
\end{proofsk}

\begin{proposition}\label{fromBwithplug}
Let $E$ and $F$ be finite graphs and assume that $(\Bsf_E,\Bsf_F)$ is in standard form with $\Bsf^\bullet_E,\Bsf^\bullet_F\in\MPplusZ$.
Assume further that  $U\in \GLPZ[\mathbf m]$ and $V\in\GLPZ$ are given with 
\[
U\Bsf_{E}^\bullet V=\Bsf_{F}^\bullet.
\]
Then
\begin{enumerate}[(i)]
\item \label{fromBwithplug-1}
If $U\in \SLPZ[\mathbf m]$ and $V\in\SLPZ$, then  $E\Meq F$.
\item \label{fromBwithplug-2}
If $V\{i\}=1$ whenever $\mytau(i)\leq 0$ and $U\{i\}=1$ whenever $\mytau(i)=0$, then  $E\MCeq F$.
\end{enumerate}
\end{proposition}
\begin{proof}
To prove \ref{fromBwithplug-1}, we pass to the plugged graphs and recall that $\Bsf_{F_\curlywedge}$ and $\Bsf_{E_\curlywedge}$ are also $\SLP$-equivalent. Since $E_\curlywedge$ and $F_\curlywedge$ have neither sinks nor
sources, we may appeal to \cite[Theorem 4.4]{MR1907894} which shows that
$\Bsf_{F_\curlywedge}$ can be obtained from $\Bsf_{E_\curlywedge}$
by a number of elementary row or column additions or subtractions,
never leaving matrices in $\MPplusZ$. In fact, Boyle produces a list of elementary equivalences $E_{u,v}$ as described in Proposition \ref{prop:matrix-moves} where there is a path from $u$ to $v$ throughout, and since we have arranged that any vertex in any graph along the way supports at least one loop, the proposition applies to yield \ref{fromBwithplug-1} when $E$ and $F$ have no
sinks.

When $E$ and $F$ do have sinks, we apply the same sequence of row and column operations to $\Bsf_E$. We
note that in any matrix addition or subtraction implemented by $E_{u,v}$, $u$ will not be one of the plugged sinks, as indeed these provide paths only to themselves. Hence $u$ will not be a sink in the original setup. In the case the matrix implements a column operation, the requirements in 
Proposition \ref{prop:matrix-moves} are still met, and thus such an operation remains implemented by moves in the original setup.  In the case the matrix implements a row operation, we observe that it has no effect, adding or subtracting a zero row from another row. It may hence be omitted, proving \ref{fromBwithplug-1}.

We prove \ref{fromBwithplug-2} by reducing to \ref{fromBwithplug-1} by Proposition~\ref{GunnarRULES}, changing any negative determinants of the given $U$ and $V$ at blocks $\{i\}$ starting from $\{1\}$ and working downwards. We then get finite graphs $E'$ and $F'$ such that $(\Bsf_{E'},\Bsf_{F'})$ is in standard form with
$\Bsf_{E'}^\bullet,\Bsf_{F'}^\bullet\in \MPplusZ[(\mathbf m+\mathbf r)\times (\mathbf n+\mathbf r)]$  where the multiindex $\mathbf{r}$ has the property that $r_j=0$ for $j$ with 
 $\mytau(j)\leq 0$ and $r_j\leq 3$ otherwise. We have that
$E\MCeq E'$, $F\MCeq F'$ and that for some  $U'\in\SLPZ[\mathbf m+\mathbf r]$ and $V'\in\SLPZ[\mathbf n+\mathbf r]$, we may arrange that
 $U'\Bsf_{E'}^\bullet V'=\Bsf_{F'}$. By \ref{fromBwithplug-1}, $E'\Meq F'$.
 \end{proof}

\begin{definition}
Let $E$ and $E'$ be graphs with finitely many vertices and assume that $\Bsf_E\in\MPZc$ and $\Bsf_{E'}\in\MPZc[\mathbf{m}'\times\mathbf{n}']$. 
We say that a \stariso from $C^*(E)$ to $C^*(E')$ (or from $C^*(E)\otimes\K$ to $C^*(E')\otimes\K$) respects the block structure, if the induced homeomorphism from $\Prime_\gamma(C^*(E))$ to $\Prime_\gamma(C^*(E'))$ commutes with the identification of $\calP$ with $\Prime_\gamma(C^*(E))$ and $\Prime_\gamma(C^*(E'))$, respectively. 

All of the elementary moves introduced in Section~\ref{sec:moves} induces a canonical stable isomorphism. We say that such an elementary move preserves the block structure if this induced stable isomorphism respects the block structure. We say that a move equivalence or a Cuntz move equivalence respects the block structure, if it is the composition of a series of elementary moves such that the composition of the induced stable isomorphisms respects the block structure. 

If the only automorphism of \calP is the trivial automorphism, then \starisos, move equivalences and Cuntz move equivalences, respectively, automatically respect the block structure --- we will in particular use that this is the case when \calP is linearly ordered.
\end{definition}

\begin{proposition}\label{toBwithplug} 
Let $E$ and $F$ be finite graphs and assume that $(\Bsf_E,\Bsf_F)$ is in standard form and has the additional property that   $\gcd(\Bsf_E\{i\})=1$ and $\gcd(\Bsf_F\{i\})=1$ at every $i$ with $\mytau(i)=\mytau[F](i)=1$. When $E_\curlywedge\MCeq F_\curlywedge$ respecting the block structure,
there exist  $U,V\in \GLPZ[\mathbf n]$  with $U\{i\}=V\{i\}=1$ whenever $\mytau(i)\leq 0$ so that $
U\Bsf_{E_\curlywedge}V=\Bsf_{F_\curlywedge}
$. When $E_\curlywedge\Meq F_\curlywedge$,
we may choose  $U,V\in \SLPZ[\mathbf n]$.
\end{proposition}
\begin{proof}
We may assume without loss of generality that $E \MCeq F$ respecting the block structure and  $E$ and $F$ have no sinks.  Since $E \MCeq F$, we have a string of moves as follows: 
\[
\xymatrix{
E \ar[r]^-{\Meq} & E_1 \ar[r]^-{\CC} & E_2 \ar[l]\ar[r]^-{\Meq} & E_3 \ar[r]^-{\CC} &E_4\ar[l]\ar[r]& \cdots \ar[r]^-{\CC} & E_{2n}\ar[l] \ar[r]^-{\Meq} & F
}
\]  
where each move between $E_{2j-1}$ and $E_{2j}$ is either a Cuntz splice or its inverse.
Note that at each stage of the move equivalence, we may have
introduced transitional vertices and we may have increased the number
of vertices in the cyclic components.  So, we collapse these
transitional vertices and the cyclic components, to obtain a graph
$F_i$ with no transitional vertices such that $E_i \Meq F_i$ and $F_i
\in \MPZccc[ \mathbf{n}_i ]$.  Note that performing these moves
commutes with any Cuntz splice, since such a move cannot take place at
a cyclic component or at a transitional vertex.  Hence, we have a
commuting diagram
\[
\xymatrix{
E \ar[r]^-{\Meq} & E_1 \ar[r]^-{\CC} \ar[d]^\Meq & E_2\ar[l] \ar[r]^-{\Meq} \ar[d]^{\Meq} & E_3 \ar[r]^-{\CC} \ar[d]^{\Meq} &E_4\ar[l]\ar[r]\ar[d]^{\Meq}& \cdots \ar[r]^-{\CC} & E_{2n}\ar[l] \ar[d]^{\Meq} \ar[r]^-{\Meq} & F \\
& F_1 \ar[r]^-{\CC} & F_2\ar[l]  & F_3 \ar[r]^-{\CC}  &F_4\ar[l]& \ar[r]^-{\CC} &  F_{2n}\ar[l]  & },
\]  
where the compositions of the move equivalences all respect the block structure. 
Let $F_0=E$ and $F_{2n+1}=F$. 

Let $k\in\{0,1,\ldots,n\}$ be given. 
Since $F_{2k} \Meq F_{2k+1}$, we have that the shift spaces $\mathsf{X}_{F_{2k}}$ and $\mathsf{X}_{F_{2k+1}}$ are flow equivalent.  By \cite[Theorem~3.1 and Theorem~3.4]{MR1907894}, there exists an \SLPEe \fctw{(U_{2k}, V_{2k})}{-\iota_{\mathbf{r}_{2k}}(-\Bsf_{F_{2k}})}{-\iota_{\mathbf{r}_{2k}'}(-\Bsf_{F_{2k+1}})}, where $\mathbf{r}_{2k}=(r_{2k,l})_{l\in\calP}$ and $\mathbf{r}_{2k}'=(r_{2k,l}')_{l\in\calP}$ with $r_{2k,l}=r_{2k,l}'=0$ whenever $\mytau(l) \leq 0$. 

Let again $k\in\{0,1,\ldots,n\}$ be given. 
A computation based on Restorff's proof of Proposition~\ref{GunnarRULES} shows that there exists a \GLPEe \fctw{(U_{2k+1} , V_{2k+1} )}{-\iota_{\mathbf{r}_{2k+1}}( -\Bsf_{F_{2k+1} })}{-\iota_{\mathbf{r}_{2k+1}'}(-\Bsf_{F_{2k}} )} such that $U_{2k+1}\{i\} = V_{2k+1}\{i \} = 1$ for all $\mytau[F_{2j+1}](i) \leq 0$, where $\mathbf{r}_{2k+1}=(r_{2k+1,l})_{l\in\calP}$ and $\mathbf{r}_{2k+1}'=(r_{2k+1,l}')_{l\in\calP}$ with $r_{2k+1,l}=r_{2k+1,l}'=0$ whenever $\mytau(l) \leq 0$. 

By \cite[Theorem~3.10]{MR1990568},
the composition of these \SLP- and \GLP-equivalences induces a
$K$-web isomorphism $\kappa$.  Hence, by
\cite[Theorem~4.5]{MR1990568}, there exists a \GLPEe $(U,V) \colon
\Bsf_{E} \rightarrow \Bsf_{F}$ inducing $\kappa$ as in Lemma \ref{lem:  Kweb 2 components}.  Since the cyclic
components of $E$ and $F$ are $1\times 1$ blocks and $(U,V)$ induces
$\kappa$, we have that $U\{i\}=V\{i\}=1$  whenever $\mathcal{T}_{\Bsf_E}(i)\leq 0$.

We now prove the statement about move equivalence. As above, we may assume that $E$ and $F$ have no sinks and we get an \SLPEe \fctw{(U,V)}{-\iota_{\mathbf{r}}(-\Bsf_{E})}{-\iota_{\mathbf{r}}(-\Bsf_{F})}, where $\mathbf{r}=(r_l)_{l\in\calP}$ with $r_l=0$ whenever $\mathcal{T}_{\Bsf_E}(l)\leq 0$. Now it follows from \cite[Proposition~4.1 and Corollary~4.9]{MR1990568} that there exists an \SLPEe from $\Bsf_E$ to $\Bsf_F$. 
\end{proof}

\begin{theorem}\label{iffcharofCeq} 
Let $E$ and $F$ be finite graphs with $(\Bsf_E,\Bsf_F)$ in standard form with $\Bsf^\bullet_E,\Bsf^\bullet_F\in\MPplusZ$. Then the following are equivalent
\begin{enumerate}[(1)]
\item\label{iffcharofCeqI} $E\MCeq F$ respecting the block structure,
\item\label{iffcharofCeqNEW} $E_\curlywedge\MCeq F_\curlywedge$ respecting the block structure,
\item\label{iffcharofCeqII}  There exist $U,V\in \GLPZ$ with $U\{i\}=V\{i\}=1$ whenever $\mytau(i)\leq 0$  so that
$U\Bsf_{E_\curlywedge}V=\Bsf_{F_\curlywedge}$,
\item\label{iffcharofCeqIII}  There exist $U\in \GLPZ[\mathbf m]$ and $V\in \GLPZ$ with $V\{i\}=1$ whenever $\mytau(i)\leq 0$ and with $U\{i\}=1$ whenever $\mytau(i)=0$ so that 
 $
U\Bsf_{E}^\bullet V=\Bsf_{F}^\bullet
$.
\end{enumerate}
\end{theorem}
\begin{proof}
Lemma \ref{CEpassestoplugged} proves that \ref{iffcharofCeqI}$\Longrightarrow$\ref{iffcharofCeqNEW}.
Since the $\gcd$ is $1$ at any block with a $1$ in the Smith form, we may apply Proposition \ref{toBwithplug} to prove \ref{iffcharofCeqNEW}$\Longrightarrow$\ref{iffcharofCeqII}. We have noted that \ref{iffcharofCeqII}$\Longleftrightarrow$\ref{iffcharofCeqIII} holds in general, and \ref{iffcharofCeqIII}$\Longrightarrow$\ref{iffcharofCeqI} is the content of Proposition \ref{fromBwithplug}\ref{fromBwithplug-2}.
\end{proof}

\begin{theorem}\label{iffcharofMeq} 
Let $E$ and $F$ be finite graphs with $(\Bsf_E,\Bsf_F)$ in standard form with $\Bsf^\bullet_E,\Bsf^\bullet_F\in\MPplusZ$. Then the following are equivalent
\begin{enumerate}[(1)]
\item\label{iffcharofMeqI} $E\Meq F$ respecting the block structure,
\item\label{iffcharofMeqNEW} $E_\curlywedge\Meq F_\curlywedge$ respecting the block structure,
\item\label{iffcharofMeqII}  There exist $U,V\in \SLPZ$ so that $U\Bsf_{E_\curlywedge}V=\Bsf_{F_\curlywedge}$,
\item\label{iffcharofMeq|||}  There exist $U\in \SLPZ[\mathbf m]$ and $V\in \SLPZ$ so that
 $U\Bsf_{E}^\bullet V=\Bsf_{F}$.
\end{enumerate}
\end{theorem}
\begin{proof}
The proof is completely analogous to the proof of Theorem~\ref{iffcharofCeq}, where we use Proposition \ref{fromBwithplug}\ref{fromBwithplug-1} in the place of Proposition \ref{fromBwithplug}\ref{fromBwithplug-2}.
\end{proof}

We warn the reader that the implication \ref{iffcharofCeqNEW}$\Longrightarrow$  \ref{iffcharofCeqI} in both results above are only true when the temperatures of $E$ and $F$ match up, as implicitly arranged by the condition of standard form.

\begin{example}\label{notalwayssamepre}
The pair of graphs $E$ and $F$ given in Figure \ref{firstexx}(b) are not Cuntz move equivalent. 
\end{example}
\begin{proof}
We see that the vertices $E$ and $F$ may be ordered with $(\Bsf_E,\Bsf_F)$  in standard form with
$\Bsf_E,\Bsf_F\in\MPZccc[\mathbf{1}]$ for $\calP=\{1,2,3\}$ ordered
linearly and with $\gcd$ of the blocks at $\{2\}$ equal to $1$. Appealing to Proposition \ref{toBwithplug}, we see that it
suffices to check, which is obviously true, that there is no solution to
\begin{equation}\label{nosol}
\begin{pmatrix}1&x&y\\0&s&z\\0&0&1
\end{pmatrix}
\begin{pmatrix}0&1&2\\0&1&1\\0&0&0
\end{pmatrix}
\begin{pmatrix}1&x'&y'\\0&s'&z'\\0&0&1
\end{pmatrix}
=
\begin{pmatrix}0&1&0\\0&1&1\\0&0&0
\end{pmatrix}
\end{equation}
with $s,s'\in\{-1,1\}$ and $x,x',y,y',z,z' \in \Z$.
\end{proof}

\section{Classifying \cas}\label{Ccas}

\subsection{A classification result}

\begin{theorem}\label{mainthm}
Let $E$ and $F$ be finite graphs and consider the statements 
\begin{enumerate}[(1)]
\item $E\MCeq F$,\label{MOV}
\item $C^*(E)\otimes\K\cong C^*(F)\otimes\K$,\label{MOR}
\item There exists a homeomorphism  \fctw{\Theta}{X=\Prime_\gamma(C^*(E))}{\Prime_\gamma(C^*(F))} so that when $C^*(E)$ and $C^*(F)$ are considered as $X$-algebras in the canonical way, then $\FKRplus(X;C^*(E))\cong \FKRplus(X;C^*(F))$.\label{FK}
\end{enumerate}
Then
\[
\text{\ref{MOV}}\Longrightarrow \text{\ref{MOR}}\Longrightarrow \text{\ref{FK}}
\]
and when $E$ and $F$ satisfy 
 Condition~(H), all
statements \ref{MOV}--\ref{FK} are equivalent. 
 \end{theorem}
\begin{proof}
The invariance of moves required to prove
\ref{MOV}$\Longrightarrow$\ref{MOR} was established in \cite{MR3082546} and
\cite{arXiv:1602.03709v2}, \cf\ Theorems~\ref{thm:moveimpliesstableisomorphism} and~\ref{thm:cuntz-splice-implies-stable-isomorphism}.

For
\ref{MOR}$\Longrightarrow$\ref{FK} one needs only note, as we did in Lemma~\ref{lem:structure-1}, that any
isomorphism between $C^*(E)\otimes\K$ and $C^*(F)\otimes\K$ must
preserve the gauge invariant ideals even if the isomorphism is not
gauge invariant.

To prove that \ref{FK}$\Longrightarrow$\ref{MOV} under the additional
assumption of Condition (H), we first note that by Lemma
\ref{taufromK}, the tempered gauge prime ideals agree, and hence by
Lemma~\ref{taugivesstd} we may assume that $(\Bsf_E,\Bsf_F)$ is in standard
form, where we may even assume that
$\Bsf_E,\Bsf_F\in\MPplusZ$. Plugging sinks we get
$(\Bsf_{E_\curlywedge},\Bsf_{F_\curlywedge})$ which is also in standard form, having
isomorphic ordered reduced filtered $K$-theories by Lemma \ref{Ktheoryplug}. The
$K$-webs then also agree, and \cite{MR1990568} applies to provide $U,V\in\GLPZ$ with
$U\Bsf_{E_\curlywedge}V=\Bsf_{F_\curlywedge}$. Thus we only need to
arrange that $U$ and $V$ satisfy the conditions in Theorem~\ref{iffcharofCeq}\ref{iffcharofCeqII} to reach the desired conclusion.

In fact, since $V\{i\}$ implements an order isomorphism from
$(\Z,\N_0)$ to $(\Z,\N_0)$ at every $i$ with $\mytau(i)\leq 0$, it must
already be in the desired form. It is straightforward to check
that whenever $U\{i\}=-1$ at some $i$ with no successors, then since
both $\Bsf_{E_\curlywedge}$ and $\Bsf_{F_\curlywedge}$ have zero rows at $i$, the corresponding row of $U$ can be multiplied by $-1$ without affecting the relation that 
$U\Bsf_{E_\curlywedge}V=\Bsf_{F_\curlywedge}$. 

We claim that in the presence of Condition (H), the remaining blocks
$U\{i\}$ at $i$ with $\mytau(i)\leq 0$ must be of the desired
form. Indeed, choosing an immediate successor $j$ of $i$ with
$\mytau(j)\leq 0$ we assume for contradiction that $U\{i\}=-1$. Note
that $\Bsf_{E_\curlywedge}\{i,j\}=x$ and
$\Bsf_{F_\curlywedge}\{i,j\}=y$ with $x,y>0$ since there must be a
path between the two components, and such a path cannot pass through
any other component. Similarly, we get from the immediate successor
condition that for any $B,B'\in\MPZccc$ and any $k\not\in\{i,j\}$,
either $B\{i,k\}=0$ or $B'\{k,j\}=0$, so that
$(BB')\{i,j\}=B\{i\}B'\{i,j\}+B\{i,j\}B'\{j\}$ (\cf\ \eqref{eq:howtomul}). From this we infer that
$(U\Bsf_{E_\curlywedge})\{i,j\}=-x$ and
$(\Bsf_{F_\curlywedge}V^{-1})\{i,j\}=y$, a contradiction.
\end{proof}

\begin{corollary}\label{maincor}
Let $E$ and $F$ be finite graphs so that $C^*(E)$ and $C^*(F)$ are either of real rank zero or type I/postliminal. Then the
 statements \ref{MOV}--\ref{FK} of Theorem~\ref{mainthm} are equivalent.
  \end{corollary}

\begin{remark}\label{onlyorderongs}
Inspection of our proof shows that only the order on $K_0(C^*(E)(\{i\}))$  and $K_0(C^*(F)(\{i\}))$ is necessary to conclude that the \cas are stably isomorphic.
It is possible to define a full (ordered) filtered K-theory (see \cite{MR3177344,MR3349327}). Isomorphism of this invariant clearly implies isomorphism of the reduced invariant (both in the case with and without order). As a consequence of the results in \cite{MR1990568}, the opposite holds without order for the cases considered in this paper. From the results in Theorem \ref{mainthm}, it follows that it holds also in the case with order. Thus the full invariant contains the same information about equivalence classes and (stable) isomorphism classes as the reduced one. 
\end{remark}

\subsection{Unplugging sinks}\label{unplugging}

\renewcommand{\efvs}[1]{{#1}^0_{\mathrm{iso}}\backslash {#1}^0_{\mathrm{sing}}}
\newcommand{\efiso}[1]{{#1}^0_{\mathrm{iso}}}

For a graph $E$, let $E^0_\mathrm{iso}$ be the set of vertices of $E$ that are either sinks or on a 
vertex-simple cycle with no exits (the notation, \cf\ \cite{arXiv:1410.2308v1}, refers to the fact that such vertices give rise to isolated points in the associated path spaces).
Assume that $E$ is a graph with finitely many vertices with $\Bsf_E\in\MPZccc$. Then 
 every vertex $v \in \efvs{E}$
supports a unique loop $e_v$.  Let $E_\curlyvee$ be the graph obtained from $E$ by removing the edges
$e_v$ for all $v \in \efvs{E}$.   We note that in general
\[
(E_\curlywedge)_\curlyvee\not=
(E_\curlyvee)_\curlywedge\not=E.
\]

\begin{proposition}\label{prop: unplugging}
Let $E$ and $F$ be graphs with finitely many vertices so that $\Bsf_E\in\MPZccc[\mathbf{m}_E\times\mathbf{m}_E]$ and $\Bsf_F\in\MPZccc[\mathbf{m}_F\times\mathbf{m}_F]$.
  If there exists a \stariso $\Phi \colon C^* (E_\curlyvee
) \otimes \K \rightarrow C^* (F_\curlyvee ) \otimes \K$ such that
$\mytau[F]\circ\Phi_\sharp=\mytau$, then $C^* ( E ) \otimes \K \cong C^*
( F ) \otimes \K$.
\end{proposition}

\begin{proof}
Note first that whenever $v\in (E_\curlyvee)^0_\mathrm{iso}$ is given, $v$ is a sink, so $\{v\}$ is a saturated and hereditary set. Thus it defines an ideal $\J_{v}$ which is minimal in $C^*(E_\curlyvee)$ and Morita equivalent to $\C$. In fact, any such ideal has this form, and since the same is true for $F_\curlyvee$, we conclude that $\Phi(\J_{v})=\J_w$ for some $w\in (F_\curlyvee)^0_\mathrm{iso}$. Since $\mytau[F]\circ\Phi_\sharp=\mytau$, $w$ will be a sink of $F$ precisely when $v$ is a sink of $E$, and thus a bijection \fct{w}{E^0_\mathrm{iso}}{F^0_\mathrm{iso}} is defined with $w(\efvs{E})=\efvs{F}$.

For any graph $G$, let $SG$ be the stabilized graph, \ie, for each
vertex $v \in G^0$, we put an infinite head at $v$, \cf\ \cite[Definition~9.4]{MR2775826}.  Note that
 $\efiso{(SG)} = \efiso{G}$ with $\efvs{(SG)} = \efvs{G}$, so that  $w$ may also be considered as a map from  $(SE)^0_\mathrm{iso}$ to $(SF)^0_\mathrm{iso}$.  Moreover, $v
\in \efiso{G}$ supports a loop if and only if $v \in \efiso{(SG)}$
supports a loop.  By the proof of \cite[Proposition~9.3 and
  Theorem~9.8]{MR2775826}, there exists a \stariso $\chi_G
\colon  C^* (SG) \rightarrow C^* (G) \otimes \K$ such that $\chi_G (
p_v ) = p_v \otimes e_{11}$ for all $v \in G^0$.  Define \fct{\Psi}{C^* ( SE_\curlyvee )}{ C^* (
SF_\curlyvee )} by $\Psi = \chi_{F_\curlyvee}^{-1} \circ \Phi \circ
\chi_{E_\curlyvee}$.

Note that $ \J_{ v } \cong \K$ in $C^*(SE_\curlyvee)$ and $ \J_{w } \cong \K$ in $C^*(SF_\curlyvee)$ for all $v \in E_{\mathrm{iso}}^{0}$ and for all
$w \in F_{ \mathrm{iso} }^{0}$.  Therefore, any generator of $K_{0} (
\J_{v} )_{+}$ is Murray-von Neumann equivalent to $p_{v}$ in $C^*(SE_\curlyvee)$ for all
$v \in E_{\mathrm{iso}}^{0}$ and any generator of $K_{0} ( \J_{w}
)_{+}$ is Murray-von Neumann equivalent to $p_{w}$ in $C^*(SF_\curlyvee)$ for all $w \in F_{
  \mathrm{iso} }^{0}$.  Consequently,
   $\Psi ( p_v ) \sim p_{w(v)}$ in $C^*( SF_\curlyvee)$, so there exists $W_v \in C^*( SF_\curlyvee)$ such that $W_v^* W_v =\Psi( p_v)$ and $W_v W_v^* = p_{w(v)}$.  Set $p = \sum_{ v \in (SE)^0_{\mathrm{iso}} } \Psi ( p_v )$ and $q =  \sum_{ v \in (SE)^0_{\mathrm{iso }}} p_{w(v)}$.  Since $C^*(SF_\curlyvee)$ is a stable $C^*$-algebra, by  \cite[Corollary 1.10]{MR869419}, 
   \[
   1_{ M( C^*(SF_\curlyvee ) ) } - p \sim 1_{ M( C^*(SF_\curlyvee ) ) }  \sim 1_{ M( C^*(SF_\curlyvee ) ) } - q.
   \] 
  Thus, there exists $W \in M( C^*(SF_\curlyvee ))$ such that $W^* W = 1_{ M( C^*(SF_\curlyvee ) } - p$ and $WW^* = 1_{ M( C^*(SF_\curlyvee ) ) } - q$.  Set $u = W + \sum_{ v \in (SE)^0_{\mathrm{iso}} } W_v$.  A computation shows that $u$ is a unitary in $M( C^*( SF_\curlyvee))$ such that $u \Psi( p_v ) u^* = p_{w(v)}$ for all $v \in SE^0_{\mathrm{iso}}$.
   So, without loss of generality, we may
assume that $\Psi ( p_v ) = p_{w(v) }$.

Note that $SE_\curlyvee$ and $SF_\curlyvee$ satisfy Condition~(L) since we have removed all cycles with no exits.  Using the universal property and the Cuntz-Krieger Uniqueness Theorem, there are injective \starhomos $\lambda_E \colon C^* ( SE_\curlyvee ) \rightarrow C^* (SE)$ and $\lambda_F \colon C^* ( SF_\curlyvee ) \rightarrow C^* (SF)$ such that $\lambda_E ( s_e ) = s_e$, $\lambda_E ( p_v ) = p_v$ for all $e \in (SE_\curlyvee)^1 \subseteq (SE)^1$ and for all $v \in (SE_\curlyvee)^0 = (SE)^0$ and $\lambda_F ( s_f ) = s_f$, and $\lambda_F ( p_w ) = p_w$ for all $f \in (SF_\curlyvee)^1 \subseteq (SF)^1$ and for all $w \in (SF_\curlyvee)^0 = (SF)^0$.  So, using these embeddings, we may assume that $C^* (SE_\curlyvee)$ is a sub-algebra of $C^* (SE)$ and $C^* (SF_\curlyvee)$ is a sub-algebra of $C^* (SF)$.

We now define a Cuntz-Krieger $SE$-family in $C^* ( SF)$.  Set $P_v = \Psi ( p_v )$ for all $v \in (SE)^0 = (SE_\curlyvee)^0$ and 
\[
S_e = 
\begin{cases}
\Psi ( s_e )  &\text{if $e \in (SE_\curlyvee)^1$} \\
s_{ e_{w(v)} } &\text{if $e = e_v$ for some $v \in \efvs{(SE)}$}.
\end{cases}
\]
The only nonobvious Cuntz-Krieger relation is at $v \in \efvs{(SE)}$.  But this is also clear since $P_v = \Psi ( p_v ) = p_{w(v)} =  s_{ e_{w(v) }  }s_{ e_{w(v) } }^*= S_{e_v} S_{e_v}^*$.  Therefore, there exists a \starhomo $\fct{\Xi}{C^* (SE)}{ C^* ( SF )}$.  Since the only vertex-simple cycles in $SE$ with no exits are $e_v$ for all $v \in \efvs{SE}$ and $\Xi ( s_ { e_v } ) = s_ { w(v) }$ has full spectrum, by the General Cuntz-Krieger Uniqueness Theorem in \cite{MR1914564}, we have that $\Xi$ is injective.  Note that $\Xi ( C^* ( SE_\curlyvee) ) = \Psi ( C^* (SE_\curlyvee) ) = C^* ( SF_\curlyvee )$.  Let $e \in (SF)^1$ such that $e$ is not an element of $(SF_\curlyvee)^1$.  Then $e= e_{w}$ for some $w \in \efvs{(SF_\curlyvee)}$.  Therefore, there exists $v \in \efvs{(SE_\curlyvee)}$ such that $w(v) = w$.  Hence, $\Xi ( s_{ e_v } ) = s_{ e_{w(v) } } = s_e$, so $\Xi$ is surjective, and thus a \stariso.
\end{proof}

\subsection{Examples}

In this section we let $E$ and $F$ denote the two graphs  given in Figure \ref{firstexx}(b). We note that
Example \ref{linearcase} applies (with $n=3$) to this case. In particular, $\Prime_\gamma(C^*(E))\cong X_3\cong\Prime_\gamma(C^*(F))$.

\begin{example}\label{notalwayssame}
The pair of graphs $E$ and $F$  satisfy condition \ref{FK} of Theorem \ref{mainthm}, but not condition \ref{MOV}. The same is true for the pair of graphs $E_\curlyvee$ and  $F_\curlyvee$.
\end{example}
\begin{proof}
  We have seen in Example \ref{notalwayssamepre} that $E\not\MCeq F$, and since $E=(E_\curlyvee)_\curlywedge$ and $F=(F_\curlyvee)_\curlywedge$ we conclude that $E_\curlyvee\not\MCeq F_\curlyvee$ by transposition of Lemma \ref{CEpassestoplugged}.

To see that the $K$-theories are isomorphic, we note that
\[
U
\begin{pmatrix}0&1&2\\0&1&1\\0&0&0
\end{pmatrix}
V
=
\begin{pmatrix}0&1&0\\0&1&1\\0&0&0
\end{pmatrix}
 \]
 with $V=I$ and 
 \[
 U=\begin{pmatrix}-1&2&0\\0&1&0\\0&0&1\end{pmatrix}
\]
This \GLP-equivalence induces an isomorphism $\FKR(X_3;C^*(E))\cong \FKR(X_3;C^*(F))$ as noted in Section \ref{sec:red-filtered-K-theory-K-web-GLP-and-SLP-equivalences}, and since $V\{i\}=1$ at all blocks, the maps induced by $V^{\mathsf T}$ on the $K_0$-groups are order isomorphisms.
The isomorphism of $K$-theory for  $E=(E_\curlyvee)_\curlywedge$ and $F=(F_\curlyvee)_\curlywedge$ follows from Lemma \ref{Ktheoryplug}.
\end{proof}

In fact, in this particular case, reversal of the chain of implications in Theorem~\ref{mainthm} breaks down at \ref{MOR}$\implies$\ref{MOV}. To prove this, we provide an \emph{ad hoc} classification of a small class of \cas of relevance.

Let $\A$ be a $C^*$-algebra and let $\mathfrak{I}$ be an ideal of $\A$.  Set 
\[
\mathcal{M} ( \A ; \mathfrak{I} ) := \setof{ x \in \mathcal{M}( \A ) }{ \text{$ax , x a \in \mathfrak{I}$ for all $a \in \A$}}
\]        
and set 
\[
\mathcal{Q} ( \A ; \mathfrak{I} ) := ( \mathcal{M} ( \A; \mathfrak{I} ) + \A) / \A.
\]

\begin{lemma}\label{lem:relative-multiplier-alg}
Let $\A$ be a stable separable $C^*$-algebra such that $\A$ has a unique non-trivial ideal $\mathfrak{I}$ with $\mathfrak{I}$ either isomorphic to $\K$ or $\mathfrak{I}$ is a stable, separable, purely infinite simple $C^*$-algebra and $\A / \mathfrak{I}$ is either isomorphic to $\K$ or is a stable, separable, purely infinite simple $C^*$-algebra.  Then $\mathcal{Q} ( \A , \mathfrak{I} )$ is the unique non-trivial ideal of $\mathcal{Q} ( \A )$.
\end{lemma}        

\begin{proof}
Note that $\mathfrak{I}$ is an essential ideal of $\mathcal{M} ( \A , \mathfrak{I} )$.  Hence, this embedding extends to an embedding $\ftn{\iota}{\mathcal{M} ( \A , \mathfrak{I} )}{\mathcal{M} (\mathfrak{I} ) }$.  We claim that $\iota ( \mathcal{M} ( \A , \mathfrak{I} ) )$ is a full hereditary subalgebra of $\mathcal{M} (\mathfrak{I} )$.  

Let $x = ( L_0, R_0 )  \in \mathcal{M} ( \mathfrak{I} )$ and let $s, t \in \mathcal{M} ( \A , \mathfrak{I} )$ (where we are using the double centralizer picture of the multiplier algebra).  Define $\ftn{L,R}{\A}{\A}$ by $L(a) =s ( L_0 (ta) )$ and $R(a) = R_0(as)t$.  Note that $L$ and $R$ are well-defined since $ta$ and $as$ are elements of $\mathfrak{I}$ for all $a \in \A$.  A computation shows that $L$ and $R$ are linear and $\| L \|$ and $\| R \|$ are bounded above by $\|s\|\cdot\|x\|\cdot\|t\|$.

Let $\{ e_n \}_{n = 1 }^\infty$ be an approximate identity for $\mathfrak{I}$.  For all $a, b \in \A$, we have that 
\begin{align*}
R( ab) &= R_0( ab s) t  = \lim_{ n \to \infty } R_0 ( ae_n b s ) t  \\
	&= \lim_{ n \to \infty} (ae_n) \left( R_0( bs) t \right) = \lim_{n \to \infty } a ( e_n R_0(bs) t) \\
	&= a ( R_0(bs) t ) = a R(b), \\
L( ab) &= s L_0 ( tab) = \lim_{n \to \infty } s L_0 ( t a e_n b )  \\
	&= \lim_{ n \to \infty } s L_0 ( t a ) (e_n b ) = \lim_{n \to \infty } ( s L_0 ( ta ) e_n ) b   \\
	&= ( s L_0 ( ta ) ) b = L(a)  b,\\
R(a)b &=( R_0 ( as ) t ) b = R_0(as) (tb)  \\
	&= as L_0 ( t b ) = a ( s L_0 (tb) ) \\
	&= a L(b).
\end{align*}
Hence, $y = (L,R)$ defines an element of $\mathcal{M}(\A)$.  

Let $a , b\in \A$.  Then 
\[
L_a L (b) = L_a (s L_0 (tb) ) = (a s) L_0 (tb)  =  R_0( as ) tb = L_{ R_0(as) t }(b)
\] 
and
\begin{align*}
R R_a (b) &= R ( ba ) = R_0( bas) t = \lim_{ n \to \infty } R_0 ( b e_n as ) t \\
		&= \lim_{ n \to \infty }b (e_n  R_0 ( as )t) = b ( R_0 ( as ) t ) = R_{ R_0(as) t }(b).
\end{align*}
Therefore, $( L_a, R_a )( L, R ) = ( L_a L , R R_a ) =  (  L_{ R_0(as) t } , R_{ R_0(as) t } ) \in \mathfrak{I}$.  Similarly computation shows that $( L, R )( L_a, R_a ) = (L_{s L_0 (ta) } , R_{ s L_0 (ta) } ) \in \mathfrak{I}$.  Hence, $(L, R ) \in \mathcal{M} ( \A , \mathfrak{I} )$.  Note that $\iota ( s ) = ( L_s, R_s )$, where we restrict $L_s$ and $R_s$ to $\mathfrak{I}$.  Similarly, for $\iota (t)$.  Thus, 
\[
\iota( s ) x \iota(t) = ( L_s, R_s ) x ( L_t, R_t ) = (L_s, R_s ) ( L_0, R_0 ) ( L_t , R_t ) = (L_s L_0 L_t , R_t R_0 R_s )
\]
and
\begin{align*}
L_s L_0 L_t (z) &= L_s ( L_0 ( t z) ) = s L_0 ( tz) = L(z) \\
R_t R_0 R_s (z) &= R_t ( R_0 ( z s ) ) = R_0 ( zs ) t = R(z)
\end{align*}
for all $z \in \mathfrak{I}$.  Hence, $\iota (y) = \iota (s ) x \iota (t)$.  Therefore, $\iota ( \mathcal{M} ( \A , \mathfrak{I} ) )$ is a hereditary subalgebra of $\mathcal{M}( \mathfrak{I} )$.

 We claim that $\iota ( \mathcal{M} ( \A , \mathfrak{I} ) ) \neq \mathfrak{I}$.  Let $\{ s_n \}_{ n = 1 }^\infty$ be a collection of isometries in $\mathcal{M}(\A)$ such that $\sum_{ n = 1}^\infty s_n s_n^*$ converges to $1_{\mathcal{M}(\A)}$ in the strict topology (note such a collection of isometries exists since $\A$ is a stable $C^*$-algebra).  Let $a \in \mathfrak{I} \setminus \{0\}$.  Then $\sum_{ n = 1 }^\infty s_n a s_n^*$ converges in the strict topology of $\mathcal{M}(\A)$.  Therefore, $x = \sum_{ n = 1 }^\infty s_n a s_n^*$ is an element of $\mathcal{M}(\A)$.  In fact, $x \in \mathcal{M}(\A, \mathfrak{I} )$ since $a \in \mathfrak{I}$.  Since $\| s_n a s_n^* \| = \|a\| \neq 0$, we have that $x \notin \A$.  Therefore, $\mathfrak{I} \neq \mathcal{M} ( \A ; \mathfrak{I} )$.  So, $\iota ( \mathcal{M} ( \A , \mathfrak{I} ) ) \neq \mathfrak{I}$, which proves our claim.

By \cite[Theorem~3.2]{MR1203034}, $\mathcal{M} ( \mathfrak{I} )$ has exactly one non-trivial ideal $\mathfrak{I}$.  Therefore, $\iota ( \mathcal{M} ( \A ; \mathfrak{I} ) )$ is a full hereditary subalgebra of $\mathcal{M} ( \mathfrak{I} )$.  Thus, $\mathcal{M} ( \A ; \mathfrak{I} )$ has exactly one non-trivial, $\mathfrak{I}$.  Consequently, $\mathcal{Q} ( \A ; \mathfrak{I} )$ is a simple $C^*$-algebra.

Let $\ftn{\pi}{\A}{\A / \mathfrak{I}}$ be the canonical projection.  Then it induces surjective \starhomos $\ftn{\widetilde{\pi}}{\mathcal{M} ( \A ) }{ \mathcal{M} ( \A / \mathfrak{I} ) }$ and $\ftn{\overline{\pi}}{ \mathcal{Q} ( \A )}{ \mathcal{Q} ( \A / \mathfrak{I} )}$.  Note that $\ker( \widetilde{\pi} ) = \mathcal{M} ( \A ; \mathfrak{I} )$ and $\ker( \widetilde{\pi} ) = \mathcal{Q} ( \A ; \mathfrak{I} )$.  Now, we have an exact sequence
\[
0 \to \mathcal{Q} ( \A ; \mathfrak{I} ) \to \mathcal{Q} (\A ) \to \mathcal{Q} ( \A / \mathfrak{I} ) \to 0.
\]
By \cite[Theorem~3.2]{MR1203034}, $\mathcal{Q} ( \A / \mathfrak{I} )$ is a simple $C^*$-algebra.  Thus, $\mathcal{Q} ( \A ; \mathfrak{I} )$ must be the unique non-trivial ideal of $\mathcal{Q} (\A)$. 
\end{proof}

\begin{theorem}\label{thm:  classification special case}
Let $\A_{1}$ and $\A_{2}$ be unital $C^{*}$-algebras equipped with gauge actions.  Suppose for each $i$, there exist gauge invariant ideals $\I_{i,1}$ and $\I_{i,2}$ of $\A_{i}$ such that
\begin{enumerate}[(i)]
\item \label{thm:  classification special case-1}
$\I_{i,1} \subseteq \I_{i,2}$,

\item \label{thm:  classification special case-2}
$\I_{i,1} \cong \K$,

\item \label{thm:  classification special case-3}
$\I_{i,2} / \I_{i,1}$ is isomorphic to the stabilization of a unital, simple purely infinite graph \ca,

\item \label{thm:  classification special case-4}
$\A_{i} / \I_{i,2} \cong C(S^{1})$, and

\item \label{thm:  classification special case-5}
$\I_{i,2} / \I_{i,1}$ is an essential ideal of $\A_{i} / \I_{i,1}$.
\end{enumerate}
If $\FKRplus (X_3; \A_{1} \otimes \K ) \cong \FKRplus (X_3;\A_{2} \otimes \K )$, then $\A_{1} \otimes \K \cong \A_{2} \otimes \K$. 
\end{theorem}

\begin{proof}
Let $\alpha$ be the isomorphism from $\FKRplus (X_3; \A_{1} \otimes \K
)$ to $\FKRplus (X_3; \A_{2} \otimes \K )$.  Let $\mathfrak{e}_{i}$ be
the extension $0 \to \I_{i,2} \otimes \K \to \A_{i} \otimes \K \to
\A_{i} / \I_{i,2} \otimes \K \to 0$.  We first show that
$\mathfrak{e}_i$ is a full extension.  By
Lemma~\ref{lem:relative-multiplier-alg}, the corona algebra $\corona{
  \I_{i,2} \otimes \K }$ has exactly one nontrivial ideal.  This
ideal is precisely the kernel of the surjective map $\overline{\pi}
\colon \corona{ \I_{i,2} \otimes \K } \rightarrow \corona{ \I_{i,2} /
  \I_{i,1} \otimes \K }$ that is induced by the surjective map $\pi
\colon \I_{i,2} \otimes \K \rightarrow \I_{i,2} / \I_{i,1} \otimes
\K$.  Therefore, $x \in \corona{ \I_{i,2} \otimes \K}$ is full if and
only if its image in $\corona{ \I_{i,2} / \I_{i,1} \otimes \K }$ is
nonzero. 
Note that the diagram
\[
\xymatrix{
0 \ar[r] & \I_{i,2} \otimes \K \ar[r] \ar[d]^{\pi} & \A_{i} \otimes \K \ar[r] \ar[d]  & ( \A_{i} / \I_{i,2} ) \otimes \K \ar[r] \ar@{=}[d]& 0 \\ 
0 \ar[r] & ( \I_{i,2} / \I_{i,1} ) \otimes \K \ar[r] & ( \A_{i} / \I_{i,1} ) \otimes \K \ar[r] & ( \A_{i} / \I_{i,2} ) \otimes \K \ar[r] & 0
}
\]
is commutative.  Hence, with $\boldsymbol{\tau}$ denoting Busby maps, $\overline{\pi} \circ \boldsymbol{\tau}_{\mathfrak{e}_i} = \boldsymbol{\tau}_{\mathfrak{g}_i}$, where $\mathfrak{g}_i$ is the extension
\[
\xymatrix{
0 \ar[r] & ( \I_{i,2} / \I_{i,1} ) \otimes \K \ar[r] & ( \A_{i} / \I_{i,1} ) \otimes \K \ar[r] & ( \A_{i} / \I_{i,2} ) \otimes \K \ar[r] & 0.
}
\]
By assumption \ref{thm:  classification special case-5}, $\boldsymbol{\tau}_{ \mathfrak{g}_i } ( x )$ is nonzero in $\corona{ (\I_{i,2} / \I_{i,1}) \otimes \K }$ for all nonzero $x \in ( \A_{i} / \I_{i,1} ) \otimes \K$.  Hence, by the above observations, $\boldsymbol{\tau}_{\mathfrak{e}_i}(x)$ is full in $\corona{ \I_{i,2} \otimes \K }$. Since $\I_{i,2} \otimes \K$ has the corona factorization property (see, \eg, \cite[Proposition 6.1]{MR3056712}) $\mathfrak e_i$ is an absorbing extension.

Since $\A_i / \I_{i,2} \otimes \K$ is $C( S^{1} ) \otimes \K$, there exists a \stariso $\beta_{2}$ from $\A_{1} / \I_{1,2} \otimes \K$ to $\A_{2} / \I_{2,2} \otimes \K$ which induces 
$\alpha$ restricted to $K_*(\A_1 / \mathfrak{I}_{1,2})$
(we are using the fact that a positive automorphism on $K_{*} ( C(S^{1}) )$ is induced by $\operatorname{id}_{ C( S^{1} )  \otimes \K}$ or $\psi \otimes \operatorname{id}_\K$ where $\psi$ sends the canonical generator of $C( S^{1})$, denoted by $z$, to $z^{-1}$).  Note that $\I_{i,2}$ has a full projection, $\I_{1,2}$ is stably isomorphic to a unital \ca with exactly one nontrivial ideal that is isomorphic to $\K$ and the quotient by this ideal is isomorphic to a unital and simple purely infinite graph \ca.  Using this observation together with \cite[Corollary~4.17 and Proposition~4.19]{arXiv:1301.7695v1}, there exists a \stariso $\ftn{ \beta_{0} }{  \I_{1,2} \otimes \K  }{  \I_{2,2} \otimes \K  }$ which induces $\alpha$ restricted to $K_*( \mathfrak{I}_{1,2})$.

Let $\mathfrak{f}_{1}$ be the extension obtained by pushing forward the extension $\mathfrak{e}_{1}$ via the \stariso $\beta_{0}$ and let $\mathfrak{f}_{2}$ be the extension obtained by pulling back the extension $\mathfrak{e}_{2}$ by the \stariso $\beta_{2}$.  Since $\mathfrak{e}_{i}$ is an absorbing extension, we have that $\mathfrak{f}_{i}$ is an absorbing extension.  By construction, $K_{*} ( \boldsymbol{\tau}_{ \mathfrak{f}_{1}} ) = K_{*} ( \boldsymbol{\tau}_{ \mathfrak{f}_{2} } )$ as homomorphisms from $K_{*} ( ( \A_{1}  /  \I_{1,2} ) \otimes \K )$ to $K_{1-*} (\I_{2,2} \otimes \K )$.  Hence, by the UCT of Rosenberg and Schochet \cite{MR894590}, $[ \boldsymbol{\tau}_{ \mathfrak{f}_{1}} ] = [ \boldsymbol{\tau}_{ \mathfrak{f}_{2} } ]$ in $\mathrm{KK}^{1} ( (\A_{1} /  \I_{1,2} ) \otimes \K  , \I_{2,2} \otimes \K )$ since $K_{i} (  ( \A_{1} \otimes \K ) / ( \I_{1,2} \otimes \K ) ) \cong \Z$ for each $i$. 

Since $\mathfrak{f}_{i}$ are absorbing extensions, there exists a unitary $U$ in $\mathcal{M} ( \I_{2,2} \otimes \K )$ such that $\mathrm{Ad} ( \pi (U) ) \circ \boldsymbol{\tau}_{ \mathfrak{f}_{1} } = \boldsymbol{\tau}_{ \mathfrak{f}_{2}}$.  One checks that $\mathrm{Ad} (U)$ induces a \stariso of extensions from $\mathfrak{f}_{1}$ to $\mathfrak{f}_{2}$.  Since $\mathfrak{e}_{i}$ is isomorphic to $\mathfrak{f}_{i}$, we have that $\A_{1} \otimes \K \cong \A_{2} \otimes \K$.
\end{proof}

It is easy to see that this result applies to conclude that
$C^*(E_\curlyvee)\otimes \K\cong C^*(F_\curlyvee)\otimes \K$ for the
pair of examples in Example \ref{notalwayssame}. To deal with $C^*(E)\otimes \K$
and $C^*(F)\otimes \K$, we apply an unplugging trick to get:

\begin{corollary}\label{cor:  classification special case}
Let $E_1$ and $E_2$ be finite graphs with $\Prime_\gamma(C^*(E_i))\cong X_3$ and $\tau_{E_i}(\{0\})\leq0$.  If $\frX{X_3}{E_1}\cong \frX{X_3}{E_2}$, then $C^*(E_1) \otimes \K \cong C^*(E_2) \otimes \K$. 
\end{corollary}

\begin{proof} Assume that  $\frX{X_3}{E_1}\cong \frX{X_3}{E_2}$. We write ${\mathfrak p}^i_j$  for the ideals in $C^*(E_i)$ as in Example \ref{linearcase}. If  $\tau_{E_i}({\mathfrak p}^i_1)=1$ we have Condition (H), and the full force of Theorem \ref{mainthm} applies. We may hence assume that $\tau_{E_i}({\mathfrak p}^i_1)=0$.
Again if $\tau_{E_i}({\mathfrak p}^i_2)=0$, we have Condition (H), so we may assume
$\tau_{E_i}({\mathfrak p}^i_2)=1$.  When $\tau_{E_i}({\mathfrak p}^i_3)=-1$ we note that all the
conditions of Theorem~\ref{thm: classification special case} are met,
so that this result applies to give the desired conclusion. We thus
need only concern ourselves with the case
$\tau_{E_i}({\mathfrak p}^i_3)=0$.

In this case, we pass to $(E_i)_\curlyvee$ and note that
Theorem~\ref{thm: classification special case} applies. Since the
isomorphism provided by that result must satisfy the conditions of
Proposition \ref{prop: unplugging} because the ideal lattice is linear,
we 
get the desired conclusion.
\end{proof}

We conclude:

\begin{example}\label{notalwayssameii}
With $E$ and $F$ the pair of graphs given in Figure \ref{firstexx}(b), we have
\[
C^*(E)\otimes \K\cong C^*(F)\otimes\K,
\]
and
\[
C^*(E_\curlyvee)\otimes \K\cong C^*(F_\curlyvee)\otimes\K.
\]
although (as seen in Example \ref{notalwayssame}) $E\not\MCeq F$ and $E_\curlyvee\not\MCeq F_\curlyvee$.
\end{example}

\section{Applications}\label{applications}

In this section, we give  applications of our results.

\subsection{Type I/postliminal $C^{*}$-algebras}
 In this section we study further the case where  no vertex supports two distinct return paths, \ie\ the case of type I/postliminal \cas in our class, \cf\ the remarks just after Lemma \ref{charKH}.

It was conjectured by Gene Abrams and Mark Tomforde in
\cite{MR2775826} that if the Leavitt path algebras $L_{\C} (E)$ and
$L_{\C} (F)$ are Morita equivalent, then $C^{*} (E)$ and $C^{*}(F)$
are strongly Morita equivalent (see \cite{MR2417402} for the
definition of $L_{\C} (E)$).  Using Theorem~\ref{mainthm}, we can show
that their conjecture holds for finite graphs whose temperatures are
never positive.  Moreover, we show that the converse holds as well in that case.

\begin{theorem}\label{coldiso}
Let $E$ and $F$ be finite graphs where $\max\tau_E,\max\tau_F\leq 0$.  Then the following are equivalent:
\begin{enumerate}[(1)]
\item \label{coldiso-1}
$E \Meq F$.
\item \label{coldiso-2}
$L_{{\mathsf k}} ( E )$ and $L_{{\mathsf k}} (F)$ are Morita equivalent for any field ${\mathsf k}$.
\item \label{coldiso-3} 
$C^{*} (E)\otimes\K\cong C^{*} (F)\otimes \K$.
\end{enumerate}
If  $\tau_{E} =\tau_{F}=\CAtemp$, then \ref{coldiso-1}--\ref{coldiso-3} are equivalent to 
\begin{enumerate}[(1)]\addtocounter{enumi}{3}
\item \label{coldiso-4}
the two-sided shift spaces ${\mathsf X}_{E}$ and ${\mathsf X}_{F}$ are flow equivalent.
\end{enumerate}
If  $(\Bsf_E,\Bsf_F)$ is in standard form, then \ref{coldiso-1}--\ref{coldiso-3}  are equivalent to 
\begin{enumerate}[(1)]\addtocounter{enumi}{4}
\item \label{coldiso-5} 
there exist matrices $U,V\in \SL_{\calP}(\mathbf 1,\Z)$ 
so that $U\Bsf_{E_\curlywedge}V=\Bsf_{F_\curlywedge}$.
\end{enumerate}
\end{theorem}

\begin{proof}
By Section~3 of \cite{MR3045151} (see also \cite{MR3082546}), \ref{coldiso-1}
implies \ref{coldiso-2}.  We can make sense of $\Prime_\gamma$ also for Leavitt path algebras over finite graphs (see Remark \ref{lpacomments}), and
 we have that when $L_{\C} (E)$ and $L_{\C} (F)$ are
Morita equivalent, then $\Prime_\gamma(L_{\C} (E))\cong X\cong \Prime_\gamma(L_{\C} (F))$ for appropriately chosen $X$. Arguing as in the proof of Theorem~4.9
of \cite{MR3188556}, we get that 
$\FKRplus(X,C^*(E))\cong \FKRplus(X,C^*(F))$.  By
this observation together with Theorem~\ref{mainthm}, since obviously
we have Condition~(H), we conclude  that \ref{coldiso-2} implies \ref{coldiso-3}.  Since $\max\tau\leq 0$, no
vertex supports two different return paths, so Move \CC\ is never
allowed, and we have that $E \MCeq F$ if and only if $E \Meq F$.
Therefore, Theorem~\ref{mainthm} gives that \ref{coldiso-3} implies \ref{coldiso-1}.

Assuming now that all components of $E$ and $F$ are cyclic, we get that \ref{coldiso-1} and \ref{coldiso-4} are equivalent by Lemma~\ref{flowvsME}.

Finally we get \ref{coldiso-1}$\Longleftrightarrow$ \ref{coldiso-5}  by appealing to Theorem \ref{iffcharofMeq}. 
\end{proof}

In general (as we shall discuss in \cite{Eilers-Restorff-Ruiz-Sorensen-2}), it may be computationally difficult to determine when two matrices are $\SLP$-equivalent. This is because the problem is equivalent to solving 
\begin{gather}
U\Bsf_E = \Bsf_F W\label{linpart}\\
\forall i\in\calP: \det U\{i\}=\det W\{i\}=1\label{nonlinpart}
\end{gather}
where \eqref{nonlinpart} is not linear.
But when all blocks are $1\times 1$, the determinant conditions are equivalent to all diagonal blocks being identity matrices, and thus deciding if $U\Bsf_EV=\Bsf_F$ as in \ref{coldiso-5} of Theorem \ref{coldiso} reduces to the linear problem \eqref{linpart} which may readily be decided.

\subsection{Quantum lens spaces}

A class of quantum lens spaces $C(L_q(r;(m_1,\dots, m_n)))$ was
studied in \cite{MR2015735}, \cite{arXiv:1603.04678v1} and proved there to be graph
\cas over finite graphs. We immediately see that the \cas are postliminal/type I with every vertex supporting a loop. To decide any isomorphism question among two such \cas one hence need only to compare their $\Primt$-spaces, and if these are homeomorphic, arrange that the corresponding matrices are in standard form and decide $\SLP$-equivalence as in Theorem \ref{coldiso}\ref{coldiso-5} (for each possible homeomorphism).

As an immediate application, we shall see that in fact in some cases there are
several different quantum lens spaces associated to different choices
of secondary parameters $m_i$ even when the dimension $n$ and the
primary parameter $r$ are fixed. Although our classification result applies in the general setting of \cite{arXiv:1603.04678v1}, we will here consider only the original setup from \cite{MR2015735} where $\Primt$ becomes the Alexandrov space of a linear order. 

We emphasize the fact that even though the $K$-groups of the quantum lens spaces carry important information (\cf\  \cite{MR2015735}, \cite{MR3448329}, \cite{arXiv:1603.04678v1}), they are not complete invariants. It follows from Theorem \ref{mainthm} that the reduced ordered filtered $K$-theory is complete, but as we shall see it is much more convenient to work with \SLP-equivalence in this setting.

\begin{definition}
For each $n \in \N$, define the directed graph $L_{2n-1}$ as the graph with $n$ vertices, $L_{2n - 1}^{0} = \{ v_{1} , \dots, v_{n} \}$, and $\frac{ n( n+1 ) }{2}$ edges $\bigcup_{ i = 1}^{n} \{ e_{i,j} \mid j = i, i+1, \dots, n \}$ with $s ( e_{i,j} ) = v_{i}$ and $r( e_{i,j} ) = v_{j}$.  For example, $L_{5}$ is the graph 
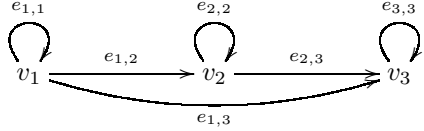
\begin{align*}
\xymatrix{
v_{1} \ar@(ul,ur)[]^{e_{1,1} } \ar[rr]^{ e_{1,2} } \ar@/_1pc/[rrrr]_{e_{1,3}} & & v_{2} \ar@(ul,ur)[]^{ e_{2,2} } \ar[rr]^{e_{2,3}} & & v_{3} \ar@(ul,ur)[]^{e_{3,3}}
}
\end{align*}
\end{definition}

\begin{definition}
For each $r , n \in \N$ and $\mm=( m_{1} , \dots, m_{n} ) \in \N^{n}$, we  define the directed graph $L_{2n-1} \times_{\mm} \Z_{r}$ as follows:

\begin{enumerate}[(i)]
\item The set of vertices is
\begin{align*}
( L_{2n-1} \times_{\mm} \Z_{r} )^{0} = L_{2n-1}^{0} \times \Z_{r}.
\end{align*}

\item The set of edges is
\begin{align*}
( L_{2n-1} \times_{\mm} \Z_{r} )^{1} = L_{2n-1}^{1} \times \Z_{r}.
\end{align*}

\item $s( e_{i,j} , k ) = ( v_{i} , k - m_{i} )$ and $r( e_{i,j} , k ) = ( v_{j} , k )$
\end{enumerate}
\end{definition}

For each $i$, let $( L_{2n-1} \times_{\mm} \Z_{r} ) \langle i \rangle$ be the subgraph with vertex set $\{ v_{i} \} \times \Z_{r}$ and edge set $\{ e_{i,i} \} \times \Z_{r}$.  For each $i_{1} \leq i_{2}\leq  \cdots \leq i_{t}$, let $( L_{2n-1} \times_{\mm} \Z_{r} ) \langle i_{1} , i_{2} , \dots, i_{t} \rangle$ be the subgraph with vertex set $\bigcup_{ l = 1}^{t} \{ v_{l} \} \times \Z_{r}$ and edge set the set of all edges $e$ in $L_{2n-1} \times_{\mm} \Z_{r}$ such that $s(e), r(e) \in \bigcup_{ l = 1}^{t} \{ v_{l} \} \times \Z_{r}$.

\begin{definition}
Let $r \in \N$ and $( m_{1} , m_{2} , \dots, m_{n} ) \in \N^{n}$ with  $r \geq 2$ and $\gcd( m_{i} , r ) =1$ for all $i$.  A path $\alpha = ( e_{i_{1} , j_{1} } , k_{1} ) \cdots ( e_{i_{r},j_{r}} , k_{\ell} )$ in $L_{2n-1} \times_{\mm} \Z_{r}$ is called \emph{$0$-simple} if $k_{1} = m_{i_{1}}$, $k_{a} \neq 0$ for $a \neq \ell$, and $k_{\ell} = 0$.
Note  that for each $0$-simple path $\alpha= ( e_{i_{1} , j_{1} } , k_{1} ) \cdots ( e_{i_{\ell},j_{\ell}} , k_{\ell} ) $, we have that $s( \alpha ) = ( v_{i_{1}} , 0 )$ and $r( \alpha ) = ( v_{j_{\ell}} , 0 )$. Thus the $0$-simple paths may be thought of as paths starting and ending at vertices of the form $(v,0)$, but avoiding all such vertices along the way.

A $0$-simple path $\alpha = ( e_{i_{1} , j_{1} } , k_{1} ) \cdots ( e_{i_{\ell},j_{\ell}} , k_{\ell} )$ is called \emph{$k$-step} if there exist positive integers $ t_{1} < t_{2} < \dots < t_{k+1}$ such that $t_{1} = i_{1}$, $t_{k+1} = j_{\ell}$, and for each $2 \leq q \leq k$, we have that 
\begin{align*}
\setof{ r( ( e_{i_{s} , j_{s} } , k_{s} ) ) }{ 1 \leq s \leq \ell } \cap ( ( L_{2n-1} \times_{\mm} \Z_{r} )\langle t_{q} \rangle )^{0}\neq \emptyset
\end{align*}
and 
\begin{align*}
\setof{ r( ( e_{i_{s} , j_{s} } , k_{s} ) )}{ 1 \leq s \leq \ell } \subseteq \bigcup_{ i = 1}^{k+1} ( ( L_{2n-1} \times_{\mm} \Z_{r} )\langle t_{i} \rangle )^{0}.
\end{align*}
\end{definition}

\begin{definition}
Let $r \in \N$ and $\mm=( m_{1} , m_{2} , \dots, m_{n} ) \in \N^{n}$ with $\gcd( m_{i} , r ) =1$ and $r \geq 2$.  Define $L_{2n-1}^{(r; \mm )}$ to be the graph with vertices $\{ (v_{1}, 0) , \dots, (v_{n}, 0 ) \}$, the edges of $L_{2n-1}^{(r; \mm )}$ consisting of all $0$-simple paths in $L_{2n-1} \times_{\mm} \Z_{r}$, and the range and source maps extending the range and source maps of $L_{2n-1} \times_{\mm} \Z_{r}$.
\end{definition}

Note that by our assumption on the $m_i$, they are always units in $(\Z_r\backslash\{0\},\cdot)$. We denote by $m_i^{-1}$ any representative in $\Z$ of a multiplicative inverse to $m_i$ modulo $r$.

\begin{lemma}\label{l:PathsQuantum}
Let $r \in \N$ and $\mm\in \N^{n}$ with $\gcd( m_{i} , r ) =1$ and $r \geq 2$.  
\begin{enumerate}[(i)]
\item\label{l:PathsQuantum:1} For each $i, j$ with $i +1 \leq j$, the number of $1$-step $0$-simple paths from $( v_{i} , 0 )$ to $( v_{j} , 0 )$ is $r$.

\item\label{l:PathsQuantum:2} For each $i,j$ with $i+2 \leq j$, the number of $2$-step $0$-simple paths from $( v_{i} , 0 )$ to $( v_{j} , 0 )$ is $\frac{ r ( r - 1 ) }{ 2 }( j - i - 1 )$.

\item\label{l:PathsQuantum:3} For each $i$, the number of $3$-step $0$-simple paths from $( v_{i} , 0 )$ to $( v_{i+3} , 0 )$ is congruent to $- m_{i+2}^{-1} m_{i+1} \left( \frac{ r ( r - 1 )( r - 2 ) }{ 3 } \right)$ modulo $r$
\end{enumerate}

Consequently, the number of $0$-simple paths from $( v_{i} , 0 )$ to $(v_{i+2}, 0 )$ is $\frac{ r ( r + 1 ) }{ 2 }$ and the number of  $0$-simple paths from $( v_{i} , 0 )$ to $(v_{i+3}, 0 )$ is congruent to $$- m_{i+2}^{-1} m_{i+1} \left( \frac{ r ( r - 1 )( r - 2 ) }{ 3 } \right)$$ modulo $r$.
\end{lemma}

\begin{proof}
We first prove \ref{l:PathsQuantum:1}.  Note that for each $0 \leq k < r$, there is exactly one edge from $( v_{i} , k )$ to $(L_{2n-1} \times_{\mm} \Z_{r})\langle j \rangle$.  Since there is exactly 1 path from $( v_{j}, l )$ to $( v_{j} , 0 )$ which passes through $( v_{j} , 0 )$ once, we have that the number of $1$-step $0$-simple paths from $( v_{i} , 0 )$ to $( v_{j} , 0 )$ is equal to the number of edges in the subgraph $(L_{2n-1} \times_{\mm} \Z_{r})\langle i \rangle$. This is  equal to $r$, so \ref{l:PathsQuantum:1} holds.

We now prove \ref{l:PathsQuantum:2}.  Let $V$ be the set of all $2$-step $0$-simple paths from $( v_{i}, 0 )$ to $( v_{j} , 0 )$.  For each $l$ with $1 \leq  l \leq j-i-1$, let $V_{l}$ be the set of all $2$-step $0$-simple paths from $( v_{i} , 0 )$ to $( v_{j} , 0 )$ that goes through the subgraph $( L_{2n-1} \times_{\mm} \Z_{r} ) \langle i+l\rangle$.  Then $V = \bigsqcup_{ l = 1}^{j-i-1} V_{l}$.  By symmetry $| V_{l} | = | V_{1} |$,  so $| V | = | V_{1} | (j - i - 1)$.

Let $\alpha = \alpha_{1} \cdots \alpha_{t} \in V_{1}$ and recall that for all $k$, $r(\alpha_k)\neq (v_i,0)$ and $r(\alpha_k)\neq (v_{i+1},0)$. Since for each $1 \leq l \leq r-1$, 
there is exactly one path from $(v_i,0)$ to $(v_i,l-m_i)$ that does not come back to $(v_i,0)$, and there is exactly one edge from $(v_i,l-m_i)$ to $(v_{i+1},l)$,
we have that 
\begin{align*}
| V_{1} | = \sum_{ l = 1}^{r-1} P_{l}
\end{align*}
where $P_{l}$ is the number of paths from $( v_{i+1} , m_{i+1}l )$
 to $( v_{j} , 0 )$ in the subgraph $( L_{2n-1} \times_{\mm} \Z_{r} )\langle i+1, j \rangle$ that do not go through $( v_{i+1} , 0 )$.  Clearly, $P_l=r-l$,  so 
\begin{align*}
| V_{1} | = \sum_{ l = 1}^{r-1} P_{l} = \sum_{ l = 1}^{r-1} ( r - l ) = r (r-1) - \frac{ r(r-1)}{2} = \frac{ r ( r - 1 ) }{ 2 }.
\end{align*}
Therefore, \ref{l:PathsQuantum:2} holds.

We now prove \ref{l:PathsQuantum:3}. For each $1\leq l\leq r-2$, we have an edge $(e_{i+1},m_{i+1}(l+1))$ from $(v_{i+1},m_{i+1}l)$ to $(v_{i+2},m_{i+1}(l+1))$. 
We let $Q_l$ be the number of paths from $(v_{i+2},m_{i+1}(l+1))$ to $(v_ {i+3},0)$ that do not go through $(v_{i+2},0)$ and only go once through $(v_{i+3},0)$. Since there are exactly $l$ paths from $(v_i,0)$ to $(v_{i+1},m_{i+1}l)$ that do not come back to $(v_i,0)$ and do not go through $(v_{i+1},0)$, we have that the number of $3$-step $0$-simple paths from $(v_i,0)$ to $(v_{i+3},0)$ is $\sum_{l=1}^{r-2} l Q_l$ .

Recall that $m_{i+2}^{-1}$ is a representative of the multiplicative inverse of $m_{i+2}$ modulo $r$, and let $s_l$ be the integer such that $0<m_{i+2}^{-1}m_{i+1}(l+1)+rs_l<r$. Since $m_{i+1}(l+1)$ is congruent to $m_{i+2}(m_{i+2}^{-1}m_{i+1}(l+1)+rs_l)$ modulo $r$, it follows from the proof of part \ref{l:PathsQuantum:2} that
$$Q_l=r-(m_{i+2}^{-1}m_{i+1}(l+1)+rs_l).$$  Hence, the number of $3$-step $0$-simple paths from $( v_{i} , 0 )$ to $( v_{i+3} , 0 )$ is 
\begin{align*}
&\sum_{ l = 1}^{r-2} l (r - m_{i+2}^{-1} m_{i+1} (l+1) - r s_{l} ) \\
&\qquad \equiv \sum_{ l = 1}^{r-2}(   - m_{i+2}^{-1} m_{i+1} l(l+1) ) \mod r \\
&\qquad \equiv -m_{i+2}^{-1} m_{i+1} \frac{ r( r-1)( r-2)}{3} \mod r.
\end{align*}
Hence, \ref{l:PathsQuantum:3} holds.  

For the last part of the lemma, by \ref{l:PathsQuantum:1} and \ref{l:PathsQuantum:2}, we have that the number of $0$-simple paths from $( v_{i} , 0 )$ to $( v_{i+2} , 0 )$ is equal to $ r + \frac{ r ( r - 1 ) }{2} = \frac{ r ( r + 1 ) }{ 2 }$ and by \ref{l:PathsQuantum:1}, \ref{l:PathsQuantum:2}, \ref{l:PathsQuantum:3}, we have that the number of $0$-simple paths from $( v_{i}, 0 )$ to $( v_{i+3} , 0 )$ is congruent to $r + \frac{ r ( r - 1 ) }{2} + \frac{ r ( r - 1 ) }{2} - m_{i+2}^{-1} m_{i+1} \left( \frac{ r ( r - 1 )( r - 2 ) }{ 3 } \right)$ modulo $r$.  It is now clear that the conclusion holds.
\end{proof}

\begin{corollary}\label{Kthy}
$K_0(C^{*} ( L_{2n-1}^{(r ; \mm )} ))\cong \Z\oplus G$ for $G$ some group of order $|G|=r^{n-1}$
\end{corollary}
\begin{proof}
The first row and the last column of $(\Bsf_{L_{2n-1}^{(r;\mm')}})^{\mathsf T}$  are zero. The remaining $(n-1)\times (n-1)$ submatrix is upper triangular and has $r$ in the diagonal as seen in Lemma~\ref{l:PathsQuantum}\ref{l:PathsQuantum:1}, and thus the determinant is $r^{n-1}$. Now the lemma follows (\eg\ by using the Smith normal form).
\end{proof}

Determining $G$ exactly is a difficult problem, \cf\ \cite{MR3448329}. 

Since we obviously have  
\[
|\Primt(C^*(L_{2n-1}^{(r;\mm)}))|=|\Gamma_{L_{2n-1}^{(r;\mm)}}|=n,
\]
 the isomorphism class of 
$C^*(L_{2n-1}^{(r;\mm)})\otimes\K$ determines $n$ and hence, by Corollary \ref{Kthy}, also $r$. Further, Lemma \ref{l:PathsQuantum}\ref{l:PathsQuantum:1} and \ref{l:PathsQuantum:2} show that the graphs and their adjacency matrices are the same irrespective of $\mm$  when $r$ is fixed and $n\leq 3$. When $n=4$, something new happens precisely when $r$ is a multiple of 3.

\begin{theorem}\label{t:QuantumIso}
Let $r\geq 2$ be given and let $\mm=( m_{1} , m_{2} , m_{3} , m_{4} )$ and $\nn=( n_{1} , n_{2} , n_{3} , n_{4} )$ be given in $\N^4$ such that $\gcd( m_{i} , r)=\gcd( n_{i} , r )  =1$ for all $i$. Then the following are equivalent:
\begin{enumerate}[(1)]
\item \label{t:QuantumIso-1}
$C^{*} ( L_{7}^{(r; \mm)} ) \cong C^{*} ( L_{7}^{(r; \nn)} )$,
\item \label{t:QuantumIso-2}
$C^{*} ( L_{7}^{(r ; \mm )} )\otimes\K\cong C^{*} ( L_{7}^{(r; \nn )} )\otimes \K$, 
\item \label{t:QuantumIso-3}
$\left( m_{3}^{-1} m_{2} - n_{3}^{-1} n_{2} \right) \left( \frac{ r ( r - 1)( r - 2 ) }{ 3 } \right) \equiv 0 \mod r$.
\end{enumerate}
\end{theorem}

\begin{proof}
Let $\Asf_{\mm}$ be the adjacency matrix for $L_{7}^{(r ;\mm )}$ and let $\Asf_{\nn}$ be the adjacency matrix for $L_{7}^{(r ;\nn )}$, with $\Bsf_{\mm}$ and $\Bsf_{\nn}$ obtained by subtraction of the identity matrix as usual.  By Lemma~\ref{l:PathsQuantum},
\begin{align*}
\Asf_{\mm} = 
\begin{pmatrix} 
1 & r &\frac{ r ( r + 1 ) }{2} & x \\
0 & 1 & r&\frac{ r ( r + 1 ) }{2}\\
0&0&1&r\\
 0&0&0& 1
\end{pmatrix} 
\quad \text{and} \quad \Asf_{\nn} = 
\begin{pmatrix} 
1 & r &\frac{ r ( r + 1 ) }{2} & y \\
0 & 1 & r&\frac{ r ( r + 1 ) }{2}\\
0&0&1&r\\
 0&0&0& 1
\end{pmatrix} 
\end{align*} 
where $x \equiv  - m_{3}^{-1} m_{2} \frac{ r(r-1)( r-2) }{3 }  \mod r$ and $y \equiv - n_{3}^{-1} n_{2} \frac{ r(r-1)( r-2) }{3 }  \mod r$.

We first show that \ref{t:QuantumIso-2} implies \ref{t:QuantumIso-3}.  By Theorem~\ref{coldiso}\ref{coldiso-5}, there
exist $U, V \in \operatorname{SL}_{\calP_4} (\mathbf{1}, \Z )$ such that $U \Bsf_{\mm}
V = \Bsf_{\nn}$, with $\calP_4=\{1,2,3,4\}$ ordered linearly.  Note that $U, V$ are upper triangular matrices and
$U\{ i \} = V \{ i \} = 1$ for $i = 1, 2, 3, 4$.
A computation implies that 
\begin{align*}
x + r s_{1} + \frac{ r ( r+1) }{2} s_{2} = y 
 \end{align*}    
for some $s_{1} , s_{2} \in \Z$.  Since $y \equiv - n_{3}^{-1} n_{2} \frac{ r(r-1)( r-2) }{3 }  \mod r $ and $x \equiv - m_{3}^{-1} m_{2} \frac{ r(r-1)( r-2) }{3 }  \mod r $, we have that 
\begin{align*}
( m_{3}^{-1}m_{2} - n_{3}^{-1} n_{2} ) \frac{ r(r-1)( r-2) }{3 } \equiv \frac{r(r+1)}{2}s_{2} \mod r. 
\end{align*}
Thus, 
\begin{align}\label{eq:QuantumIso1}
( m_{3}^{-1}m_{2} - n_{3}^{-1} n_{2} ) \frac{ r(r-1)( r-2) }{3 } + r m = \frac{r(r+1)}{2}s_{2}
\end{align} 
for some $m \in \Z$.

Suppose $r$ is odd.  Then $\frac{r(r+1)}{2}s_{2} \equiv 0 \mod r$ and hence \ref{t:QuantumIso-3} holds.  Suppose $r$ is even, say $r = 2^{t} k$ where $\gcd( k , 2 ) = 1$.  Dividing Equation~(\ref{eq:QuantumIso1}) by $2^{t-1}$, we get 
\begin{align}\label{eq:QuantumIso2}
( m_{3}^{-1}m_{2} - n_{3}^{-1} n_{2} ) \frac{ 2k(r-1)( r-2) }{3 } + 2m k = k(r+1)s_{2}.
\end{align}     
Since $3$ divides $r( r - 1)( r - 2)$, we have that $3$ divides $k(r-1)(r-2)$.  Therefore, $\frac{ k ( r - 1)( r - 2) }{3} \in \Z$.  Hence, the left hand side of Equation~(\ref{eq:QuantumIso2}) is divisible by $2$ which implies that $2$ divides $(r+1) s_{2}$.  Since $r$ is even, $2$ divides $s_{2}$.  Thus, $\frac{ r ( r+1) }{2} s_{2} \equiv 0 \mod r$.  Hence, \ref{t:QuantumIso-3} holds.

We now show that \ref{t:QuantumIso-3} implies \ref{t:QuantumIso-1}.  Since $\left( m_{3}^{-1} m_{2} - n_{3}^{-1} n_{2} \right) \left( \frac{ r ( r - 1)( r - 2 ) }{ 3 } \right) \equiv 0 \mod r$, $x \equiv - m_{3}^{-1} m_{2} \frac{ r(r-1)( r-2) }{3 }  \mod r$, and $y \equiv - n_{3}^{-1} n_{2} \frac{ r(r-1)( r-2) }{3 }  \mod r$, we have that $x \equiv y \mod r$.  Therefore, $x + r s  = y + r t$ for some positive integers $s, t$.  

Consider the matrix
\begin{align*}
C = 
\begin{pmatrix} 
1 & r &\frac{ r ( r + 1 ) }{2} & x+rs \\
0 & 1 & r&\frac{ r ( r + 1 ) }{2}\\
0&0&1&r\\
 0&0&0& 1
\end{pmatrix} = \begin{pmatrix} 
1 & r &\frac{ r ( r + 1 ) }{2} & y+rt\\
0 & 1 & r&\frac{ r ( r + 1 ) }{2}\\
0&0&1&r\\
 0&0&0& 1
\end{pmatrix} \end{align*} 

By applying Proposition~\ref{p:isoaddingrows}, $s$ times (note that $x>0$), we get that $C^{*} ( L_{7}^{(r; \mm)} ) \cong C^{*} ( \Esf_{C} )$.  Similarly, we can apply Proposition~\ref{p:isoaddingrows}, $t$ times, we get that $C^{*} ( L_{7}^{(r; \nn)} ) \cong C^{*} ( \Esf_{C} )$.  
\end{proof}

It is in fact true in general (also in the general setting of \cite{arXiv:1603.04678v1}) that whenever two quantum lens spaces are stably isomorphic, they are isomorphic. We will pursue this in \cite{Eilers-Restorff-Ruiz-Sorensen-2}.

\begin{corollary}\label{c:QuantumIso}
If $3$ does not divide $r$, then 
\begin{align*}
C^{*} ( L_{7}^{(r ; (1,1,1,1) )} ) \cong C^{*} ( L_{7}^{(r ;  \mm )} )
\end{align*}
for all $\mm=( m_{1} , m_{2} , m_{3}, m_{4} ) \in \N^{4}$ with $\gcd( m_{i} , r ) = 1$.

Suppose $r = 3s$ and let $\mm\in \N^{4}$ with $\gcd( m_{i} , r ) = 1$ be given.  Then
\begin{align*}
C^{*} ( L_{7}^{(r ; \mm )} ) \cong C^{*} ( L_{7}^{(r ; (1,1,1,1) )} )
\end{align*}
if and only if $m_{2} \equiv m_{3} \mod 3$ and 
\begin{align*}
C^{*} ( L_{7}^{(r ; \mm )} ) \cong C^{*} ( L_{7}^{(r ; (1, 1, r-1 ,1 ))} )
\end{align*}
 if and only if $m_{2} \not\equiv m_{3} \mod 3$.
\end{corollary}

The isomorphism question for quantum lens spaces was introduced in \cite{MR2015735} and some $K$-groups were explicitly computed there. We note here that the $K$-groups in their own right do not contain sufficient information to classify, even if one takes the order into account.

\begin{remark}
The triple 
\begin{align*}
\left( K_{0} ( C^{*} ( L_{7}^{(r ; \mm )} ) ) , K_{0} ( C^{*} ( L_{7}^{(r ;  \mm )} ) )_{+} , K_{1} ( C^{*} ( L_{7}^{(r ;  \mm )} ) ) \right)
\end{align*}
is not a complete isomorphism invariant.  

Set $E = L_{7}^{(3; (1, 1,1,1)) }$ and $F= L_{7}^{(3;(1,1,2,1))}$ with adjacency matrices
\[
\Asf_E=\begin{pmatrix} 
1 & 3 &6& 10 \\
0 & 1 & 3&6\\
0  & 0& 1 &  3 \\
 0 & 0 & 0 & 1
\end{pmatrix} 
\qquad 
\Asf_F=\begin{pmatrix} 
1 & 3 &6& 11 \\
 0& 1 & 3&6\\
0  & 0& 1 &  3 \\
 0 & 0 & 0 & 1
\end{pmatrix} 
\]
  By Corollary~\ref{c:QuantumIso}, we have that $C^{*} (E)$ and $C^{*} (F)$ are not stably isomorphic.  We will show that  
\begin{align*}
\left( K_{0} ( C^{*} ( E ) ) , K_{0} ( C^{*} ( E ) )_{+} , K_{1} ( C^{*} ( E ) )  \right) \cong \left( K_{0} ( C^{*} ( F ) ) , K_{0} ( C^{*} ( F ) )_{+} , K_{1} ( C^{*} ( F ) )  \right)
\end{align*} 
Because of the symmetry in the antidiagonal of these two matrices, we have $\Csf_E=\Bsf_E$ and  $\Csf_F=\Bsf_F$ and may hence  consider the $K$-groups as given by the kernels and cokernels of $\Bsf_E$ and $\Bsf_F$ themselves (see Remark \ref{howtocompute} and Section \ref{sec:red-filtered-K-theory-K-web-GLP-and-SLP-equivalences})

Let $e_{i}$ be the vector with $1$ in the $i$-th coordinate and zero
elsewhere, let $[ e_{i} ]_{E}$ be the class in $\coker(
\Bsf_{E})$, and let $[ e_{i} ]_{F}$ be the class in
$\coker( \Bsf_{F})$.  Under our identification of the $K_0$-groups of $C^*(E)$ and $C^*(F)$ with cokernels of $\Bsf_E$ and $\Bsf_F$, the positive cones become exactly
\begin{align*}
S_{E} = \{ n_{1} [ e_{1} ]_{E} + n_{2} [ e_{2} ]_{E} + n_{3} [ e_{3} ]_{E} + n_{4} [ e_{4} ]_{E} : n_{i} \in \N_{ 0 } \}
\end{align*} 
and 
\begin{align*}
S_{F} = \{ n_{1} [ e_{1} ]_{F} + n_{2} [ e_{2} ]_{F} + n_{3} [ e_{3} ]_{F} + n_{4} [ e_{4} ]_{F} : n_{i} \in \N_{ 0 } \},
\end{align*}  
respectively.
Hence, it is enough to show that
\begin{align*}
( \coker( \Bsf_{E}) , S_{E} , \ker( \Bsf_{E} )) \cong ( \coker( \Bsf_{F}) , S_{F} , \ker( \Bsf_{F} )).
\end{align*}

Set 
\begin{align*}
U= 
\begin{pmatrix} 
10&-18&9&0\\ 6&-11&6&0\\3&-6&4&0\\0&0&0&1
\end{pmatrix} 
\qquad
W=
\begin{pmatrix} 
1&0&0&0\\0&-1&0&-1\\0&0&1&0\\0&3&0&2
\end{pmatrix} 
\end{align*}
A computation shows that $U$ and $W$ are in $\GL_{4} (\Z)$ and 
$U \Bsf_{ E } =  \Bsf_{F} W$.
Thus, $U$ induces an isomorphism from $\coker( \Bsf_{E})$ to $\coker( \Bsf_{F})$ and $W$ induces an isomorphism from $\ker( \Bsf_{E})$ to $\ker( \Bsf_{F})$ as described in Section \ref{UVinduce}.

It is clear that $U ( [e_{i}]_{E} ) \in S_{F}$ for all $i \neq 2$.
Note that in $\coker( \Bsf_{F})$, we have that
\begin{align*}
U ( [ e_{2} ]_{E} ) = 
\begin{pmatrix} 
-18\\-11\\-6\\0
\end{pmatrix}  =
\begin{pmatrix} 
15 \\
7\\
3 \\
0
\end{pmatrix}  
+  \Bsf_{F}\begin{pmatrix} 
0 \\
0 \\
0 \\
-3 
\end{pmatrix}  \in S_{F}
\end{align*}

In the other direction, since
\[
U^{-1} = \begin{pmatrix}   
-8&18&-9&0\\-6&13&-6&0\\-3&6&-2&0\\0&0&0&1\end{pmatrix} 
\]
we may argue similarly.
\end{remark}

\subsection{Atlas of graph $C^{*}$-algebras of small graphs}
Inspired by a similar undertaking for Leavitt path algebras (\cite{MR3201827}), we end by a complete analysis of the stable isomorphism problem for small graphs, focusing on simple graphs with no more than 4 vertices. Although our invariant may be efficiently computed by methods outlined in \cite{serj:ciugc} we do not know an efficient general procedure for deciding whether or not an isomorphism exists between a pair of invariants, and further we will attempt to study also the few cases where our Condition (H) is not met, so instead of appealing exclusively to our invariant we will proceed by defining two equivalences on the set of graphs under investigation, approximating stable isomorphism of the associated graph algebras on both sides. The number of cases in need of further study is then so small that we may resolve it case by case.

\begin{definition}
The \emph{$K$-temperature} of a finite graph $E$ is the map \fctw{\mathfrak{t}_E^K}{\Gamma_E}{\{0,-1\}\cup\Ab} given by  
$$\mathfrak{t}_E^K(\gamma)=\begin{cases}\tau_E(\upsilon_E(\gamma)), & \tau_E(\upsilon_E(\gamma))<1,\\ K_0(C^*(E)(\{\upsilon_E(\gamma)\})),&\tau_E(\upsilon_E(\gamma))=1.\end{cases}$$
\end{definition}

Note that when $\Bsf_E\in\MPZccc$, then $K_0(C^*(E)(\{\upsilon_E(\gamma)\}))\cong \cok((\Bsf_E^\bullet\{\mathcal{Y}_{\Bsf_E}^{-1}(\gamma)\})^{\mathsf T})$.

\begin{definition}
We say that two graphs $E$ and $F$ with $(\Bsf_E,\Bsf_F)$ in standard form are \emph{outer equivalent}, and write $E\oo F$, if
\begin{enumerate}[(i)]
\item $\coker((\Bsf_E^\bullet)^{\mathsf T})\cong\coker((\Bsf_F^\bullet)^{\mathsf T})$, and 
\item for some order isomorphism \fct{h}{\Gamma_E}{\Gamma_F}, $\mathfrak{t}_F^K(h(\gamma))$ and $\mathfrak{t}_E^K(\gamma)$ are either  isomorphic Abelian groups or equal numbers for all $\gamma\in\Gamma_E$
\end{enumerate}
\end{definition}

We will say that a row or column addition in a matrix $\Bsf_E$
representing a simple graph (\ie, all diagonal entries are in $\{-1, 0\}$ and all other entries are in $\{0, 1\}$) is \emph{legal} if
it meets the requirements of Proposition \ref{prop:matrix-moves} and produces another such matrix.
Similarly, we say that a Move \CO is a legal collapse if it is applied to a regular vertex not supporting a loop, and if it takes a simple graph to another simple graph.

\begin{definition} 
Fix an integer $M$ and let $E$ and $F$ be simple graphs both with finite numbers of vertices $m,n\leq M$ respectively. We say that $E$ and $F$ are \emph{elementary equivalent through simple graphs of size $M$} if either $m=n$ and one of 
\begin{enumerate}[(i)]
\item $E$ is isomorphic to $F$,\label{iifirst}
\item $F$ arises from $E$ by performing a legal row addition in $\Bsf^\bullet_E$,
\item $F$ arises from $E$ by performing a legal column addition in $\Bsf^\bullet_E$,
\end{enumerate}
holds, or if $m=n+1$ and
\begin{enumerate}[(i)]\addtocounter{enumi}{3}
\item $F$ arises from $E$ by deleting a regular source,
\item $F$ arises from $E$ by a legal collapse.\label{iilast}
\end{enumerate}
The coarsest equivalence relation containing elementary equivalence through simple graphs of size $M$ is called \emph{$M$-inner equivalence}, and we write $E\ii{M} F$ when $E$ and $F$ are $M$-inner equivalent.
\end{definition}

The following is now clear.

\begin{proposition}\label{iioorel}
When $E$ and $F$ are finite simple graphs both with $M$ vertices or less, we have
\[
\xymatrix@C=4mm{
E\ii{M} F\ar@{=>}[r]&E\Meq F\ar@{=>}[d]\ar@{=>}[r]&L_{{\mathsf k}} ( E )\sim_{\text{Morita}}L_{{\mathsf k}} (F)
\ar@{=>}[r]& E\oo F\\
& E\MCeq F\ar@{=>}[rr]&&C^*(E)\otimes \K\cong C^*(F)\otimes \K\ar@{=>}[u]}
\]
\end{proposition}

Although counting the number of nonisomorphic graphs of a certain size $M$  is easy by Burnside's lemma (\cf\ \cite{oeis} A595), producing lists of them is rather computationally demanding. The most efficient way to obtain such lists is provided by McKay and Piperno (\cite{MR3131381}).
Developing algorithms to decide $M$-inner equivalence is then straightforward by testing for elementary equivalence and partitioning the set (using, \eg, Warshall's algorithm) by the smallest equivalence relation containing the relations found. Drawing on methods developed in \cite{serj:ciugc} it is not much harder to design an algorithm to decide outer equivalence.  At $M=4$, 
it then only takes a few minutes of computing time to partition these sets of graphs into $\ii{4}$- and $\oo$-classes, obtaining the numbers listed in Table~\ref{numberofiioo}. At $M=5$ we have not attempted a complete analysis, as it takes hours even to compute all the $K$-temperatures and divide the graphs into $\oo$-classes.

\begin{table}
\begin{center}
\begin{tabular}{|c|||c|c|c|c|c|}\hline
$M$&1&2&3&4&5\\\hline\hline
nonisomorphic graphs&2&10&104&3044&291968\\
$\ii{M}$-classes&2&8&35&218&?\\
$\oo$-classes&2&8&35&199&1310\\\hline
\end{tabular}
\end{center}
\caption{Number of classes for $M\in\{1,2,3,4,5\}$}\label{numberofiioo}
\end{table}

It follows directly from Proposition \ref{iioorel} that $\oo$-classes are unions of $\ii{M}$-classes, and that when they coincide, they also coincide with $\Meq$-classes, $\MCeq$-classes or stable isomorphism classes of the \cas. Consequently, the \emph{ad hoc} invariant defining outer equivalence is complete whenever the graph has 1, 2  or 3 vertices. Note that this confirms the Abrams-Tomforde conjecture in these special cases.

In the case with $M=4$ vertices, the notions differ by 17 $\oo$-classes being divided into a total of 36 $\ii{4}$-classes, which we now address. We organize these classes into four groups as indicated in Figures \ref{groupI}--\ref{groupIV}, drawing one representative for each $\ii{4}$-class and indicating the boundaries of each $\oo$-class by triple vertical lines. In the cases, explained below,  where the graphs fail to be $\MCeq$-equivalent we draw a  vertical line between them.
\begin{figure}
\begin{center}
\begin{tabular}{c c c c c c c c}\hline
\multicolumn{1}{|||c|}
{
\begin{tikzpicture}
\fournodes
\flo 1
\fed 1 2
\fed 1 3
\flo 2
\flo 3
\fed 4 1
\end{tikzpicture}
}
&
\multicolumn{1}{c|||}
{
\begin{tikzpicture}
\fournodes
\flo 1
\fed 2 1
\flo 2
\fed 2 3
\fed 2 4
\flo 3
\fed 4 1
\end{tikzpicture}
}
&
\multicolumn{1}{c|}
{
\begin{tikzpicture}
\fournodes
\flo 1
\fed 1 3
\flo 2
\fed 2 3
\flo 3
\fed 4 1
\end{tikzpicture}
}
&
\multicolumn{1}{c|||}
{
\begin{tikzpicture}
\fournodes
\flo 1
\fed 2 1
\flo 2
\fed 2 4
\fed 3 1
\flo 3
\fed 4 1
\end{tikzpicture}
}
&
\multicolumn{1}{c|}
{
\begin{tikzpicture}
\fournodes
\flo 1
\fed 1 3
\fed 2 1
\flo 2
\fed 2 4
\flo 3
\fed 4 1
\end{tikzpicture}
}
\\\hline
\multicolumn{1}{|c|||}
{
\begin{tikzpicture}
\fournodes
\flo 1
\fed 2 1
\flo 2
\fed 2 4
\fed 3 2
\flo 3
\fed 4 1
\end{tikzpicture}
}
&
\multicolumn{1}{c|}
{
\begin{tikzpicture}
\fournodes
\flo 1
\fed 1 4
\flo 2
\fed 2 4
\fed 3 1
\end{tikzpicture}
}
&
\multicolumn{1}{c|||}
{
\begin{tikzpicture}
\fournodes
\fed 1 4
\fed 2 1
\flo 2
\fed 2 4
\fed 3 1
\flo 3
\end{tikzpicture}
}
&

\multicolumn{1}{c|}
{
\begin{tikzpicture}
\fournodes
\flo 1
\fed 1 3
\fed 1 4
\fed 2 1
\end{tikzpicture}
}
&
\multicolumn{1}{c|||}
{
\begin{tikzpicture}
\fournodes
\flo 1
\fed 1 2
\fed 1 3
\fed 1 4
\fed 2 3
\end{tikzpicture}
}
\\\hline
\multicolumn{1}{|||c|}
{
\begin{tikzpicture}
\fournodes
\flo 1
\fed 1 4
\fed 2 1
\flo 2
\fed 2 3
\fed 3 1
\end{tikzpicture}
}
&
\multicolumn{1}{c|||}
{
\begin{tikzpicture}
\fournodes
\flo 1
\fed 1 2
\fed 1 4
\fed 2 4
\fed 3 1
\flo 3
\end{tikzpicture}
}
&
\multicolumn{1}{c|}
{
\begin{tikzpicture}
\fournodes
\flo 1
\fed 1 2
\fed 1 4
\flo 2
\fed 3 1
\end{tikzpicture}
}
&
\multicolumn{1}{c|}
{
\begin{tikzpicture}
\fournodes
\flo 1
\fed 2 1
\flo 2
\fed 2 3
\fed 2 4
\fed 3 1
\end{tikzpicture}
}
&
\multicolumn{1}{c|||}
{
\begin{tikzpicture}
\fournodes
\flo 1
\fed 1 2
\fed 1 3
\fed 1 4
\fed 2 4
\flo 3
\end{tikzpicture}
}
\\\hline
\end{tabular}
\end{center}
\caption{Group I}\label{groupI}
\end{figure}

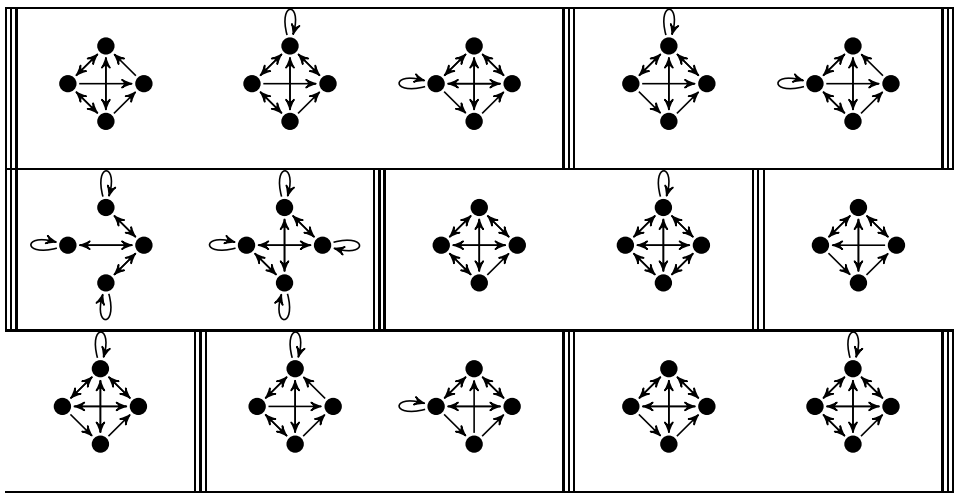
\begin{figure}
\begin{center}
\begin{tabular}{c c c c c c c c}\hline
\multicolumn{1}{|||c}
{
\begin{tikzpicture}
\fournodes
\fed 1 2
\fed 1 3
\fed 2 1
\fed 2 3
\fed 2 4
\fed 3 1
\fed 3 2
\fed 3 4
\fed 4 1
\end{tikzpicture}
}
&
\multicolumn{1}{c}
{
\begin{tikzpicture}
\fournodes
\flo 1
\fed 1 2
\fed 1 3
\fed 1 4
\fed 2 1
\fed 2 3
\fed 2 4
\fed 3 1
\fed 3 2
\fed 3 4
\fed 4 1
\end{tikzpicture}
}
&
\multicolumn{1}{c|||}
{
\begin{tikzpicture}
\fournodes
\fed 1 2
\fed 1 3
\fed 1 4
\fed 2 1
\flo 2
\fed 2 3
\fed 2 4
\fed 3 1
\fed 3 4
\fed 4 1
\fed 4 2
\end{tikzpicture}
}
&
\multicolumn{1}{c}
{
\begin{tikzpicture}
\fournodes
\flo 1
\fed 1 2
\fed 1 3
\fed 1 4
\fed 2 1
\fed 2 3
\fed 2 4
\fed 3 1
\fed 3 4
\fed 4 1
\end{tikzpicture}
}
&
\multicolumn{1}{c|||}
{
\begin{tikzpicture}
\fournodes
\fed 1 2
\fed 1 3
\fed 2 1
\flo 2
\fed 2 3
\fed 2 4
\fed 3 1
\fed 3 2
\fed 3 4
\fed 4 1
\end{tikzpicture}
}
\\\hline
\multicolumn{1}{|||c}
{
\begin{tikzpicture}
\fournodes
\flo 1
\fed 1 4
\flo 2
\fed 2 4
\flo 3
\fed 3 4
\fed 4 1
\fed 4 2
\fed 4 3
\end{tikzpicture}
}
&
\multicolumn{1}{c|||}
{
\begin{tikzpicture}
\fournodes
\flo 1
\fed 1 3
\fed 1 4
\flo 2
\fed 2 3
\fed 2 4
\fed 3 1
\fed 3 2
\flo 3
\fed 4 1
\fed 4 2
\flo 4
\end{tikzpicture}
}
&
\multicolumn{1}{c}
{
\begin{tikzpicture}
\fournodes
\fed 1 2
\fed 1 3
\fed 1 4
\fed 2 1
\fed 2 3
\fed 2 4
\fed 3 1
\fed 3 2
\fed 3 4
\fed 4 1
\fed 4 2
\end{tikzpicture}
}
&
\multicolumn{1}{c|||}
{
\begin{tikzpicture}
\fournodes
\flo 1
\fed 1 2
\fed 1 3
\fed 1 4
\fed 2 1
\fed 2 3
\fed 2 4
\fed 3 1
\fed 3 2
\fed 3 4
\fed 4 1
\fed 4 2
\fed 4 3
\end{tikzpicture}
}
&
\multicolumn{1}{c}
{
\begin{tikzpicture}
\fournodes
\fed 1 2
\fed 1 3
\fed 1 4
\fed 2 1
\fed 2 3
\fed 3 1
\fed 3 4
\fed 4 1
\fed 4 2
\end{tikzpicture}
}
\\\hline
\multicolumn{1}{c|}
{
\begin{tikzpicture}
\fournodes
\flo 1
\fed 1 2
\fed 1 3
\fed 1 4
\fed 2 1
\fed 2 3
\fed 2 4
\fed 3 1
\fed 3 4
\fed 4 1
\fed 4 2
\end{tikzpicture}
}
&
\multicolumn{1}{|||c}
{
\begin{tikzpicture}
\fournodes
\flo 1
\fed 1 2
\fed 1 3
\fed 2 1
\fed 2 3
\fed 2 4
\fed 3 1
\fed 3 2
\fed 3 4
\fed 4 1
\end{tikzpicture}
}
&
\multicolumn{1}{c|||}
{
\begin{tikzpicture}
\fournodes
\fed 1 2
\fed 1 4
\fed 2 1
\flo 2
\fed 2 3
\fed 2 4
\fed 3 1
\fed 3 4
\fed 4 1
\fed 4 2
\end{tikzpicture}
}
&\multicolumn{1}{c}
{
\begin{tikzpicture}
\fournodes
\fed 1 2
\fed 1 3
\fed 1 4
\fed 2 1
\fed 2 3
\fed 2 4
\fed 3 1
\fed 3 4
\fed 4 1
\fed 4 2
\end{tikzpicture}
}
&
\multicolumn{1}{c|||}
{
\begin{tikzpicture}
\fournodes
\flo 1
\fed 1 2
\fed 1 3
\fed 1 4
\fed 2 1
\fed 2 3
\fed 2 4
\fed 3 1
\fed 3 2
\fed 3 4
\fed 4 1
\fed 4 2
\end{tikzpicture}
}\\\hline
\end{tabular}
\end{center}
\caption{Group II}\label{groupII}
\end{figure}
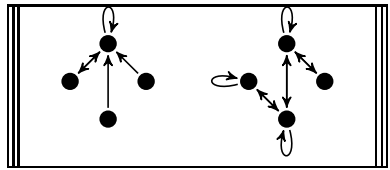
\begin{figure}
\begin{center}
\begin{tabular}{c c}\\\hline
\multicolumn{1}{|||c}
{
\begin{tikzpicture}
\fournodes
\flo 1
\fed 1 2
\fed 2 1
\fed 3 1
\fed 4 1
\end{tikzpicture}
}
&
\multicolumn{1}{c|||}
{
\begin{tikzpicture}
\fournodes
\flo 1
\fed 1 3
\fed 1 4
\flo 2
\flo 3
\fed 2 3
\fed 3 1
\fed 3 2
\fed 4 1
\end{tikzpicture}
}\\\hline
\end{tabular}
\end{center}
\caption{Group III}\label{groupIII}
\end{figure}
\begin{figure}
\begin{center}
\begin{tabular}{c c c c}\\\hline
\multicolumn{1}{|||c}
{
\begin{tikzpicture}
\fournodes
\flo 1
\fed 1 3
\fed 1 4
\fed 2 1
\flo 2
\fed 3 1
\end{tikzpicture}
}
&
\multicolumn{1}{|c|||}
{
\begin{tikzpicture}
\fournodes
\fed 1 3
\fed 1 4
\fed 2 1
\flo 2
\fed 2 4
\fed 3 1
\flo 3
\end{tikzpicture}
}
&
\multicolumn{1}{c|}
{
\begin{tikzpicture}
\fournodes
\flo 1
\fed 1 3
\fed 1 4
\fed 2 1
\flo 2
\flo 3
\fed 4 1
\end{tikzpicture}
}
&
\multicolumn{1}{c|||}
{
\begin{tikzpicture}
\fournodes
\flo 1
\flo 2
\fed 2 4
\fed 3 1
\flo 3
\fed 3 4
\fed 4 1
\fed 4 2
\end{tikzpicture}
}\\\hline
\end{tabular}
\end{center}
\caption{Group IV}\label{groupIV}
\end{figure}
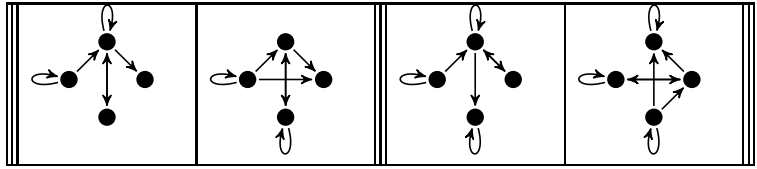

\begin{observation}
None of the graphs in the outer equivalence classes listed in Group I are $\MCeq$-equivalent, and none of them give stably isomorphic \cas.
\end{observation}
\begin{proof}
Since Theorem~\ref{coldiso} applies, this follows directly by checking that no solution 
to the small linear systems in \eqref{linpart} exists.
\end{proof}

In this case, the $\ii{4}$-classes coincide with the $\Meq$-classes as well as with the classes giving stably isomorphic graph \cas, and the invariant used to define outer equivalence fails to be complete. This is simply because the information needed to distinguish the matrices up to \SLP-equivalence may not be reconstructed from the partial data contained in the $K_0$-group of the whole system and of the irreducible components.

\begin{observation}
All graphs in the outer equivalence classes listed in
Group II  are mutually $\Meq$-equivalent.
\end{observation}
\begin{proof}
In every case the given graph defines an irreducible SFT, and hence by \cite{MR758893} (see also \cite{MR3082546}), since we know that the Bowen-Franks groups are the same in each outer equivalence class, we just need to check --- which is easily done --- that the signs of the determinants match up. 
\end{proof}

This observation contains the result that indeed the $\oo$-classes coincide with the $\MCeq$-classes and $\Meq$-classes as well as the classes with stably isomorphic graph \cas.
The explanation of the lack of success of our approach to establish elementary equivalence through simple graphs is that since the graphs have so many edges, there is not room for enough row or column additions to pass from one to another. Indeed, all the graphs in each outer equivalence class turn out to be $\ii{5}$-equivalent.

\begin{observation}
The graphs in  Group III  are  $\MCeq$-equivalent without being $\Meq$-equivalent. The graphs in the outer equivalence classes listed in
Group IV  fail to be $\MCeq$-equivalent, yet produce stably isomorphic \cas.
\end{observation}
\begin{proof}
For the first claim, we see that the two graphs given are clearly move equivalent to the graph given by the matrix $(2)$ and its Cuntz splice.
For the second, we note that we get the four graphs  considered in Examples \ref{notalwayssamepre}, \ref{notalwayssame}, \ref{notalwayssameii} after applying Move \CO\ to the unique regular vertex not supporting a loop. 
 \end{proof}

Combining these results, we get

\begin{observation}
The 3044 different simple graphs with four vertices are divided into 210 different $\Meq$-classes and 209 different $\MCeq$-classes. They define a total of 207 different graph \cas, identified up to stable isomorphism.
\end{observation} 

The number of different Leavitt path algebras (say with $\mathsf k=\C$) defined, identified up to Morita equivalence, is not known, but must be in the range $\{207,208,209,210\}$ since for all the graphs giving isomorphic stabilized $C^*$-algebras except the ones in Group III and IV we have established $\Meq$, which implies Morita equivalence of the Levitt path algebras as well.

\section*{Acknowledgements}

This work was partially supported by the Danish National Research Foundation through the Centre for Symmetry and Deformation (DNRF92), by VILLUM FONDEN through the network for Experimental Mathematics in Number Theory, Operator Algebras, and Topology, by a grant from the Simons Foundation (\# 279369 to Efren Ruiz), and by the Danish Council for Independent Research | Natural Sciences.

The third and fourth named authors would like to thank the School
of Mathematics and Applied Statistics at the University of Wollongong
for hospitality during their visit there, and the first and second named authors likewise thank the Department of
Mathematics, University of Hawaii, Hilo. The initial work was carried out at these two long-term visits, and
it was completed while
all four authors were attending the research program
\emph{Classification of operator algebras: complexity, rigidity, and                                                              
 dynamics} at the Mittag-Leffler Institute, January--April 2016. We
thank the institute and its staff for the excellent work conditions
provided.

The authors would also like to thank Mike Boyle and James Gabe for many fruitful discussions.

\newcommand{\etalchar}[1]{$^{#1}$}
\providecommand{\bysame}{\leavevmode\hbox to3em{\hrulefill}\thinspace}
\providecommand{\MR}{\relax\ifhmode\unskip\space\fi MR }
\providecommand{\MRhref}[2]{%
  \href{http://www.ams.org/mathscinet-getitem?mr=#1}{#2}
}
\providecommand{\href}[2]{#2}

\end{document}